\documentclass[10pt]{amsart}

\usepackage{a4wide}

\usepackage{amssymb,amsmath,bbm,bm,amsthm}
\usepackage{mathtools}
\usepackage{color}
\usepackage{enumitem}
\usepackage[all, cmtip]{xy}
\usepackage{xcolor}
\usepackage{footnote}
\usepackage{etoolbox,mathrsfs}
\usepackage[T1]{fontenc}
\usepackage[utf8]{inputenc}

\usepackage[colorlinks, pagebackref]{hyperref}
\hypersetup{
  colorlinks=true,
  citecolor=blue,
  linkcolor=blue,
  urlcolor=blue}
\usepackage{cleveref}

\usepackage[colorinlistoftodos, textsize=tiny]{todonotes}
\makeatletter
\providecommand\@dotsep{5}
\def\listtodoname{List of Todos}
\def\listoftodos{\@starttoc{tdo}\listtodoname}
\makeatother

\makeatletter
\def\@tocline#1#2#3#4#5#6#7{\relax
  \ifnum #1>\c@tocdepth 
  \else
    \par \addpenalty\@secpenalty\addvspace{#2}%
    \begingroup \hyphenpenalty\@M
    \@ifempty{#4}{%
      \@tempdima\csname r@tocindent\number#1\endcsname\relax
    }{%
      \@tempdima#4\relax
    }%
    \parindent\z@ \leftskip#3\relax \advance\leftskip\@tempdima\relax
    \rightskip\@pnumwidth plus4em \parfillskip-\@pnumwidth
    #5\leavevmode\hskip-\@tempdima
      \ifcase #1
      \or\or \hskip 2em \or \hskip 2homologyem \else \hskip 3em \fi%
      #6\nobreak\relax
    \dotfill\hbox to\@pnumwidth{\@tocpagenum{#7}}\par
    \nobreak
    \endgroup
  \fi}
\makeatother

\theoremstyle{plain}

\newtheorem{theorem}{Theorem}[section]

\newtheorem{lemma}[theorem]{Lemma}
\newtheorem{corollary}[theorem]{Corollary}
\newtheorem{proposition}[theorem]{Proposition}

\theoremstyle{definition}

\newtheorem{notation}[theorem]{Notation}
\newtheorem{remark}[theorem]{Remark}

\newtheorem{definition}[theorem]{Definition}
\newtheorem{example}[theorem]{Example}

\numberwithin{equation}{subsection}

\newlist{myenum}{enumerate}{1}
\setlist[myenum,1]{wide, labelwidth=!, labelindent=0pt, label=(\arabic*)}

\newcommand{\ilim}{\mathop{\varprojlim}\limits} 
\newcommand{\dlim}{\mathop{\varinjlim}\limits}  

\DeclareMathOperator{\Sm}{Sm}
\DeclareMathOperator{\eSm}{\overline{Sm}}
\DeclareMathOperator{\Fun}{Fun} 
\DeclareMathOperator{\Tri}{Tri} 

\newcommand{\Sch}{\mathscr Sch} 

\DeclareMathOperator{\colim}{colim} 
\DeclareMathOperator{\nrv}{N} 

\DeclareMathOperator{\codim}{codim}

\DeclareMathOperator{\Ker}{Ker}
\DeclareMathOperator{\coker}{Coker}
\DeclareMathOperator{\op}{op}

\DeclareMathOperator{\uHom}{\underline{Hom}}

\DeclareMathOperator{\Map}{Map}

\DeclareMathOperator{\im}{Im}
\DeclareMathOperator{\Spec}{Spec}

\newcommand{\GL}{{\rm GL}\,}        
\newcommand{\Ho}{{\rm Ho}}
\newcommand{\ie}{{\it i.e.\/},\ }

\renewcommand{\tilde}{\widetilde}

\DeclareMathOperator{\Et}{Et}
\newcommand{\et}{\mathrm{et}}
\newcommand{\nis}{\mathrm{Nis}}
\newcommand{\zar}{\mathrm{Zar}}

\newcommand{\N}{{\mathbb N}}

\newcommand{\Z}{{\mathbb Z}}

\newcommand{\piA}{{\bm \pi}^{\A^1}}   		

\newcommand{\A}{\mathbb A}

\renewcommand{\AA}{\mathbb A}
\newcommand{\PP}{\mathbb P}

\newcommand{\cG}{\mathcal G}

\def\<{\langle}
\def\>{\rangle} 
\def\-{\overline} 
\def\~{\widetilde}
\def\^{\widehat}

\def\@{\mathcal}
\def\!{\mathscr}
\def\#{\mathbb}
\def\&{\mathbf}
\def\_{\underline}
\def\Dot{\bullet} 
\def\x{\times}

\def\.{\cdot}

\newcommand{\E}{\mathscr E} 
\newcommand{\C}{\mathscr C} 
\newcommand{\T}{\mathscr T} 
\newcommand{\disc}{\mathrm{disc}} 

\DeclareMathOperator{\id}{Id} 

\newcommand{\cX}{\mathcal X} 
\newcommand{\cY}{\mathcal Y} 
\newcommand{\cO}{\mathcal O} 

\newcommand{\sX}{\mathfrak X} 

\DeclareMathOperator{\Sh}{Sh} 
\newcommand{\iSh}{\operatorname{Sh}^\infty} 
\newcommand{\ihSh}{\widehat{\operatorname{Sh}}^\infty} 
\newcommand{\iPSh}{\operatorname{PSh}^\infty} 
\newcommand{\iDer}{\operatorname{D}^\infty} 
\DeclareMathOperator{\HA}{\mathscr H_{\A^1}} 

\newcommand{\iSet}{\mathscr S} 

\DeclareMathOperator{\flag}{\mathscr F}
\DeclareMathOperator{\uPi}{\underline\Pi}
\newcommand{\uE}{\underline E}

\DeclareMathOperator{\pro}{pro}

\DeclareMathOperator{\uG}{\underline \Gamma} 
\DeclareMathOperator{\uH}{\underline H} 

\newcommand{\pplim}[1] {\underset{#1}{\text{``$\varprojlim$''}}} 

\newcommand{\Set}{\mathscr Set}

\DeclareMathOperator{\Cz}{\underline{Cz}} 
\DeclareMathOperator{\Ge}{\underline{Ge}} 

\newcommand{\Coef}{\mathbb G} 

\begin{document}

\title[Spectral sequences in unstable higher homotopy theory]{Spectral sequences in unstable higher homotopy theory
 and applications to the coniveau filtration}
\author{Fr\'ed\'eric D\'eglise}
\address{CNRS, UMPA (UMR 5669), ENS de Lyon, UMPA, 46, alle\'{e} d'Italie, 69364, Lyon Cedex 07, France.}
\email{frederic.deglise@ens-lyon.fr}
\author{Rakesh Pawar}
\address{Harish-Chandra Research Institute, A CI of Homi Bhabha National Institute,
Chhatnag Road, Jhunsi, Prayagraj 211019, India.}
\email{rakeshpawar@hri.res.in, pawarrakesh.math@gmail.com}
\date{\today}


\begin{abstract}
With the aim of understanding Morel's result on the $\A^1$-homotopy sheaves
 over a field, we extend the theory of unstable spectral sequences of Bousfield and Kan
 in the $\infty$-categorical setting. With this natural extension, parallel to the classical formalism of {\it cohomology theory with supports}, we introduce the notion of {\it cohomotopy theory with supports}.  
 We extend the Bloch-Ogus-Gabber theorem for Cohomology theory with supports to that of unstable setting,
 in order to obtain unstable Gersten (or Cousin) resolutions associated with the coniveau filtration,
 under suitable assumptions. 
 We apply this theory to motivic homotopy, Nisnevich-local torsors and Artin-Mazur \'etale homotopy types.
\end{abstract}

\maketitle
\tableofcontents

\section{Introduction}

The notion of spectral sequence for homotopy groups is classical in algebraic topology ---
 the first occurrence seems to be in \cite{Fed56}. It has been popularized
 by Bousfield and Kan in \cite[IX, \textsection 4]{BK72} as a tool to study towers of fibrations
 and obstruction theory for CW-complexes (equivalently, simplicial sets).

Next, in the stable case, it was put into a very general setup by Jacob Lurie,
 via stable $\infty$-categories and $t$-structures (\cite{LurieHA}). The primary aim of these notes
 is to extend these constructions to the \emph{unstable setting}, or rather
 to the setting of not-necessarily stable $\infty$-categories.
 This is based on a new notion of \emph{(co)homotopy functors} which
 play the role of (co)homology functors for triangulated categories
 (or in fact stable $\infty$-categories). See \Cref{sec:htp_functors}.
 Note in particular that
 we have chosen to work internally within a given $1$-topos to get a flexible enough
 category which receives "homotopy pointed objects and groups".
 We also take this opportunity to show that in any pointed $\infty$-category
 with finite limits/colimits, fiber/cofiber sequences
 satisfy the exact analog of the properties of a triangulated category,
 including the octahedron axiom. See \Cref{sec:fiber_seq}.
 This property is crucial in our study of unstable spectral sequences.

Once this theory has been developed, one can define the notion of unstable exact couple
 following \cite{Fed56, BK72}, as well as the associated unstable spectral sequence.
 Here, we follow the path of Bousfield and Kan,
 though we work internally within some fixed topos.
 We observe that there are two possible notions of unstable exact couples,
 and that the associated unstable spectral sequences are equivalent (up to a sign)
 using the unstable formulation of the octahedron axiom.
 We also introduce a useful notion of \emph{augmentation} which is critical to relate terms
 between successive pages.
 It plays
 a crucial role in deriving "resolutions" from unstable spectral sequences,
 as it provides canonical augmentations of the given resolutions
 (see \Cref{unstable-cx} and \Cref{df:coaugmented_cohtp_cpx}).

\bigskip

One of our main motivations for developing the above machinery is to obtain
 coniveau spectral sequences in an unstable setting, drawing strong inspiration
 from Morel's pioneering work in \cite{MorelLNM}.
 We get two types of results here.

First, we set up a general framework of coniveau spectral sequences in the
 homotopical/unstable setting, extending the setting from \cite{CHK}.
 This leads us to extend Grothendieck's definition of Cohen-Macaulay sheaves and
 the associated Cousin complexes to the non-abelian setting.
 We introduce an appropriate notion of being \emph{homotopy Cohen-Macaulay} for sheaves
 of groups on the Zariski site of a scheme $X$ (and also a weaker Nisnevich variant) in \Cref{df:CM-non-ab},
 and we draw out a kind of Cousin resolution to approximate the associated set of torsors 
 (see \Cref{prop:htpy-CM}). If the sheaf of groups is representable by a separated algebraic group scheme, we prove the following result
 (\Cref{thm:alg-gp-CM} in the main text).
\begin{theorem}\label{thm:alg-gp-CM} 
Let $X$ be a noetherian scheme which is regular in codimension less than $3$.
 Let $G$ be a separated algebraic group scheme over $X$, which we assume to be affine if $\dim(X)>1$.

Then $G$ is a homotopy Cohen-Macaulay sheaf on the small site $X_t$ for $t=\zar, \nis$.
\end{theorem}

As a consequence:
\begin{corollary}[\Cref{cor:alg-gp-CM} and \Cref{cor:dedekind}]
Let $X$ be a connected Dedekind scheme with function field $K$,
 and $G$ be a separated $X$-group scheme. Let $t$ be either the Zariski or Nisnevich topology.
 Then $G$ is homotopy Cohen-Macaulay on the small site $X_t$, and the canonical pointed map:
$$
H^1(X_t,G) \rightarrow \left({\prod}'_{x \in X^{(1)}}  H^1_x(X_t,G)\right)/G(K) \simeq G(\#O_X)\backslash G(\#A_X)/G(K)
$$
is a bijection, where $$G(\#O_X):= {\prod}'_{x \in X^{(1)}} G(X^t_{(x)})$$ and $$G(\#A_X):={\prod}'_{x \in X^{(1)}} G(X^t_{(x)}-\{x\}).$$
\end{corollary}

Here, the symbol ${\prod}'$ denotes the restricted product (see \Cref{num:restricted-product}).
 This result is well-known for a reductive groups and curves over a field;
 in particular, the case of $\GL_n$ and curves over a field is due to Weil.
 In fact, a close inspection of Morel's fundamental results for strongly $\AA^1$-invariant sheaves of groups $\cG$
 over a perfect field $k$ shows that the method of \cite[\textsection 2.2]{MorelLNM} does prove the above corollary
 for a smooth curve $X$ over $k$.\footnote{This is proved in our notes without assuming $k$ perfect,
 but one can also look at the clear account of Morel's argument from \cite{AD09}: use 4.9, 4.10 and 4.11 of \emph{loc. cit.}
 and the assumption $\dim(X)=1$.}
 We can further compare our statement with various enhancements of Weil's seminal theorem. 
 M. Groechenig~\cite[Theorem 0.1, Theorem 0.3]{Gro17} proved the statement for Noetherian scheme $X$ and $\@G$ representable
 by a Noetherian affine algebraic group scheme that is special (\ie every $G$-torsor on a Noetherian scheme is Zariski locally trivial).
 In contrast, the above statement holds for connected Dedekind scheme $X$ and any separated $X$-group scheme $G$. Additionally,
V. Chernousov, A. Rapinchuk and I. Rapinchuk considered in~\cite[Corollary B.7]{CRR20}
 the statement of the above corollary for $X=\Spec R$ where $R$ is a regular domain of Krull dimension $\leq 2$ and $G=GL_n$. For the case of a regular integral scheme of dimension 2 in \'etale topology and groups schemes which are either a smooth reductive group scheme or a finite type of multiplicative type, we make the following observations. 
 \begin{theorem}\label{dim2} For regular integral scheme $X$ of dimension 2 and $G$ either a smooth reductive $X$-group scheme or a finite type $X$-group scheme of multiplicative type, there is an exact sequence 
$$*\to {\prod}'_{x \in X^{(1)}}  H^1_x(X_{\et},G)/G(K)\xrightarrow{\gamma} H^1(X_{\et}, G)\xrightarrow{\theta} \im\bigg(H^1(X_{\et}, G)\to H^1((\Spec K)_{\et}, G)\bigg)\to * $$
of pointed sets.
\end{theorem}
\begin{corollary}For regular integral scheme $X$ of dimension 2, and  $G$ a smooth reductive $X$-group scheme such that $G$ is a special group (\ie Zariski cohomology and \'etale cohomology agree), there is a bijection
$$ {\prod}'_{x \in X^{(1)}}  H^1_x(X_{\et},G)/G(K)\xrightarrow{\gamma} H^1(X_{\et}, G).$$
In particular, this holds for $G=GL_n$ ($n\geq 1$). 
\end{corollary}

We next extend the \emph{trick of Gabber} to our unstable setting by applying it to a \emph{cohomotopy theory with supports}
 (see \Cref{num:cohtp_support} for a precise definition).
 This allows us to establish Gersten resolutions for certain homotopy sheaves of groups, pushing further the approach of Morel ---
 by allowing to drop perfectness assumptions in some cases (see \Cref{cor:unstable-BOG}). 
 
More generally, we have the following result (see \Cref{prop:Gabber&Gersten} in the main text):
 \begin{proposition}
Assume $\Pi_*$ is a cohomotopy theory with supports  defined on $\Sm_k$,
 where $k$ is an arbitrary field, that satisfies the Gabber property.

Then, for any essentially smooth semi-local (resp. local if $k$ is finite) $k$-scheme $X$,
 the unstable coniveau exact sequence $E_{1,c}^{**}(X,\Pi_*)$
 of $X$ with coefficients in $\Pi_*$ (see \Cref{df:coniv-ssp-E-coef})
 collapses on the column $*=0$ in the sense of \Cref{df:ssp-collapse}.
 \end{proposition}
 As a consequence, we deduce the following corollary.
 \begin{corollary}\label{cor:unstable-BOG}
Under the above assumptions, let $q>1$ be an integer.

Let $\uPi_q$ be the Zariski sheaf on $\Sm_k$ associated with $\Pi_q$. The following results then hold:
\begin{enumerate}
\item For any smooth $k$-scheme $X$, the sheaf $\uPi_q$ is Cohen-Macaulay up to degree $q$ on $X_\zar$
 and is fully Cohen-Macaulay if $\dim(X)\leq q$.
\item $\uPi_q$ is a Nisnevich sheaf on $\Sm_k$.
 For any $0\leq p<q$, and any smooth $k$-scheme $X$,
 there exist isomorphisms that are natural with respect to flat pullbacks:
$$
E_{2,c}^{p,q}(X,\Pi_*) \simeq H^p_\zar(X,\uPi_q) \simeq H^p_\nis(X,\uPi_q).
$$
\item For $t=\zar, \nis$, and any smooth $k$-scheme $X$,
 there exists a unique isomorphism of complexes of sheaves on $X_t$:
$$
\Ge^*(X_t,\Pi_*,q) \simeq \tau_{nv}^{\leq q} \Cz_t^*(\uPi_q^X)
$$
between the $t$-local Gersten complex in degree $q$ of $\Pi_*$ and the naively truncated
 $t$-local Cousin complex on the $t$-sheaf $\uPi_q$ restricted to the small Nisnevich site $X_t$.
 This isomorphism extends to the site of smooth $k$-schemes with morphisms being flat (equivalently, syntomic).

In particular, for any point $x \in X^{(p)}$, one obtains isomorphisms:
$$
\Pi_{q-p}^x(X_{(x)}) \simeq \Pi_{q-p}^x(X^h_{(x)}) \simeq H^{p}_x(X_\zar,\uPi_q) \simeq H^{p}_x(X_\nis,\uPi_q).
$$
\end{enumerate}
\end{corollary}

\begin{remark}
We extend this result in \Cref{cor:unstable-BOG} to get a \emph{weaker} Gersten property over positive-dimensional bases
 (see \Cref{eff2} and \Cref{cor:overbase}).
\end{remark}

Specifically, point (3) allows one to compute the $E_1$-term $E_{1,c}^{*,q}(X,\Pi_*)$,
 up to isomorphism rather than merely quasi-isomorphism, as the global sections of the naive truncation of the Cousin complex
 associated with the homotopy sheaf $\uPi_q^X$. Such a result was suggested by the terminology of
 \cite[p. 1172-1174]{AD09}.

We note that over a perfect field $k$, it follows from \cite[Th. 6.1]{MorelLNM} that
 for any pointed simplicial sheaf $\cX$ on the Nisnevich site $\Sm_k$,
 the sheaf of groups $\pi_1^{\AA^1}(\cX)$ is homotopy Cohen-Macaulay in the above sense
 over $X_t$, for any smooth $k$-scheme $X$. In this paper, we generalize this result
 to show that over any field, certain $\AA^1$-homotopy sheaves are Cohen-Macaulay in the sense of \cite{Hart66},
 up to some naive truncation as above (see \Cref{morel-CM} for details).
 This extends recent works of \cite{DFJ22} and \cite{DKO} on Cousin resolutions to the unstable setting.
 
 Finally, we apply our theory to Artin-Mazur \'etale homotopy types to obtain Gersten resolution in
 the homotopical setting without resorting to $\AA^1$-homotopy. This requires proving
 an effaceability property in \'etale homotopy (\Cref{prop:et-gabber} in the main text):
 \begin{theorem}\label{prop:et-gabber}
Let $k$ be a separably closed field and $\Coef$ be a pro-space.
 Then the cohomotopy theory with supports $\Pi_*(-;\Coef)$ restricted to the category $\Sm_k$ satisfies the Gabber property
 (\Cref{def:gabber}). In particular, it satisfies all the properties stated in the preceding corollary.
\end{theorem}

While the above result is well established when $\Coef$ is abelian, according to the Gabber's original proof,
 our theorem seems to be the first instance of a purely homotopical, non-abelian, statement.
 In particular, this result, along with our broader framework, substantially extend earlier considerations
 by Colliot-Th\'el\`ene \cite[\textsection 5.1]{Colliot95}.

In a companion work~\cite{DP25+}, we will generalize the comparison between the coniveau spectral sequence
 and the cohomological spectral sequence associated with the homotopy $t$-structure of \cite{Bondarko} and \cite{Deg13}
 to the unstable setting.
 Let us indicate finally that the $\infty$-categorical unstable spectral sequence machinery developed here
 can be applied to define an unstable motivic Adams spectral sequence,
 a potentially powerful tool for computations of unstable motivic homotopy sheaves.

\section*{Plan of this work}

The second section is the technical heart of the paper, and is of independent interest as it applies
 to $\infty$-categories with appropriate finite limits or colimits.

We first show that Verdier's axioms for triangulated categories are incarnated in
 unstable homotopy, which are now properties of pointed $\infty$-categories with appropriate
 limits (resp. colimits). We also highlight properties specific to the unstable setting
 such as the interaction between (homotopy) fiber and cofiber sequences, and an $\infty$-categorical
 definition of the \emph{monodromy} of a homotopy fiber sequence.
 We then introduce a new notion of homotopy functors playing the role of homological functors for $\infty$-categories.
 This is based on the notion of $\pi_*$-structure, and leads to the notion of homotopy complexes.
 This allows us to set up the machinery of unstable exact couples and unstable spectral sequences,
 drawing significant inspiration from \cite{BK72}. The main result of this section is the proof of a degeneracy criterion
 \Cref{thm:degeneracy}
 for unstable spectral sequences (which also applies to the stable setting and extends an original formulation
 of Quillen).

The third section applies the previous machinery to define coniveau spectral sequences in the unstable setting,
 inspired by Grothendieck's theory of Cousin (residual) complexes as exposed in \cite{Hart66}.
 We then study the induced complexes, called either Gersten or Cousin homotopical complexes depending on the context.
 We introduce relevant properties of these complexes, referred to as the Gersten and Cohen-Macaulay conditions,
 in accordance with domain-specific terminology. We study these conditions in the particular case of Eilenberg-MacLane
 complexes, leading to a homotopical analogue of the theory of Cohen-Macaulay sheaves of groups and Cousin complexes.
 We prove that separated algebraic groups $G$ give examples of such (see \Cref{thm:alg-gp-CM}) and deduce
 an adelic computation of the set of $G$-torsors (see  \Cref{cor:alg-gp-CM}).

The fourth section focuses on the trick of Gabber to deduce Gersten properties in various contexts.
 Notably, one applies this theory in \Cref{sec:AM-htp} to extend the results of \cite{CHK} to Artin-Mazur \'etale
 homotopy types. 

\section*{Notation}

One says that a pointed map is trivial if it is equal to the composite
 of the projection to the final object followed by the map to the base point.
 A final object in an $(\infty$-)category (if it exists) is generically denoted by $*$.
 We usually just say colimit/limit (resp. commutative) for homotopy colimit/limit (homotopically commutative) in an $\infty$-category.
 Sometimes we identify a $1$-category with its nerve when it is clearly intended by the context.

\section{Unstable exact couples and spectral sequences}

\subsection{Fiber sequences and monodromy action}\label{sec:fiber_seq}

\begin{notation}
Given a pointed $\infty$-category $\C$,\footnote{An $\infty$-category which admits an object which is both final and initial.}
  we will conventionally denote by $*$ an initial object of $\C$.
 Products (resp. coproducts) in $\C$ will be denoted by the symbol $\times$ (resp. $\vee$).

Recall the following classical definition (\cite[I, \textsection 3]{Quillen}, \cite[1.1.4]{LurieHA}).
\end{notation}
\begin{definition}\label{df:(co)fiber_seq}
Let $\C$ be a pointed $\infty$-category.
A \emph{triangle} in $\C$ is a commutative square of the form:
$$
\xymatrix@=8pt{
X\ar[r]\ar^f[r]\ar[d]\ar@{}|\Delta[rd] & Y\ar^g[d] \\ {}*\ar[r] & Z.
}
$$
Let $\Box$ be the nerve of the finite category 
 associated with the poset $\{(0,0), (1,0), (0,1), (1,1)\}$ with the lexicographic order. 
 The $\infty$-category of triangles is the sub-$\infty$-category $\Tri(\C)$ of the $\infty$-category of functors
 $\Fun(\Box,\C)$ made of objects $\Delta$ such that $\Delta(1,0)$ is 
 the zero object $*$.
 We will usually abusively denote triangles as a sequence $X \rightarrow Y \rightarrow Z$,
 the homotopy $gf \simeq *$ being implied.

Such a triangle is called a \emph{fiber} (\text{resp.} \emph{cofiber}) \emph{sequence} in $\C$ 
 if $\Delta$ is a (homotopy) pullback (\text{resp.} \emph{pushout}) square.
 In this case, $X$ (resp. $Z$) is uniquely determined by the pullback (resp. pushout) diagram
 $\Delta$, and one says that $X$ (or $(X,f)$) (resp. $Z$ (or $(Z,g)$) is the \emph{(homotopy) fiber} 
 (resp. \emph{cofiber}) of $g$ (resp. $f$).
\end{definition}
In particular, fiber and cofiber sequences are exactly dual: a cofiber sequence
 in $\C$ is a fiber sequence in $\C^{op}$. Therefore, we will restrict to fiber sequences below.

\begin{notation}\label{num:loop&suspension}
Let again $\C$ be a pointed $\infty$-category which admits finite limits (resp. colimits).
 Then one defines the loop space (resp. suspension) of an object $X$ of $\C$ by
 the following pullback (resp. pushout) square:
$$
\xymatrix@=10pt{
\Omega X\ar[r]\ar[d] & {}*\ar[d]
 & X\ar[r]\ar[d] & {}*\ar[d] \\
{}*\ar[r] & X & {}*\ar[r] & \Sigma X.
}
$$
This defines an $\infty$-functor $\Omega:\C \rightarrow \C$ (resp. $\Sigma:\C \rightarrow \C$).
 By symmetry of the diagram, the functor that exchanges the lower left and upper right corner
 induces a natural auto-equivalence $\Omega \rightarrow \Omega$ ($\Sigma \rightarrow \Sigma$),
 and we will denote by $-\Omega$ (resp. $-\Sigma$) the $\infty$-functor obtained by composing the original one
 with the latter auto-equivalence.
\end{notation}

\begin{remark}
Recall finally that, when $\C$ admits finite limits and colimits,
 $\Omega$ is left adjoint to $\Sigma$.
\end{remark}

\begin{notation}
Let us fix a pointed $\infty$-category $\C$ with finite limits.
 Then fiber sequences satisfy properties that precisely correspond to the axioms of a triangulated category,
 except one is not allowed to suspend. This has been observed in \cite[I, \S 3]{Quillen} for model categories,
 but not formalized to our knowledge. So we state the exact properties as they will be critical for
 our results.
\end{notation}
\begin{proposition}\label{prop:unstable_triangulated}
Let $\C$ be a pointed $\infty$-category admitting all finite limits.
 The following properties of fiber sequences hold:
\begin{enumerate}[label=(Fib\arabic*)]
\item \begin{enumerate}
\item Given any morphism $f:Y \rightarrow X$, there exists a
 fiber sequence $F \xrightarrow i X \xrightarrow f Y$ in $\C$, unique up to a contractible space of choices. 
\item If a triangle $\Delta:F \rightarrow Y \rightarrow X$ is isomorphic to a fiber sequence (in the $\infty$-category $\Tri(\C)$),
 then $\Delta$ is a fiber sequence.
\item The homotopy fiber of the identity $\mathrm{Id}_X:X \rightarrow X$ is the zero object $*$ (with canonical map $* \rightarrow X$).
 \end{enumerate}
\end{enumerate}

\begin{enumerate}[label=(Fib\arabic*)]\setcounter{enumi}{1}
\item A fiber sequence $F \xrightarrow i Y \xrightarrow f X$ can be extended on the \emph{left}
$$
\hdots \xrightarrow{\Omega^2 f} \Omega^2 X \xrightarrow{-\Omega \partial_f} \Omega F \xrightarrow{-\Omega g} \Omega Y \xrightarrow{-\Omega f} \Omega X \xrightarrow{\ \partial_f\ }
F \xrightarrow{\ i\ } Y \xrightarrow{\ f\ } X
$$
in such a way that any couple of consecutive maps is a fiber sequence. This extension is unique up to a contractible space of choices.
 The resulting sequence is called a \emph{long fiber sequence}.
\item Given a commutative diagram of solid arrows in $\C$
$$\xymatrix@=12pt{
\Omega X\ar^d[r]\ar_{\Omega a}[d] & F\ar^i[r]\ar@{-->}[d] & Y\ar^f[r]\ar^b[d] & X\ar^a[d] \\
\Omega X'\ar^{d'}[r] & F'\ar^{i'}[r] & Y'\ar^{f'}[r] & X'
}
$$
such that the horizontal lines are the end of a long exact fiber sequence, there exists a dotted arrow making the resulting squares
 commutative.
\item Consider a commutative diagram $\mathcal D$ in $\C$: 
 $\xymatrix@C=14pt@R=-4pt{X\ar^h[rr]\ar_g[rd] & & Z \\ & Y\ar_f[ru] &}$. Then there exists an essentially unique commutative diagram in $\C$
 made of pullback squares and of the following form:
$$
\xymatrix@R=15pt@C=26pt{
Z\ar@{}|{(1)}[rd] & Y\ar_-f[l]\ar@{}|{(2)}[rd] & X\ar_-g[l]\ar@{}|{(3)}[rd] & {}*\ar[l] & & \\
{}*\ar[u] & F\ar[l]\ar|i[u]\ar@{}|{(4)}[rd] & H\ar |{g'}[l]\ar|k[u]\ar@{}|{(5)}[rd] & \Omega Z\ar|{\partial_h}[l]\ar[u]\ar@{}|{(6)}[rd] & {}*\ar[l] & \\
& {}*\ar[u] & G\ar[l]\ar|{j'}[u]\ar@{}|{(7)}[rd] & \Omega Y\ar|{\partial_g}[l]\ar|{-\Omega f}[u]\ar@{}|{(8)}[rd] & \Omega F\ar|{-\Omega i}[l]\ar[u]\ar@{}|{(9)}[rd] & {}*\ar[l]  \\
&& {}*\ar[u] & \Omega X\ar[l]\ar|{-\Omega g}[u] & \Omega H\ar|{-\Omega k}[l]\ar|{-\Omega g'}[u] & \Omega G\ar|{-\Omega j'}[l]\ar[u] 
}
$$
This diagram can be organized in an (unstable) octahedral diagram:
$$
\xymatrix@R=15pt@C=26pt{
Y\ar@{-->}_{\partial_g}[dd]\ar^f[rr] & & Z\ar@{-->}^{\partial_h}[dd]
 & Y\ar@{-->}_{\partial_g}[dd]\ar^f[rr] & & Z\ar@{-->}^{\partial_h}[dd]\ar@{-->}[ld] \\
\ar@{}|/-2pt/{(2+4)^*}[r] & X\ar|g[lu]\ar|{h}[ru]\ar@{}|{\mathcal D}[u] & \ar@{}|/-6pt/{(1+2)^*}[l]
 & & F\ar@{-->}[ld]\ar|i[lu]\ar@{}|/2pt/{(1)^*}[u]\ar@{}|/2pt/{(!)}[l] & &  \\
G\ar_{j'}[rr]\ar[ru] && H\ar|k[lu]
 & G\ar_{j'}[rr] & \ar@{}|/-2pt/{(4)^*}[u] & H\ar|{g'}[lu]
}
$$
\end{enumerate}
in such a way that dotted arrows means a map such that the source is composed with the functor $\Omega$,
 the diagrams indicated by a symbol $(-)^*$ represent cofiber sequences corresponding to the pullback diagram obtained
 by eventually pasting the pullback diagrams coming from the first diagram in $\C$. The other parts of the diagram are commutative diagram,
 except for diagram $(!)$ which is anti-commutative.\footnote{The sign comes from the sign in the bottom map of diagram (6).}
 Finally, one has the two homotopies:
$$
(2): i g' \simeq g k \qquad (5): j' \partial_g \simeq \partial_h (-\Omega f).
$$
\end{proposition}
\begin{proof}
The property (Fib1) follows from the properties of pushouts in $\C$.
 The subtlety of (Fib2) comes from the needed signs. They follow from the definition of fiber sequences,
 as explained in \cite[Lem. 1.1.2.9 and proof of (TR2) in section 1.1.2]{LurieHA}.
 (Fib3) follows from the functoriality of pushouts.
 The existence and uniqueness of the first diagram in (Fib4) is a consequence of the existence and uniqueness of pullback squares,
  through an obvious iterative construction, which starts by adding the pullback squares (1), (2) and so on.
 Signs appear for the same reason that for axiom (Fib2).
 The translation into an octahedron diagram as in the final assertion follows from the existence of the first diagram.
\end{proof}

\begin{remark}\label{rem:fibers&cofibers&triangulated}
\begin{myenum}
\item All the above properties hold dually for cofiber sequences. One changes the direction of all the maps
 and replace $\Omega$ by $\Sigma$ (signs remain).
\item These axioms are more or less classical in homotopy theory, with the exception of the octahedron axioms.\footnote{We
 refer the reader to \cite[Rem. p. 3.10]{Quillen} for an amusing commentary about this axiom.}
 It is remarkable that the axioms of $\infty$-categories turn all of Verdier's axioms into properties.
\item The octahedron axiom is usually not fully stated. In particular,
 the relations (2) and (5) are often forgotten. However, we will use all these relations
 in our analysis of the unstable spectral sequence associated with a tower
 (see \Cref{thm:degeneracy}). We refer the reader to \cite[1.1.5]{BBD} for
 the complete statement in a triangulated category.\footnote{Beware that we have rotated the octahedron diagram.}
\end{myenum}
\end{remark}

\begin{notation}
Further, there is a property which is specific to the unstable case and concerns
 the compatibility of fiber and cofiber sequences. 
 It was underlined by Quillen in \cite[I.3, prop. 6]{Quillen}.

Let $\C$ be a pointed $\infty$-category with finite limits and colimits.
 In this case, the functor $\Sigma$ is left adjoint to $\Omega$.\footnote{This
 follows essentially as $\Sigma$ (resp. $\Omega$) is computed as a finite colimit
 (resp. finite limit), therefore left (resp. right) adjoint of a constant diagram
 functor.}
 Quillen's result can be stated in $\C$ as follows.
\end{notation}

\begin{proposition}
Let $\C$ be a pointed $\infty$-category with finite limits and colimits.
 Consider a diagram: $A \xrightarrow g B \xrightarrow \tau Y \xrightarrow f X$ in $\C$.
 Then it can be essentially uniquely completed into the following commutative diagram in $\C$
$$
\xymatrix@=10pt{
A\ar^g[r]\ar_\alpha[d] & B\ar[r]\ar_\beta[d]\ar|\tau[rd]
 & Q\ar[r]\ar^\gamma[d] & \Sigma A\ar^\delta[d] \\
\Omega X\ar[r] & F\ar[r] & Y\ar^f[r] & X 
}
$$
in such a way that the upper (resp. lower) line
 is the beginning of a long cofiber (resp. fiber) sequence.
 Moreover, there exists an equivalence: $\delta=-\hat \alpha$ where
 $\hat \alpha$ is the map adjoint to $\alpha$ with respect to the adjunction
 $(\Omega,\Sigma)$, and the minus sign correspond to the canonical automorphism
 of $\Omega$ (see \Cref{num:loop&suspension}). 
\end{proposition}
\begin{proof}
The existence of the left (resp. right) part of the commutative diagram follows
 as $F$ and $\Omega X$ (resp. $Q$ and $\Sigma A$) are computed as fibers 
 (resp. cofibers).
 The last assertion follows from functoriality of pullbacks, paying attention to the
 orientations of the resulting squares and use \cite[Lem. 1.1.2.9]{LurieHA} to get the correct sign.
\end{proof}

We will use this proposition in the following corollary.
\begin{corollary}\label{cor:homotopy}
Let $g:A \rightarrow B$ and $f:Y \rightarrow X$ be morphisms,
 fitting in the following commutative square in $\C$:
$$
\xymatrix@=10pt{
\Sigma A\ar^{-\Sigma g}[r]\ar_\phi[d] & \Sigma B\ar^\psi[d] \\
X\ar^f[r] & Y
}
$$
Then there exists an essentially unique commutative diagram in $\C$:
$$
\xymatrix@R=12pt@C=22pt{
A\ar^g[r]\ar_{\hat \phi}[d] & B\ar[r]\ar^{\hat \psi}[d]
 & Q\ar[r]\ar|\eta[d] & \Sigma A\ar^-{-\Sigma g}[r]\ar_\phi[d] & \Sigma B\ar^\psi[d] \\
\Omega Y\ar^{-\Omega f}[r] & \Omega X\ar[r] & F\ar[r] & Y\ar^f[r] & X 
}
$$
in such a way that the upper (resp. lower) line is the beginning of a long cofiber
 (resp. fiber) sequence.
\end{corollary}

\begin{notation}
Let again $\C$ be an $\infty$-category with finite limits.
 Recall from \cite[6.1.2.7]{LurieHTT} that a \emph{groupoid object} of $\C$
 is a simplicial object $G_\bullet:\Delta^{op} \rightarrow \C$ satisfying the
 \emph{Segal condition} (see \emph{loc. cit.} 6.1.2.6).
 Examples are provided by the \v Cech simplicial object $\check S_\bullet(Y/X)$ associated with any morphism
 $p:Y \rightarrow X$ in $\C$.

One says (\emph{loc. cit.} 7.2.2.1) that $G_\bullet$ is a \emph{group object} if $G_0 \simeq *$.
 Note that $G_\bullet$ can also be seen as an $A_\infty$-algebra in $\C$ according to \cite[5.1.3.3]{LurieHA}.
 One can see $G_\bullet$ as an object $G_1$ of $\C$ equipped with a multiplication map
 $G_1 \times G_1 \simeq G_2 \rightarrow G_1$,\footnote{The first equivalence exists because of the Segal condition.}
 the simplicial structure encoding the associativity and unity properties.

Usually, one just says that $G=G_1$ is a group object in $\C$,
 referring to the underlying simplicial object $G_\bullet$ as the structure of $G$.
\end{notation}

\begin{example}\label{ex:loop_group_object}
Consider a pointed object $x:* \rightarrow X$ in $\C$.
 Then the loop space $\Omega_x X=* \times_X *$ (see \Cref{num:loop&suspension})
 admits a canonical group object structure $\Omega^\bullet_x(X)$ in $\C$,
 given by the \v Cech construction $\check S_\bullet(x:* \rightarrow X)$.
 In particular, the $n$-th component of this group object structure
 is the iterated loop space $\Omega^n_x(X)$.
 Moreover, the automorphism indicated by the minus sign in \Cref{num:loop&suspension}
 does in fact correspond to the inverse morphism with respect to that group
 structure.
\end{example}

\begin{notation}
Consider again an $\infty$-category with finite limits.
 Let $G$ be a group object in $\C$.
Following \cite[Def. 3.1]{NSS}, an action of $G$ on an object $X$ in $\C$
 is a groupoid object $(X//G)_\bullet$ such that $(X//G)_0=X$
 and with a morphism of groupoid objects $(X//G)_\bullet \rightarrow G_\bullet$.
 By looking at the first two stages of the simplicial object $(X//G)_\bullet$
 this determines a diagram
$$
\xymatrix@=40pt{
d_1^0,d_1^1:G \times X\ar@<2pt>[r]\ar@<-2pt>[r] & X
}
$$
such that $d_1^1$ corresponds to the projection on the second factor.
 The other map $d_1^0:G \times X \rightarrow X$ corresponds to the bare action
 of $G$ on $X$, while the rest of the data can be seen as coherences.

Let us now assume $\C$ is pointed, and consider a fiber sequence:
\begin{equation}\label{eq:mon_fiber_seq}
F \xrightarrow i Y \xrightarrow f X
\end{equation}
Then one deduces an action of the loop space $\Omega X$ with its canonical group structure
 on the fiber $F$ as the map of \v Cech simplicial objects
$$
\mu_\bullet:(F//\Omega X)_\bullet:=\check S_\bullet(i:F \rightarrow Y)
 \rightarrow \check S_\bullet(* \rightarrow X)=\Omega X_\bullet
$$
associated with the underlying pullback square.
 As explained previously, this determines a canonical map
$$
\mu:\Omega X \times F \rightarrow F.
$$
\end{notation}
\begin{definition}\label{df:monodromy}
Consider the above notation.
 Then the action $\mu_\bullet$ of $\Omega X$ on the fiber $F$ of $p$ will be
 called the \emph{monodromy action} associated with the fiber sequence \eqref{eq:mon_fiber_seq}.
\end{definition}

\begin{remark}
The dual monodromy operation exists for cofiber sequences
 (as a coaction of a cogroup $\Sigma A$ on the cofiber).
\end{remark}

\subsection{Homotopy functors in higher homotopy theory}\label{sec:htp_functors}

\begin{notation}\label{num:1-topos}
We will work inside a topos $\E$,
 using the internal notion of (abelian) groups
 in $\E$ called $\E$-groups.

We also consider
 the category of pointed objects of $\E$, $\E_*=*/\E$. 
 An $\E$-group object
 will always be assumed to be pointed by its neutral element.
 Note that it makes sense to speak about the kernel of maps in $\E_*$,
 as well as image and cokernel in $\E_*$ or in the category of $\E$-groups.
 Given a morphism $f:X \rightarrow Y$ of pointed objects
 (resp.  $u:G \rightarrow H$ of $\E$-group objects), one has the following equivalent conditions:
\begin{enumerate}[label=(T\arabic*)]
\item $f$ is an epimorphism $\Leftrightarrow$ $\coker(f)=*$ $\Leftrightarrow$  $\im(f)\rightarrow H$ iso.
\item $u$ is a monomorphism $\Leftrightarrow$ $\Ker(u)=*$.
\end{enumerate} 

A group $G$, being a monoid for the cartesian structure,
 can act on the left on an object (resp. pointed objet) $X$ of $\E$ as in the usual set theory.
 One says that $X$ is a $G$-object (aka $G$-equivariant object).
 The notion of morphism of $G$-objects, aka $G$-equivariant morphism, is clear.
 The \emph{orbit object} $X/G$ of a $G$-object $X$ is the coequalizer of 
 $\xymatrix@=10pt{G \times X\ar@<2pt>^-\gamma[r]\ar@<-2pt>_-p[r] & X}$
 where $\gamma$ is the $G$-action map and $p$ the projection.
 The action of $G$ on $X$ is transitive if $X/G=*$.
 We will also use the following equivalent conditions on a $G$-equivariant pointed map
 $f:X \rightarrow Y$ in the sequel:
\begin{enumerate}[label=(T\arabic*),resume]
\item $f$ is a monomorphism $\Leftrightarrow$ $\Ker(f)=*$.
\end{enumerate}
\end{notation}
\begin{definition}\label{df:pi*-structure}
 A \emph{$\pi_*$-structure} in $\E$ is an $\N$-graded object $G_*=(G_n)_{n \in \N} \in (\E_*)^{\N}$
 such that $G_1$ is an $\E$-group,
 and for all $n>1$, $G_n$ is an abelian $\E$-group,
 with an action of $G_1$ by $\E$-group automorphisms.

A morphism of $\pi_*$-structures is a morphism $f:G_* \rightarrow H_*$ which is a homogeneous morphism
 of degree $0$ of $\N$-graded objects, such that each $f_n:G_n \rightarrow H_n$ respects the relevant algebraic structures.
 The corresponding category is denoted by $\E_{\pi_*}$.
\end{definition}

\begin{remark}\label{rem:pi*-(co)limits}
The category $\E_{\pi_*}$ admits finite products and filtered colimits.
 They both are computed term-wise,
 in the categories of pointed $\E$-objects, $\E$-groups, abelian $\E$-groups in the appropriate degrees.
\end{remark}

\begin{example}\label{ex:pi_*-struct}
\begin{myenum}
\item Let $\T$ be an $\infty$-topos (\cite[Chap. 6]{LurieHTT}), $\T_*$ the $\infty$-category of pointed objects of $\T$.
 Let $X$ be an object of $\T_*$.
 We let $\T^\disc$ be the sub-$1$-category of $\T$ made of discrete objects,
 \emph{i.e.} the underlying topos, and let $\pi_0:\T_* \rightarrow \T^\disc_*$ be the natural projection
 (in other words, $\T^\disc$ is the $0$-th truncation of the $\infty$-category $\T$, \cite[Not. 5.5.6.2]{LurieHTT}).

 For any integer $n \geq 0$, we put $\pi_n(X):=\pi_0(\Omega^nX)$.
 It now readily follows from the canonical group object structure on the loop space $\Omega X$
 (see \Cref{ex:loop_group_object}) that $\pi_*(X)$ is a $\pi_*$-structure of $\T^\disc$.
 Note that our definition coincides with \cite[6.5.1]{LurieHTT}. 
\item An important example for us comes from the $\A^1$-homotopy category $\HA(S)$.
 As it is the $\A^1$-localization of the Nisnevich $\infty$-topos $\iSh(\Sm_S)$ on the smooth site over $S$,
 one has a forgetful functor $\cO:\HA(S) \rightarrow \iSh(\Sm_S)$.
Given a pointed object $\cX$ of $\HA(S)$, one deduces a $\pi_*$-structure
 $\piA_*(\cX):=\pi_*(\cO(X))$ in the topos of Nisnevich sheaves $\Sh(\Sm_S)^\disc$ over $S$,
 using the construction of the preceding point.
\end{myenum}
\end{example}

The following definition is an abstraction of the properties of long exact sequences
 of homotopy groups.
\begin{definition}\label{df:long_htp_seq}
A \emph{long homotopy sequence} in $\E$ is a triangle
$
\xymatrix@=10pt{
F_*\ar^f[rr]  && G_*\ar^g[ld] \\
& H_*\ar@{=>}^\partial[lu]
}
$
such that:
\begin{enumerate}
\item $F_*$, $G_*$ and $H_*$ are $\pi_*$-structures in $\E$, $f$ and $g$ are morphisms of $\pi_*$-structures in $\E$.
\item $\partial:H_* \rightarrow F_*$ is a homogeneous morphism of degree $-1$ of $(\E_*)^\N$
 such that for all $n>1$, $\partial_n:H_n \rightarrow F_{n-1}$ is a morphism of groups.
\item $F_0$ is an $H_1$-object\footnote{beware it is not necessarily a pointed $H_1$-object:
 the orbit of the base point of $F_0$ can be non-trivial.}
 such that $\partial_1:H_1 \rightarrow F_0$
 (resp. $f_0:F_0 \rightarrow G_0$) is $H_1$-equivariant where $H_1$ acts by left multiplication on the source 
 (resp. trivially on the target).
\item The image of $\partial_2:H_2 \rightarrow F_1$ lands in the center of $F_1$.
\item The composite of any two consecutive maps is trivial.
\end{enumerate}
One says that this long homotopy sequence is \emph{exact} if for all $n \geq 0$, one has an equality of sub-objects:
$$
\im(f_n)=\Ker(g_n), \im(g_{n+1})=\Ker(\partial_{n+1}), \im(\partial_{n+1})=\Ker(f_n),
$$
and moreover, the map $\tilde f_0:F_0/H_1 \rightarrow G_0$ induced by $f_0$ according to point (3) is a monomorphism.
\end{definition}
Following Morel (see \cite[\textsection 2.2]{MorelLNM}), we have used the symbol $\Rightarrow$ to indicate the boundary map $\partial$,
 meaning this map is actually extended to an action of the left hand-side on the right hand-side.
 We will also use the following notation to indicate such a sequence:
$$
\hdots H_2 \xRightarrow{\partial_2} F_1 \xrightarrow{f_1} G_1 \xrightarrow{g_1} H_1 \xRightarrow{\partial_1} F_0 \xrightarrow{f_0} G_0 \xrightarrow{g_0} H_0
$$

\begin{remark}
\begin{myenum}
\item Axiom (2) can also be formulated by saying that for all $n>1$, $F_{n-1}$ is an $H_{n}$-object such that
 $\partial_n:H_n \rightarrow F_{n-1}$ is $H_n$-equivariant where $H_n$ acts on the source by left multiplication.
\item An important consequence of the axiom (3) is that the map $\partial_1$ is equal to the composite
 $H_1 \times * \xrightarrow{\id \times x} H_1 \times F_0 \xrightarrow \gamma F_0$
 where $x$ is the base point of $F_0$ and $\gamma$ is the action of $H_1$ on $F_0$.
\end{myenum}
\end{remark}

\begin{example}\label{ex:long_htp_seq}
Consider the setting of \Cref{ex:pi_*-struct}(1). Let $f:Y \rightarrow X$ be a morphism in $\T_*$. Let $i:F=\Ker(f) \rightarrow Y$ be the homotopy fiber of $f$.
 On deduces from the long fiber sequence (Fib2) of \Cref{prop:unstable_triangulated} a homotopy exact sequence in $\T^\disc$:
$
\xymatrix@=10pt{
\pi_*(F)\ar^{i_*}[rr]  && \pi_*(Y)\ar^{f_*}[ld] \\
& \pi_*(X)\ar@{=>}^\partial[lu]
}
$
where the $\pi_*$-structures come from \Cref{ex:pi_*-struct}
 and the action of $\pi_1(X)$ on $\pi_0(F)$ is given by the
 monodromy action of \Cref{df:monodromy}.
\end{example}

The preceding definition allows us to extend the classical definition of a homological functor:
\begin{definition}\label{df:htp_functor}
Let $\C_*$ be as above and $\E$ be a $1$-topos.

A \emph{homotopy functor} is an $\infty$-functor $\Pi_*:\C_* \rightarrow \nrv \E_{\pi_*}$,
 which turns fiber sequences (\Cref{df:(co)fiber_seq}) into long homotopy exact sequences.
  Dually, a cohomotopy functor on $\C$ will be a homotopy functor starting from the opposite of $\C_*$:
 so that  it turns cofiber sequences into long homotopy exact sequences.

In addition, we will say that the homotopy functor $\Pi_*$ is \emph{additive} if it commutes with products.
 As a result of this definition, a cohomotopy functor is additive if it sends
 coproducts (of pointed $\E$-objects) to products in the category of $\pi_*$-structures
 (see \Cref{rem:pi*-(co)limits}).
\end{definition}

\begin{example}
\begin{myenum}
\item \Cref{ex:long_htp_seq} implies that $\pi_*:\T_* \rightarrow \T^\disc_{\pi_*}$
 is a homotopy functor.
\item Consider a pointed $\infty$-category $\C_*$.
 Then $\C_*$ admits a canonical (unique in a suitable sense) enrichment in pointed spaces.
 Then for objects $A$ and $X$ of $\C_*$, the $\pi_*$-structure in sets $\pi_*(\Map_{\C_*}(A,X))$
 defines a cohomotopy functor in $A$ (for $X$ fixed) and a homotopy functor in $X$ (for $A$ fixed).
\end{myenum}
\end{example}

The following definition is motivated by the unstable Bousfield-Kan spectral sequence that will be reviewed in the next section (see there for examples).
\begin{definition}\label{unstable-cx}
A \emph{homotopical complex} with coefficients in $\E$ is an $\N$-graded object $E_*$ of $\E_*$ with a homogeneous morphism
 $d:E_* \rightarrow E_*$ in $\E_*^\N$ of degree $-1$ such that
\begin{itemize}
\item $d \circ d=*$.
\item For all $n>1$, $E_n$ is an abelian $\E$-group, $E_1$ is an $\E$-group and $E_0$ is a pointed $\E$-object,
 and an $\E_1$-object (after forgetting the base point).
\end{itemize}
The homotopy groups of the complex $E_*$ are defined for $n>0$ as:
$$
\pi_n(E_*)=\Ker(d_n)/\im(d_{n+1}).
$$
One says that the complex is \emph{exact} if for all $n>0$, $\pi_n(E_*)=*$.\footnote{Equivalently, kernel = image at all possible places.}

An \emph{augmentation} of $E_*$ is an $E_1$-equivariant map $\epsilon:E_0 \rightarrow F$ of pointed objects in $\E$
 where $F$ has been given the trivial $E_1$-action.
 One also says that $(E_*,\epsilon)$ is an augmented complex. Then one defines the homotopy groups of $(E_*,\epsilon)$
 by the formula $\pi_n(E_*,\epsilon)=\pi_n(E_*)$ for $n>0$, and $\pi_0(E_*,\epsilon)=\Ker(\epsilon)/E_1$ (orbit space).
 We will say that the augmented homotopy complex $(E_*,\epsilon)$ is \emph{exact}
 if for all integers $n\geq 0$, $\pi_n(E_*,\epsilon)=*$. One further says that $(E_*,\epsilon)$ is \emph{strongly exact}
 if it is exact and the induced map  $\tilde \epsilon:E_0/E_1 \rightarrow F$ is a monomorphism.
\end{definition}
In other words, the exactness condition concerning the augmentation
 means that $E_1$ acts transitively on the kernel $\Ker(\epsilon)$.

\begin{notation}\label{num:cohomotopical_notation}
We will further consider \emph{bounded} homotopical complexes:
 such a homotopical complex $E_*$ as above is said to be \emph{$d$-truncated}
 if for all $n>d$, $E_n=*$.

To relate to the classical literature on the coniveau spectral sequence,
 we will use a ``cohomotopical'' indexing for truncated homotopical complexes:
 a $d$-truncated homotopical complex $C^*$ will be denoted as:
$$
* \rightarrow C^0 \xrightarrow{d^0} C^1 \xrightarrow{d^1} \cdots \Rightarrow C^d
$$
with the previous notation, one has: $C^n=E_{d-n}$.
 An augmentation map $\epsilon:E_0 \rightarrow F$ of $E_*$ will be called a \emph{coaugmentation}
 $\epsilon:C^d \rightarrow F$ of the associated cohomotopical complex $C^*$.
 Then the notions of exactness and strong exactness of $(C^*,\epsilon)$ are the same.

Such a complex can also be co-augmented (we say ``bi-augmented'') by a map $X \xrightarrow \tau C^0$
 which is a morphism of pointed sets, groups or abelian groups if $d=0, 1$ or $d\geq 2$,
 such that $d^0 \circ \tau =*$.
 The co-augmented complex is exact if $\Ker(\tau)=*$ and $\im(\tau)=\Ker(d^0)$.
 Strong exactness is defined as for augmented complexes.

Following Morel \cite[Def. 2.20, 6.1]{MorelLNM}, we further adopt the following
 notation:
\end{notation}
\begin{definition}\label{df:coaugmented_cohtp_cpx}
Consider a $d$-truncated cohomotopical complex $C^*$ as above.
 An \emph{augmentation} of $C^*$ will be a map
 $X \xrightarrow \tau C^0$
 which is respectively a morphism of pointed sets, groups or abelian groups if $d=0, 1$ or $d\geq 2$,
 and such that $d^0 \circ \tau =*$.
 One says that $(C^*,\tau)$ is augmented.
 One says that $(C^*,\tau)$ is exact if $\Ker(\tau)=*$ and $\im(\tau)=\Ker(d^0)$.

When $C^*$ admits an augmentation $\tau$ and a co-augmentation $\epsilon$,
 one says that $(C^*,\tau,\epsilon)$ is bi-augmented.
 Notions of exactness (resp. strong exactness) are to be taken
 for the augmented and the co-augmented cohomotopical complex.
\end{definition}

\begin{remark} The above definition contains three important and rather distinct cases:
\begin{myenum}
\item When $d=0$,
 an augmented complex is given by a diagram $* \rightarrow X \xrightarrow \tau C^0$.
 Exactness simply means that $\Ker(\tau)=*$.
\item The case $d=1$ is possibly the most intricate, and corresponds
 to Morel's Definition 2.20 in \cite{MorelLNM} (with slightly more flexibility).
 A bi-augmented $1$-truncated cohomotopical complex is given by a diagram:
$$
* \rightarrow X \xrightarrow \tau C^0 \xRightarrow d C^1 \xrightarrow \epsilon F
$$
where $X$ and $C^0$ are just $\E$-groups, while $C^1$ and $F$ a pointed $\E$-objects.
 The actual complex $C^*$ is only given by one map $d$,
 which is in fact an action of the $\E$-group $C^0$ on the pointed $\E$-object $C^1$.

Recall that exactness on the right means that $C^0$ acts transitively
 on $\Ker(\epsilon)$ while strong exactness implies in addition that the induced (pointed) map
 $\tilde \epsilon:C^1/C^0 \rightarrow F$ is a monomorphism (of pointed objects of $\E$).
 Here, $C^1/C^0$ means the orbit "space" of $C^0$ acting on $C^1$.
\item Finally for $d>1$, we can see a bi-augmented $d$-truncated complex $(C^*,\tau,\epsilon)$
 as a co-augmented $(d+1)$-truncated complex, by interpreting the augmentation $\tau$
 as a differential starting from degree $-1$.
\end{myenum}
\end{remark}

We end this section with an important technical point that will be used
 for defining and studying the unstable coniveau spectral sequence (see \Cref{num:coniveau_ssp}).
\begin{proposition}\label{prop:colimits-htp-seq}
Let $\E$ be a topos.
 The category of long homotopy sequences (resp. homotopy complexes, augmented homotopy complexes, bi-augmented complexes) in $\E$
 admits filtered colimits. These colimits are term-wise computed by filtered colimits of abelian $\E$-groups in degree bigger than $2$,
 filtered colimits of $\E$-groups in degree $1$, and by filtered colimit of pointed $\E$-objects in degree $0$.
 Finally, filtered colimits preserve exactness in all cases.
\end{proposition}
\begin{proof}
We consider only the case of long homotopy sequences as the other cases are similar.
 As the associated sheaf functor commutes with colimits, it is sufficient to consider the case of discrete toposes,
 and therefore one is reduced to the case $\E=\Set$.

Let $\{ (F_*^i, G_*^i, H_*^i)\}_{i\in I}$ be an $I$-filtered diagram of \emph{long homotopy sequences} in sets:
$$
\hdots H_2^i \xRightarrow{\partial_2} F_1^i \xrightarrow{f_1} G_1^i \xrightarrow{g_1} H_1^i \xRightarrow{\partial_1} F_0^i \xrightarrow{f_0} G_0^i \xrightarrow{g_0} H_0^i
$$
The proposition states that in the category of exact long homotopy sequences,
 $\colim_{i\in I} (F_*^i, G_*^i, H_*^i)= (F_*, G_*, H_*)$ where for $\Box=F, G, H$
\begin{align}\Box_*&=\begin{cases}
\colim_i \Box_*^i &  \text{if} \ \ast >1 \\
\colim'_i \Box_1^i &  \text{if} \ \ast =1 \\
\colim'' \Box_0^i &  \text{if} \ \ast =0
\end{cases}
\end{align}
Here $\colim_i \Box_*^i$ for $\ast>1$ (resp. $\colim'_i \Box_1^i$, and resp. $\colim''_i \Box_0^i$) denotes respectively the colimit in the category of pointed sets
 (resp. groups, and resp. abelian groups).
 In all three cases, the colimit is computed by taking the appropriate quotient of the corresponding coproduct (obtained after forgetting maps in $I$),
 that is respectively the wedge sum, free product and direct sums.
 In the case of groups, $\colim'_i \Box_1^i$ is given by the quotient of the free product $\bigstar_{i\in I}\Box_1^i$ of groups $\Box_1^i$
 by the sub-group generated by the relations in $\Box_1^k$ and $(s_j\circ \mu_{ij}(g_i))\cdot s_i(g_i)^{-1}$ where for $i\leq j$, $\mu_{ij}: \Box_1^i\to \Box_1^j$ and $s_i: \Box_1^i\to \bigstar_{i\in I} \Box_1^i$.

We first prove that the sequence 
\begin{equation}\label{eq:colim_long_seq}
\hdots H_2 \xRightarrow{\partial_2} F_1 \xrightarrow{f_1} G_1 \xrightarrow{g_1} H_1 \xRightarrow{\partial_1} F_0 \xrightarrow{f_0} G_0 \xrightarrow{g_0} H_0
\end{equation}
is a long homotopy sequence. The action of $H_1$ on $F_0$ is obtained by taking the colimit of the action maps:
$$
H_1^i \times F_0^i \rightarrow F_0^i
$$
and using the fact filtered colimits preserve finite products. The rest of the axioms of \Cref{df:long_htp_seq}, \emph{i.e.} points (3), (4) and (5),
 are clear.
 It is then easy to check that the long exact sequence \eqref{eq:colim_long_seq} satisfies the universal properties of colimits.

We finally prove that, if all long homotopy sequences $(F_*^i, G_*^i, H_*^i)$ are exact, then \eqref{eq:colim_long_seq} is exact.
\paragraph{\bf Exactness at $F_i, H_i$ and $G_i$,  for $i>1$:} clear as filtered colimits of abelian groups are exact.
\paragraph{\bf Exactness at $F_1$:} we need to show that the sequence 
$$ H_2 \xRightarrow{\partial_2} F_1 \xrightarrow{f_1} G_1 $$
is exact. 
 We have exact sequence $$\colim'_{i\in I} H_2^i \xRightarrow{\partial_2}\colim'_{i\in I}  F^i_1 \xrightarrow{f_1} \colim'_{i\in I} G^i_1 $$ of $\E$-group objects.  
 Now since the image of $H^i_2\to F^i_1$ is in the center of $F^i_1$ (axiom (4) of \Cref{df:long_htp_seq}),
 it follows that the map $ \colim'_{i\in I} H_2^i \xRightarrow{\partial_2} \colim'_{i\in I}  F^i_1$ factors through the abelianization $ (\colim'_{i\in I} H_2^i)^{ab}=\colim_{i\in I} H_2^i$, where $\colim_{i\in I} H_2^i$ is the colimit in the category of abelian groups, and this concludes.

\paragraph{\bf Exactness at $G_1$:} as in the abelian case.

\paragraph{\bf Exactness at $H_1$:} to prove the exactness at $H_1$, we need to show that given $h\in H_1$ such that $h\cdot \ast=\ast$, then $h$ is in the image of $G_1\to H_1$.  
Since $h\in \colim'_{i\in I} H^i_1$ and $I$ is filtered, there exists an index $i\in I$ and an element $h_i\in H_1^i$ such that $h=s_i(h_i)$ for $s_i: H_1^i\to \colim'_{i\in I} H^i_1$. Hence $h\cdot \ast=\ast$ in $F_0=\colim''_{i\in I} F_0^i $ implies that $\mu_{ij} (h_i)\cdot *=*$ in $F_0^j$ for some $j\geq i$. By exactness of the sequence $ G^j_1\to H^j_1 \xRightarrow{\partial_2} F^j_0$, the element $\mu_{ij} (h_i) $ is in the image of the map $G^j_1\to H^j_1$.  Hence $h=s_i(h_i)=s_j(\mu_{ij}(h_i))$ is in the image of the map $G_1\to H_1$.  This implies the exactness at $H_1$.

\paragraph{\bf Exactness at $F_0$:} we need to show $f_0^{-1}(\ast)=\im (\partial_1).$ For $x\in f_0^{-1}(\ast)$, there is $i\in I$ such that $x\in F_0^i$ and $ f_0^i(x)=\ast$ in $G^i_0$.  Hence by exactness of $H_1^i\to F^i_0\to G_0^i$, there is $h_i\in H_1^i$ such that  $\partial_1^i(h_i)=x$. Hence, $\partial_1(s_i(h_i))=x$ in $H_1$. This implies the exactness at $F_1$.

\paragraph{\bf Exactness at $G_0$:} This is clear.
\end{proof}

\begin{remark}
In particular, one deduces that the category of left (resp. right) unstable exact couples admits filtered colimits.
 Moreover, one can check that the derived exact couple functor, as well as the associated spectral sequence functor,
 commutes with filtered colimits.
\end{remark}

\subsection{Unstable exact couples and spectral sequences}

The notion of exact couples in unstable homotopy is very old (see \cite{Fed56}, but also \cite{BK72}).
 It is a delicate subject, because one needs to take care of $\pi_1$-actions (\emph{i.e.} monodromy as
 explained in \Cref{df:monodromy}).
 The next definition is a synthesis of the approaches via exact couples (and Rees system, see \Cref{df:two_exact_couples})
 and the direct one from \cite[IX, \textsection 4]{BK72}, with the aim to apply it to $\infty$-categories
 and the previously introduced homotopy exact functors.
\begin{definition}\label{df:unstable_ecpl}
Let $\E$ be a topos.
 A \emph{right (\emph{resp.} left) unstable exact couple of degree $1$} is a triangle
$$
\xymatrix@R=12pt@C=30pt{
D^{**}\ar_\alpha^{(-1,-1)}[rr]  && D^{**}\ar@{=>}_/5pt/\beta^{(1,0)}[ld]
 & \text{resp.} & \bar D^{**}\ar_\alpha^{(-1,-1)}[rr]  && \bar D^{**}\ar_/5pt/{\bar \beta}^{(0,0)}[ld]
 \\
& E^{**}\ar_/-5pt/\gamma^{(0,0)}[lu] & 
 && & E^{**}\ar@{=>}_/-5pt/{\bar \gamma}^{(1,0)}[lu] & 
}
$$
of objects of $\E_{\pi_*}^\Upsilon$, $\Upsilon=\{ (p,q)\in \Z^2 \mid q-p\geq 0\}$, such that
 the three maps are bi-homogeneous with indicated bidegree and for all $q$, the sequence
\begin{align*}
\cdots \to D^{p, q+1}\xrightarrow{\alpha} D^{p-1, q} \xRightarrow{\beta} E^{p, q} \xrightarrow{\gamma} D^{p, q} \xrightarrow{\alpha} \cdots  D^{q, q+1} \xRightarrow{\beta} E^{q+1, q+1} \xrightarrow{\gamma} D^{q+1, q+1}  \xrightarrow{\alpha} D^{q, q} \\
\text{resp.}
\cdots \to \bar D^{p+1, q+1}\xrightarrow{\bar \alpha} \bar D^{p, q}
 \xRightarrow{\bar \beta} E^{p, q} \xrightarrow{\bar \gamma} \bar D^{p+1, q} \xrightarrow{\bar \alpha}
 \cdots  \bar D^{q, q+1} \xRightarrow{\bar \beta} E^{q, q+1} \xrightarrow{\bar \gamma} \bar D^{q+1, q+1}  \xrightarrow{\bar \alpha} \bar D^{q, q}
\end{align*}
is an exact long homotopy sequence in $\E$, in the sense of \Cref{df:long_htp_seq}.

Given any integer $r \geq 1$, a right (resp. left) unstable exact couple of degree $r$ is defined similarly as above
 except that $\gamma$ (resp. $\bar \beta$) has degree $(r-1,r-1)$.
\end{definition}
The situation of right and left unstable exact couples is completely analogous.
 In the following, without precision, an unstable exact couple will be a right one.

\begin{example}\label{ex:unst_exact_couples}
\begin{myenum}
\item Let $\T$ be an $\infty$-topos, and consider a (decreasing) tower in $\T$:
$$
\cdots \rightarrow X_n \xrightarrow{f_n} X_{n-1} \rightarrow X_1 \rightarrow X_0
$$
 \ie an object of the $\infty$-category $\T^{\overleftarrow \N}$, where $\overleftarrow \N$ is the category
 associated with the opposite ordered set of non-negative integers.
Then one gets a right unstable exact couple of degree $1$ by considering the homotopy fibre $i_p:F_p=\Ker(f_p) \rightarrow X_p$ and putting:
$$
\cdots \xrightarrow{\alpha=f_{p*}} D^{p-1,q}=\pi_{q-p+1}(X_{p-1}) \xRightarrow{\beta=\partial}
 E^{p,q}=\pi_{q-p}(F_p)  \xrightarrow{\gamma=i_{p*}} D^{p,q}=\pi_{q-p}(X_p)
$$
according to \Cref{ex:long_htp_seq}(1).

Dually, if one considers the homotoy cofiber $\pi_p:X_p \rightarrow \coker(f_p)=:C_p$, one gets a left unstable
 exact couple of degree $1$, $E^{p,q}=\pi_{q-p}(C_p)$, $\bar D^{p,q}=\pi_{q-p}(X_p)$.
\item Let $\Pi_*:\C_* \rightarrow \E_{\pi_*}$ be a homotopy functor in the sense of \Cref{df:long_htp_seq}.
 Then one associates to any tower $X_\bullet$ in $\C_*^{\overleftarrow \N}$
 a right unstable exact couple of degree $1$
 using the above recipe: $D^{p,q}=\Pi_{q-p}(X_p)$, $E^{p,q}=\Pi_{q-p}(F_p)$, where $F_p$ is the homotopy fibre
 of $X_p \rightarrow X_{p-1}$.
\item The preceding considerations dualize in the obvious way: given a cohomotopy functor $\Pi_*:\C_*^{op} \rightarrow \E_{\pi_*}$,
 one associates to any (increasing) tower\footnote{this notation is designed to fit in with the case
 of the coniveau tower in \Cref{num:coniveau_ssp};}
 $X^\bullet$ in $\C_*^{\overrightarrow \N}$ a right unstable exact couple of degree $1$
 by considering the homotopy cofiber $C^p=\coker(X^{p-1} \rightarrow X^p)$,
 and using the formulas: $D^{p,q}=\Pi_{q-p}(X^p)$, $E^{p,q}=\Pi_{q-p}(C^p)$.
\end{myenum}
\end{example}

\begin{remark}
The distinction between right and left exact couples is special to the unstable case.
 Up to our knowledge, it has not been highlighted in the literature yet, because one usually restricts
 to tower of fibrations.
 We refer the reader to the \Cref{df:two_exact_couples} for the relevance of considering both forms
 of exact couples.
\end{remark}

\begin{notation}\label{num:unstable_ssp}
Let $(D,E,\alpha,\beta,\gamma)$ be an exact couple of degree $r$.
 
One associates to it a co-augmented cohomotopical complex (\Cref{num:cohomotopical_notation})
 by defining the boundary operator on $E$ using the formula $d_r=\beta \circ \gamma$:
$$
\cdots \rightarrow E^{q-(n+1)r,q-(n+1)r+n+1} \rightarrow E^{q-nr,q-nr+n} \rightarrow \cdots \rightarrow E^{q-r,q-r+1} \xRightarrow{d_r^{q-r,q-r+1}} E^{q,q} \xrightarrow{\epsilon} \tilde D^{q,q}
$$
where $\tilde D^{q,q}$ is the cokernel of the pointed map $D^{q+r,q+r} \xrightarrow \alpha D^{q,q}$ and the co-augmentation map is the composite:
 $\epsilon:E^{q,q} \xrightarrow{\gamma} D^{q,q} \twoheadrightarrow \tilde D^{q,q}$.
 Note that $d_r$ is homogeneous of bidegree $(r,r-1)$.\footnote{We follow here the indexing
 of Bousfield and Kan, though it is unusual in the general theory of spectral sequences.}

Using the classical procedure, one can now define a \emph{derived exact couple} by the following formulas:
$$
D'=\im(\alpha), (E')^{p,q}=\pi^p(E_{r}^{*,q},\epsilon).
$$
The bi-grading on $D'$ is obtained by the bi-grading on $D$, according to the inclusion $D^{\prime p,q} \subset D^{p,q}$.
 Similarly, one defines the bi-grading on $E'$ such that $E^{\prime p,q}$ is a sub-quotient of $E^{p,q}$.
 Then one obtains the maps $\alpha'$, $\beta'$, $\gamma'$ as the one induced respectively by $\alpha$, $\beta$, $\gamma$
 by the universal properties of the image and of the homotopy of an augmented unstable complex.
 Note that the resulting homotopy exact couple $(D',E',\alpha',\beta',\gamma')$ has degree $r+1$.
\end{notation}
\begin{definition}\label{df:unstable_ssp}
We call $(D',E',\alpha',\beta',\gamma')$ the derived unstable exact couple
 associated with $(D,E,\alpha,\beta,\gamma)$.
 Iterating this procedure, for any $s>0$, one defines the $s$-th derived unstable exact couple
 $(D^{(s)},E^{(s)},\alpha^{(s)},\beta^{(s)},\gamma^{(s)})$ which has degree $r+s$.

We define the \emph{unstable spectral sequence} associated with the exact couple $(D,E,...)$
 by putting, for $n \geq r$, $s=n-r$:
$$
E_n^{p,q}=(E^{(s)})^{p,q},
$$
with differential $d_n^{p,q}=\beta^{(s)} \circ \gamma^{(s)}:E_n^{p,q}  \rightarrow E_n^{p+n+1,q+n}$.
 We view $E_n^{*,q}$ as an augmented unstable complex with augmentation
 $\epsilon^{(s)}=E_{n}^{q,q} \rightarrow \tilde D_n^{q+n,q+n}$.
\end{definition}
One will remember that the $(n+1)^{th}$ term $E_{n+1}^{*,q}$ is obtained as the (co)homotopy
 of the augmented unstable complex $E_{n}^{*,q}$. Moreover, we view it as an \emph{augmented} unstable complex
 (\Cref{unstable-cx}).

\begin{remark}\label{rem:compute_E_r_BK}
The consideration of augmented complexes seems to be new in this context.
 While not changing the definitions, we find it very enlightening.
 Note in particular, that one gets the following formulas to directly compute the $r$-th derived exact couple
 (as in \cite[IX, 4.1]{BK72}):
\begin{align*}
D_{r+1}^{p,q}&=\im\big(\alpha^r:D^{p+r,q+r} \rightarrow D^{p,q}\big), \\
E_{r+1}^{p,q}&=\Ker\big(E^{p,q} \rightarrow D^{p,q}/D^{p,q}_{r+1}\big)/\Ker\big(D^{p-1,q} \xrightarrow{\alpha^r} D^{p-r-1,q-r}\big),
\end{align*}
where the quotient when $p=q$ means the set of orbits for the relevant group action.
 On the other hand, the notion of derived unstable exact couple is useful for inductive arguments.
\end{remark}

\begin{example}\label{ex:unstable_ssp}
\begin{myenum}
\item Consider the notation of \Cref{ex:unst_exact_couples}(1).
 We let $X_\infty=\lim_n X_n$. Then the spectral sequence associated with the exact couple of \emph{loc. cit.}
 starts from the $E_1$-term and abuts to $\pi_{q-p}(X_\infty)$:
$$
E_1^{p,q}=\pi_{q-p}(F_p) \Rightarrow \pi_{q-p}(X_\infty).
$$
When $\T$ is the $\infty$-topos of spaces, and one starts from a tower of fibrations $X_\bullet$,
 this is precisely the Bousfield-Kan spectral sequence (\cite[IX, 4.2]{BK72}).
 We refer to \emph{loc. cit.}, 5.3 and 5.4 for conditions of convergence of this spectral
 sequence. In fact, strong convergence will always be fulfilled in our examples.
\item The same construction is valid starting from a homotopy functor $\Pi_*:\C_* \rightarrow \E_{\pi_*}$
 as in \Cref{ex:unst_exact_couples}(2), such that $\C_*$ admits sequential limits.
 The same observation holds dually for a cohomotopy functor as in \Cref{ex:unst_exact_couples}(3).
 The abutment is given by the cohomotopy of the colimit (topologically,
 this compares to a tower of cofibrations).
\end{myenum}
\end{example}

\begin{remark}
The above examples extend the discussion of \cite[1.2.2]{LurieHA} to the unstable case.

In fact, if we work in a stable $\infty$-category $\mathcal C_*$,
 and $\Pi_*$ is the restriction of a homology functor to non-negative degrees,
 then the spectral sequence in point (2) of the above example
 is half of the spectral sequence one naturally derives from the given tower by considering
 the whole homology functor.
 On the other hand, one can recover the whole spectral sequence by using
 the suspended tower $X_\bullet[n]$ for arbitrary $n \geq 0$. Thus, all the results
 that we will get in the unstable case will imply the analogous result
 for the complete spectral sequences in a stable situation.
\end{remark}

\subsection{The degeneracy criterion}\label{sec:degen_lemma}

\begin{notation}\label{num:tower_under}
We will consider a situation similar to \Cref{ex:unstable_ssp}(2).
 Let $\C_*$ be a pointed $\infty$-category and $\Pi_*:\C_* \rightarrow \E_{\pi_*}$
 be a homotopy functor.

We let $X$ be an object of $\C_*$ and consider a tower of objects under $X$,
$$
 \cdots \to X_n \xrightarrow{f_n} X_{n-1} \to\cdots \to X_1 \xrightarrow{f_1} X_0
$$
with projection maps $\pi_n:X \rightarrow X_n$. In other words,
 $X \rightarrow X_\bullet$ is an object of $(X/\C_*)^{\overleftarrow \N}$.
 We usually extend  this tower slightly by letting $f_0:X_0 \rightarrow X_{-1}=*$
 be the (essentially) unique map.
\end{notation}
\begin{definition}
We will say that the tower $X/X_\bullet$ is bounded if there exists an integer $d \in \N$
 such that for all $n>d$, $p_n:X \rightarrow X_n$ is an isomorphism
 in the $\infty$-category $\C_*$.
 We say that $X/X_\bullet$ is $d$-bounded ($\infty$-bounded for no condition).
\end{definition}

\begin{remark}
One can say that $X/X_\bullet$ is a coaugmented tower.
 This is the (pointed) $\infty$-categorical generalization of a cofiltered object.
\end{remark}

\begin{notation}\label{df:two_exact_couples}
Consider again the situation of an arbitrary coaugmented tower $X/X_\bullet$,
 as in \Cref{num:tower_under}.
 By applying the octahedron property (Fib4) of \Cref{prop:unstable_triangulated},
 one deduces for any integer $p\geq 0$ an octahedron diagram\footnote{Apply the axiom
 to the composable maps $X \xrightarrow{\pi_p} X_p \xrightarrow{f_p} X_{p-1}$,
 use the right-hand square in (Fib4), and use the left-hand square to complete the diagram.}:
\begin{equation}\label{eq:octaedron}
\xymatrix@=10pt{
& X_{p}\ar^{f_p}[rr]\ar@{-->}[dd] && X_{p-1}\ar@{-->}[dd]\ar@{-->}[ld] & \\
X\ar^{\pi_p}[ru]\ar@{}|/2pt/{(*)}[r]
 & & F_p\ar@{-->}[ld]\ar[lu]\ar@{}|/3pt/{(*)}[u]\ar@{}|/3pt/{(*)}[d]\ar@{}|/2pt/{(!)}[l]
 & & X\ar@{}|/2pt/{(*)}[l]\ar[lu] \\
& G_p\ar[rr]\ar[lu] && G_{p-1}\ar[ru]\ar[lu] &
}
\end{equation}
where a symbol (*) refers to a fiber sequence, the dotted arrows indicate
 a boundary map of the form $\Omega A \rightarrow B$,
 the other parts of the diagram are commutative except for (!) which is anti-commutative.
Note in particular that $F_p$ (resp. $G_p$) is the fiber of $f_p$ (resp. $\pi_p$).

If one applies the homotopy functor $\Pi_*$ to this diagram, one gets the following diagrams
 for $p\leq q$:
\begin{equation}\label{eq:doucle-ec}
\xymatrix@C=4pt@R=14pt{
& \Pi_{q-p}(X_{p})\ar^{\alpha}[rr]\ar@{=>}_b[dd]
 && \Pi_{q-p}(X_{p-1})\ar@{=>}^b[dd]\ar@{=>}^{\beta}[ld] &&
 \Pi_{q-p}(X_{p})\ar^{\alpha}[rr]\ar@{=>}_b[dd]
 && \Pi_{q-p}(X_{p-1})\ar@{=>}^b[dd]
 \\
\Pi_{q-p}(X)\ar^a[ru]\ar@{}|/8pt/{(3*)}[r]
 & & \Pi_{q-p}(F_p)\ar@{=>}_{\bar \gamma}[ld]\ar^{\gamma}[lu]\ar@{}|/2pt/{(1*)}[u]\ar@{}|/2pt/{(2*)}[d]\ar@{}|/14pt/{(1!)}[l]\ar@{}|/14pt/{(2)}[r]
 & & \Pi_{q-p}(X)\ar_a[lu] &
 & \Pi_{q-p}(X)\ar@{}|/2pt/{(3)}[d]\ar@{}|/2pt/{(4)}[u]\ar^a[lu]\ar_a[ru] & \\
& \Pi_{q-p}(G_p)\ar_{\bar \alpha}[rr]\ar^c[lu]
 && \Pi_{q-p}(G_{p-1})\ar_c[ru]\ar_{\bar \beta}[lu]
 && \Pi_{q-p}(G_p)\ar_{\bar \alpha}[rr]\ar^c[ru]
 && \Pi_{q-p}(G_{p-1})\ar_c[lu]
}
\end{equation}
where we put $\alpha=f_{p*}$, $a=g_{p*}$.
 All diagrams $(?*)$ (resp. (?), (1!)) are homotopy long exact sequences
 (resp. commutative, anti-commutative) in the $1$-topos $\E$.
 There are two abuses of notation. First,  all double arrows are boundary maps,
 and therefore are of the form $\Pi_{q-p}(-) \rightarrow \Pi_{q-p-1}(-)$.
 Second, we have used the notation two times for the maps $a$, $b$, $c$.
 This abuse can be solved by making the degrees in the notation precise. 
 Finally, using the last two relations ((2) and (5)) of \Cref{prop:unstable_triangulated}(Fib4), one gets:
\begin{enumerate}
\item[(5)] $\gamma \circ \bar \beta=a \circ c$ 
\item[(6)] $b \circ \alpha=-\bar \alpha \circ b$.
\end{enumerate}

Consider the left commutative diagram above. Then the upper (resp. lower) triangle defines
 a right (resp. left) unstable exact couple of degree $+1$ in the sense of \Cref{df:unstable_ecpl}.
 Further, the above diagram is the unstable analog of a \emph{Rees system} in the terminology of Eilenberg and Moore
 (see \cite[3.1]{McC}).
 In particular, both exact couples induce the same spectral sequence, up to the sign of the differential,\footnote{One could have avoided
 the sign issue by changing $\bar \gamma$ for example. However, the sign of the differential does not change either its kernel
 or its cokernel, so that it will not interfere with our computations at any time.}
 by using the procedure of \Cref{num:unstable_ssp}
 and \Cref{df:unstable_ssp}:
\begin{equation}\label{eq:unst-ssp-tower}
E_1^{p,q}=\Pi_{q-p}(F_p) \Rightarrow \Pi_{q-p}(X).
\end{equation}
Indeed, one computes the two possible differentials on $E_1^{**}$ as follows:
\begin{equation}\label{eq:relation-differentials}
d_1:=\beta \circ \gamma\stackrel{(2)}=\bar \beta \circ b \circ \gamma\stackrel{(1!)}=-\bar \beta \circ \bar \gamma=:-d_1'
\end{equation}
We still put $D_1^{p,q}=\Pi_{q-p}(X_p)$
 (resp. $\bar D_1^{p,q}=\Pi_{q-p}(G_{p-1})$), so that $(E_1,D_1,\alpha,\beta,\gamma)$ is
 the exact couple of \Cref{ex:unst_exact_couples}(2).
\end{notation}

\begin{remark}
Beware that the symbol $\Rightarrow$ does not imply any convergence assertion.
 However, if the tower $X/X_\bullet$ is bounded, one gets that
 the preceding spectral sequence
 is strongly convergent in the sense of \cite[IX, 5.3]{BK72}.
\end{remark}

\begin{example}
The following example can help the reader understand our construction.
 We let $X$ be a pointed CW-complex (or a pointed simplicial sheaf), and $X_p$ be the $p$-th stage
 of the associated Postnikov tower. Then with the above notation,
 $F_p$ is the Eilenberg-MacLane space $K(\pi_p(X),p)$ and $G_p$ is the $p$-th stage of the Moore tower
 (in particular, $G_p$ is $p$-connected). The above unstable spectral sequence
 is well-known in topology (the so-called \emph{unstable Atiyah-Hirzebruch spectral sequence}).
\end{example}

\begin{notation}\label{num:co_augmentations_ssp}
Let us continue to introduce notation in the situation of the preceding paragraph.
 Fix an integer $q \geq 0$. We define a canonical map $\tau$ by the following commutative diagram:
\begin{equation}\label{eq:co-augment_hcpx}
\xymatrix@C=28pt@R=0pt{
& \Pi_q(X_0) & \\
\Pi_q(X)\ar|a[ru]\ar@{-->}|\tau[rr] & & \Pi_q(F_0).\ar|/2pt/\gamma_/-4pt/\sim[lu] \\
 & \Pi_q(G_{-1})\ar|/2pt/c^/-4pt/\sim[lu]\ar|{\bar \beta}[ru] &
}
\end{equation}
which follows from the relation (5).
 
We note finally that the $q$-truncated cohomotopical complex $E_1^{*,q}$
 admits $\tau$ as an augmentation (see \Cref{df:coaugmented_cohtp_cpx}).
 Going back to homotopical indexing, one gets the following co-augmented homotopical complex:
\begin{equation}
\label{eq:augmented_hcpx0}
* \rightarrow \Pi_q(X)\xrightarrow{\tau} \Pi_q(F_0)\xrightarrow{d_1^{0,q}} \Pi_{q-1}(F_1)
 \rightarrow \cdots \rightarrow \Pi_1(F_{q-1}) \xRightarrow{d_1^{q-1,q}} \Pi_0(F_q)
\end{equation}
Before stating our main result,
 we start with a lemma which shows that it is possible to directly compute all the terms
 at once on the diagonal $p=q$.
\end{notation}
\begin{lemma}\label{lm:abstract-ssp-diag}
Consider the above notation, and assume that the tower $X/X_\bullet$ is $d$-bounded for $d \in \N$.
 Then one gets:
$$
E_r^{q,q}=
\begin{cases}
\im\big(\Pi_0(X) \xrightarrow \tau \Pi_0(F_0)\big) & q=0 \wedge r>d \\
\gamma^{-1}\big(\im\big(\Pi_0(X) \xrightarrow a \Pi_0(X_q)\big)\big)/\Pi_1(X_{q-1}) & (0< q \leq d) \wedge r>\max(d-q,q-1) \\
* & (q>d) \wedge r\geq 1 \\
\end{cases}
$$
where, on the second line, we used the morphism of pointed $\E$-objects $\gamma:\Pi_0(F_q) \rightarrow \Pi_0(X_q)$,
 and the action of $\Pi_1(X_{q-1})$ on $\Pi_0(F_q)$.

In particular, the abutment filtration on $\Pi_0(X)$ is finite, with graded pieces given by
 the pointed $\E$-objects $E_\infty^{q,q}$ for $0\leq q \leq d$.
\end{lemma}
\begin{proof}
This follows from \Cref{rem:compute_E_r_BK} taking into account the $d$-boundedness condition.
\end{proof}

Our main technical result is the following degeneracy criterion:
\begin{theorem}\label{thm:degeneracy}
Consider the above notation, and assume the tower $X/X_\bullet$ is $d$-bounded for $d \in \N \cup \{+\infty\}$.
 Fix an integer $q \geq 0$, and let $I$ be the set of pairs of integers $(p,i)$ such that 
 $0 \leq p \leq \min(d,q)$, $i \in \{0,1\}$, except that $i \neq 1$ if $p=0$.
 Then the following conditions are equivalent.
\begin{enumerate}
\item[(i)] $\forall (p,i) \in I$,
 the pointed map $\bar \alpha:\Pi_{q-p}(G_{p-i}) \rightarrow \Pi_{q-p}(G_{p-i-1})$ is trivial.
\item[(i')] $\forall (p,i) \in I$,
 the pointed map $\bar \beta:\Pi_{q-p}(G_{p-i-1}) \rightarrow \Pi_{q-p}(F_{p-i})$ has trivial kernel,
 and is even a monomorphism if $q>0$.
\item[(ii)] $\forall (p,i) \in I$, the pointed maps $c:\Pi_{q-p}(G_{p-i}) \rightarrow \Pi_{q-p}(X)$
 and $b \circ \alpha:\Pi_{q-p+1}(X_{p-i}) \rightarrow \Pi_{q-p+1}(X_{p-i-1}) \rightarrow \Pi_{q-p}(G_{p-i-1})$ are trivial.
\item[(iii)] For the fixed integer $q$, the coaugmented homotopical complex \eqref{eq:augmented_hcpx0} is exact
 in the sense of \Cref{df:coaugmented_cohtp_cpx}.
\end{enumerate}
Note that point (iii) implies that:
$$
E_2^{p,q}=\begin{cases}
\Pi_q(X) & p=0 \wedge q>0, \\
* & 0<p<q.
\end{cases}
$$
\end{theorem}
In particular, when $d<+\infty$ and under the above equivalent conditions,
 the spectral sequence degenerates at $E^{**}_{d+1}$,
 one has $E_\infty^{0,q}=\pi_q(X)$ for $q>0$,
 and there is a finite number of non-necessarily trivial terms
 $E_\infty^{q,q}$ for $0\leq q \leq d$
 which describe the graded pieces of a finite filtration on $\Pi_0(X)$.
\begin{proof}
We heavily use notation from \Cref{df:two_exact_couples}. It can also be helpful for the reader
 to have the following ``unfolded'' diagram:
\begin{equation}
\begin{split}
\xymatrix@=10pt{
&&& \Pi_q(X_{-1}) & \Pi_q(X_{1})\ar^\alpha[ld] & \Pi_{q-1}(X_{0}) & \Pi_{q-1}(X_{2})\ar^\alpha[ld] & \\
& \Pi_{q+1}(X_{-1})\ar^\beta[rd] && \Pi_{q}(X_{0})\ar^\beta[rd]\ar^\alpha[u] && \Pi_{q-1}(X_{1})\ar^\beta[rd]\ar^\alpha[u] & \\
\Pi_q(X)\ar^\tau[rr] && \Pi_q(F_0)\ar@{==>}[rr]\ar^\gamma_\sim[ru]\ar_{\bar \gamma}[rd]
 && \Pi_{q-1}(F_1)\ar@{==>}[rr]\ar^\gamma[ru]\ar_{\bar \gamma}[rd] && \Pi_{q-2}(F_2) & \hdots \\
& \Pi_{q}(G_{-1})\ar_{\bar \beta}[ru]\ar^c_\sim[lu] && \Pi_{q-1}(G_{0})\ar_{\bar \beta}[ru]\ar_{\bar \alpha}[d]
 && \Pi_{q-2}(G_{1})\ar_{\bar \beta}[ru]\ar_{\bar \alpha}[d] & \\
&&& \Pi_{q-1}(G_{-1}) & \Pi_{q-1}(G_{1})\ar_{\bar \alpha}[lu] & \Pi_{q}(G_{0}) & \Pi_{q}(G_{2})\ar_{\bar \alpha}[lu] &
} \qquad \\
\qquad\qquad \xymatrix@=10pt{
& & \Pi_1(X_{q-2}) & \Pi_1(X_{q})\ar^\alpha[ld] & \Pi_{0}(X_{q-1}) & \Pi_{0}(X_{q+1})\ar^\alpha[ld] \\
& & \Pi_{1}(X_{q-1})\ar^\beta[rd]\ar^\alpha[u] && \Pi_{0}(X_{q})\ar^{\text{projection}}[rd]\ar^\alpha[u] & \\
\hdots & \Pi_1(F_{q-1})\ar@{==>}[rr]\ar^\gamma_\sim[ru]\ar_{\bar \gamma}[rd]
 && \Pi_0(F_q)\ar_\epsilon[rr]\ar^\gamma[ru] && \Pi_0(X_q)/\Pi_0(X_{q+1}) \\
& & \Pi_{0}(G_{q-1})\ar_{\bar \beta}[ru]\ar_{\bar \alpha}[d] &&&  \\
& & \Pi_{0}(G_{q-2}) & \Pi_{0}(G_{q})\ar_{\bar \alpha}[lu] & & &
}
\end{split}
\end{equation}
The dotted double arrows are the differentials of the unstable spectral sequence.
 Beware that the square part of this diagram is anti-commutative (see \eqref{eq:relation-differentials}).
 Finally, when one follows a path of two or more arrows in the diagram, one gets a part of
 a homotopy exact sequence.

With all these preparations, one can start our proof. 
The fact $(i')$ implies $(i)$ follows from $(2*)$. Reciprocally, $(i)$ implies that $\ker(\bar \beta)=*$,
 which implies the result for $q=0$. When $q>0$, we need to prove that $\bar \beta:\Pi_0(G_{p-1}) \rightarrow \Pi_0(F_p)$
 is a monomorphism, we apply property (T3) of \Cref{num:1-topos} as $\bar \beta$ is $\Pi_1(X_{p-1})$-equivariant
 according to (2), and the action of $\Pi_1(X_{p-1})$ on $\Pi_0(G_{p-1})$ is transitive
 because of $(2*)$ and the assumption (i).

The implication $(i) \Rightarrow (ii)$ follows by combining (3) and (6).
 To prove $(ii) \Rightarrow (i)$, we first remark that $(ii)$ implies that $b$ is an epimorphism
 according to $(3*)$. Thus we can conclude according to (6).

It remains to show that $(iii)$ is equivalent to all the other conditions.
 Let us first remark that, when $q=0$, $(iii)$ just means that $\tau$ has trivial kernel.
 This is equivalent to $(i')$ has $\bar \beta$ can be identified with $\tau$ thanks
 to the commutativity of \eqref{eq:co-augment_hcpx}.
 Thus, we assume $q>0$ for the remaining of the proof.

Let us first prove that all the other conditions imply $(iii)$.
 According to $(i')$, we know that $\bar \beta$ is a monomorphism.
 Thus the commutative diagram \eqref{eq:co-augment_hcpx} shows that $\tau$ is a monomorphism.
 According to $(2*)$, we know that $\im(\bar \beta)=\Ker(\bar \gamma)$.
 Applying again the fact that $\bar \beta$ is a monomorphism, one deduces that $\im(\tau)=\Ker(d_1^{0,q})$.
 Next, condition $(i)$ and $(2*)$ imply that $\bar \gamma$ is an epimorphism.
 Thus, one deduces $\im(d_1^{0,q})=\im(\bar \beta)\stackrel{(2*)}=\Ker(\bar \gamma)=\Ker(d_1^{1,q})$.
 Repeating this argument allows one to conclude by an obvious induction.
 
We finally prove that $(iii) \Rightarrow (i)$. 
 According to the commutativity of \eqref{eq:co-augment_hcpx},
 one deduces that $\bar \beta:\Pi_{q}(G_{-1}) \rightarrow \Pi_{q+1}(F_0)$ is a monomorphism.
 Thus $(2*)$ implies that $\bar \alpha:\Pi_*(G_{-1}) \rightarrow \Pi_*(G_0)$ is trivial.
 The latter triviality together with $(2*)$ implies that $\bar \gamma:\Pi_{*+1}(F_0) \rightarrow \Pi_{*}(G_0)$
 is an epimorphism. Now we apply the exactness of the unstable complex at $\Pi_q(F_0)$:
 one gets that $\Ker(\bar \gamma)=\im(\tau)=\Ker(d_1^{0,q})=\Ker(\bar \beta' \circ \bar \gamma)$.
 This implies that $\bar \beta':\Pi_*(G_0) \rightarrow \Pi_*(F_1)$ has trivial kernel when
 restricted to the image of $\bar \gamma$. So $\Ker(\bar \beta)=*$ and applying $(2*)$ again,
 one deduces that $\bar \alpha:\Pi_*(G_{0}) \rightarrow \Pi_*(G_1)$ is trivial.
 Iterating this reasoning, one concludes by induction on $p$.
\end{proof}

\begin{remark}
Point $(iii)$ in the case $q=0$ states that the canonical map $\tau:\Pi_0(X) \rightarrow \Pi_0(F_0)$ has trivial kernel.
 One should be careful that it is not a monomorphism in general.
 The defect of injectivity is measured by the filtration on $\Pi_0(X)$, and therefore by
 the non-triviality of the graded pieces $E_\infty^{q,q}$ when $q> 0$
 (computed in \Cref{lm:abstract-ssp-diag}).
\end{remark}

\begin{definition}\label{df:ssp-collapse}
When the equivalent conditions of the preceding theorem are satisfied,
 we will say that the spectral sequence \eqref{eq:unst-ssp-tower}
 essentially collapses at $E_2^{**}$ on the column $*=0$ and line $q$.
 If it collapses at all the lines, we just say that it collapses at $E_2^{**}$ on the column $*=0$.
\end{definition}

\begin{notation}\label{num:augmentation-ssp}
One can refine slightly the exactness part of the above \emph{collapsing} property.
 Consider again the assumption of the preceding theorem, that the spectral sequence $E_1^{**}$
 essentially collapses and that $d<+\infty$.

Recall from \Cref{num:unstable_ssp} that the coaugmented homotopical complex
 \eqref{eq:augmented_hcpx0} also admits an augmentation $\epsilon$:
\begin{equation}
\label{eq:augmented_hcpx2}
* \rightarrow \Pi_q(X)\xrightarrow{\tau} \Pi_q(F_0)\xrightarrow{d_1^{0,q}} \Pi_{q-1}(F_1)
 \rightarrow \cdots \rightarrow \Pi_1(F_{q-1}) \xRightarrow{d_1^{q-1,q}} \Pi_0(F_q)
 \xrightarrow{\epsilon} \Pi_0(X_q)/\Pi_0(X_{q+1}).
\end{equation}
where the target of $\epsilon$ is the cokernel of the map $\alpha$.
\end{notation}
\begin{proposition}\label{prop:q-large}
Consider the above notation and assumptions.
 Then for any $q \geq d-1$ the biaugmented homotopical complex \eqref{eq:augmented_hcpx2}
 is exact in the sense of \Cref{df:coaugmented_cohtp_cpx} (up to reindexing).

Moreover, for $q=1$ and $d \leq 2$, the corresponding biaugmented $1$-truncated homotopical complex 
 is strongly exact:
\begin{equation}
\label{eq:augmented_hcpx3}
* \rightarrow \Pi_1(X)\xrightarrow{\tau} \Pi_1(F_0)
 \xRightarrow{d_1^{0,1}} \Pi_0(F_1) \xrightarrow{\epsilon} \Pi_0(X_1)/\Pi_0(X_2).
\end{equation}
\end{proposition}
Left to the reader.

\begin{remark}
The exactness of \eqref{eq:augmented_hcpx3} can be surprising at first sight.
 Let us first mention that it appears crucially in Morel's study of the $\AA^1$-fundamental sheaf:
 see in particular \cite[Def. 2.20]{MorelLNM}.
 
For further links with Morel's theory, we refer the reader to \Cref{ex:htp-CM}
 and \Cref{rem:htp-CM}.
 We will also highlight examples in terms of torsors under algebraic groups for which the preceding proposition
 is crucial: see \Cref{cor:alg-gp-CM}. 
\end{remark}

If we drop the equivalence of the three conditions in \Cref{thm:degeneracy},
 we can give a refined version that will be useful in some of our examples
 (in particular in the relative case of Gersten conjecture).
\begin{proposition}\label{prop:truncated-degeneracy}
We consider the notation of \Cref{df:two_exact_couples}.
 We assume in addition that the tower $X/X_\bullet$ under $X$ is $d$-bounded
 for an integer $d>0$ and consider two integers $0 \leq e \leq d$,
 and $q \geq 0$. We assume the following condition:
\begin{enumerate}
\item[(i)] $\forall p \in  [e,d]$ such that $p\leq q$,
 both pointed maps
$$
\Pi_{q-p}(G_p) \xrightarrow{\bar \alpha^{p,q}} \Pi_{q-p}(G_{p-1})
 \xrightarrow{\bar \alpha^{p-1,q-1}} \Pi_{q-p}(G_{p-1})
$$
are trivial.
\end{enumerate}
Then the $(q-e)$-truncated cohomotopical complex
$$
\Pi_{q-e}(F_e)\xrightarrow{d_1^{e,q}} \Pi_{q-e-1}(F_{e+1})
 \rightarrow \cdots \rightarrow \Pi_1(F_{q-1}) \xRightarrow{d_1^{q-1,q}} \Pi_0(F_q)
$$
is exact in the sense of \Cref{df:coaugmented_cohtp_cpx}.
\end{proposition}
\begin{proof}
The proof works as the one of \Cref{thm:degeneracy}: condition (i) implies that,
 for $e \leq p \leq d$ and $p \leq q$,
 $\bar \gamma^{p-1,q}$ is an epimorphism and $\bar \beta^{p-1,q-1}$ is a monomorphism
 (beware that for $p=q$, one needs the argument at the beginning of the proof
 of \emph{loc. cit.} to deduce from $\Ker(\bar \beta^{p-1,p-1})=*$ to $\bar \beta^{p-1,p-1}$
 is a monomorphism). These two facts allow to conclude, as in \emph{loc. cit.}
\end{proof}

\section{Unstable coniveau spectral sequence}

\subsection{Unstable coniveau exact couple}

\begin{notation}\label{num:coniveau_ssp}
Let $X$ a scheme, with topology $t=\zar, \nis$.
 Let $\iSh(X_t)$ be the associated $\infty$-topos,
 and consider an additive cohomotopy functor (\Cref{df:htp_functor}) on $\iSh(X_t)$ with values in an arbitrary topos $\E$:
$$
\Pi_*:\iSh(X_t)_*^{op} \rightarrow \E_{\pi_*},
$$
as in \Cref{df:htp_functor}. We will also consider $\C=\pro-\iSh(X_t)$
 the $\infty$-category of pro-objects,
 and let $\Pi_*:\C_*^{op} \rightarrow \E_*$ be the homotopy functor
 obtained by Kan extension, according to the formula:
$$
\Pi_*(\pplim{i \in I} F_i)=\dlim_{i \in I} \Pi_*(F_i).
$$
The filtered colimit is taken in the category of $\pi_*$-structures in $\E$
 (see \Cref{rem:pi*-(co)limits}).
 The fact that this formula does provide a homotopy functor follows
 from \Cref{prop:colimits-htp-seq}.

Recall that a flag on $X$ is a decreasing sequence $Z^*=(Z^p)_{p \in \N}$ of closed subschemes of $X$
 such that $\codim_X(Z^p) \geq p$.\footnote{Note that $Z^p=\varnothing$ if $p>\dim(X)$.}
 We let $\flag(X)$ be the set of flags of $X$, ordered by inclusion. This is a filtered ordered set.
 For a non-negative integer $p$, we will consider the pro-object of $X_t$:
$$
X^{\leq p}:=\pplim{Z^* \in \flag(X)} (X-Z^{p+1}).
$$
The notation reflects the fact that the limit of this pro-object in the underlying category of sets
 is equal to the set of points of codimension less or equal to $p$ in $X$.

One deduces the so-called \emph{coniveau tower of $X$}, which is an \emph{increasing} tower\footnote{In other words,
 with the notation of \Cref{ex:unst_exact_couples}, this is a functor $X^{\leq \bullet}:\overrightarrow{\N} \rightarrow \pro-X_t/X$.}
 over the object $X$ in the category $\pro-X_t$ of pro-objects in $X_t$:
$$
\xymatrix@C=10pt@R=2pt{
\cdots & X^{\leq p}\ar[l]\ar_{\pi^p}[rd] && X^{\leq p-1}\ar_{f^p}[ll]\ar[ld] & \cdots\ar[l] \\
&& X &&
}
$$
If $X$ is of finite dimension $d$, this tower is bounded and $X^{\leq n}=X$ for $n\geq d$.

By considering the Yoneda embedding and adding a base point,
 we will view this as an increasing tower $X_+^{\leq *}$ over $X_+$ in the category $\C_*$.
 Equivalently, this is a tower over $X_+$ in $\C^{op}_*$ in the sense of \Cref{num:tower_under}
 (in particular, decreasing). One can therefore apply the constructions of \Cref{df:two_exact_couples}
 to this latter tower.

We do it explicitly for future references and for the comfort of the reader.
 One starts with a dual octahedron diagram in $\C_*$, dual to \eqref{eq:octaedron}:
\begin{equation}\label{eq:octaedron-coniv}
\xymatrix@=10pt{
& X^{\leq p}_+\ar_{\pi^p_+}[ld]\ar[rd] && X^{\leq p-1}_+\ar_{f^p_+}[ll]\ar[rd] & \\
X_+\ar@{}|/2pt/{(*)}[r]\ar[rd]
 & & X^{=p}\ar@{-->}[ru]\ar[rd]\ar@{}|/3pt/{(*)}[u]\ar@{}|/3pt/{(*)}[d]
 & & X_+\ar@{}|/2pt/{(*)}[l]\ar[ld] \\
& X^{>p}\ar@{-->}[uu]\ar@{-->}[ru] && X^{>p+1}\ar[ll]\ar@{-->}[uu] &
}
\end{equation}
where $X^{=p}$ and $X^{>p}$ are respectively the cokernel of $f^p_+$ and $\pi_{p+}$.
 Then one applies the contravariant functor $\Pi_*$ and gets a diagram for $p\leq q$
 satisfying the same properties as \eqref{eq:doucle-ec}:
\begin{equation}\label{eq:double-ec-coniv}
\xymatrix@C=10pt@R=14pt{
& \Pi_{q-p}(X^{\leq p}_+)\ar^{\alpha}[rr]\ar@{=>}^b[dd]
 && \Pi_{q-p}(X^{\leq p-1}_+)\ar@{=>}[dd]\ar@{=>}^{\beta}[ld] & \\
\Pi_{q-p}(X_+)\ar^a[ru]\ar@{}|/8pt/{(*)}[r]
 & & \Pi_{q-p}(X^{=p})\ar@{=>}_{\bar \gamma}[ld]\ar^{\gamma}[lu]\ar@{}|/2pt/{(*)}[u]\ar@{}|/2pt/{(*)}[d]
 & & \Pi_{q-p}(X_+)\ar@{}|/8pt/{(*)}[l]\ar[lu] \\
& \Pi_{q-p}(X^{>p})\ar_{\bar \alpha}[rr]\ar^c[lu]
 && \Pi_{q-p}(X^{>p-1})\ar[ru]\ar_{\bar \beta}[lu] &
}
\end{equation}
Therefore, according to \emph{loc. cit.}, one gets an unstable spectral sequence in $\E$:
\begin{equation}\label{eq:coniv_ssp}
E_{1,c}^{p,q}(X,\Pi_*):=\Pi_{q-p}(X^{=p}) \Rightarrow \Pi_{q-p}(X_+)
\end{equation}
which is associated with both exact couples
 $(E_{1,c}^{**},D_1^{p,q}=\Pi_{q-p}(X^{\leq p}_+),\alpha,\beta,\gamma)$ and $(E_{1,c}^{**},\bar D_1^{p,q}=\Pi_{q-p}(X^{>p-1}),\bar \alpha,\bar \beta,\bar \gamma)$
 respectively of type I and II.
\end{notation}
\begin{definition}\label{df:coniv-ssp-E-coef}
With the above notation,
 we will call \eqref{eq:coniv_ssp} the unstable coniveau spectral sequence of $X$ with coefficients in $\Pi_*$. 
\end{definition}
The spectral sequence is strongly convergent whenever the scheme $X$ is of finite dimension.

\begin{example}\label{ex:rep_cohtp_functor}
Consider a simplicial sheaf $\cX$ on $X_t$.
 It represents an additive cohomotopy functor with values in the punctual topos $\Set$:
$$
\Pi^\cX_*:=\Pi_*(-,\cX):\iSh(\Sm_k)_*^{op} \rightarrow \Set_{\pi_*}, \cY \mapsto \Pi_*(\cY,\cX):=\pi_*\big(\Map_*(\cY,\cX)\big).
$$
The unstable coniveau spectral sequence associated with the latter cohomotopy functor
 and any scheme $X$ has the form:
$$
E^{p,q}_{1,c}(X,\cX):=\pi_{q-p}\big(\Map(X^{=p},\cX)\big)
 =\dlim_{Z^* \in \flag(X)} [S^{q-p}\wedge (X-Z^{p+1}/X-Z^p),\cX]
 \Rightarrow [S^{q-p}\wedge X_+,\cX]
$$
where $[-,-]$ means homotopy classes of pointed $t$-sheaves on $\Sm_k$.\footnote{Beware
 that the above inductive limit must be taken in the appropriate category.}
\end{example}

\begin{notation}\label{num:cohtp_support}\textit{Cohomotopy theory with supports}.
To obtain the preceding spectral sequence, we can weaken the needed theory.

Let $\Sch$ be a category of schemes whose objects are stable under \'etale extensions.
 A closed $\Sch$-pair is a pair $(X,Z)$ such that $X$ is a scheme in $\Sch$
 and $Z$ is a closed subscheme of $X$.

A \emph{cohomotopy theory with supports} $\Pi_*$ on $\Sch$ and coefficients in $\E$\footnote{If
 the coefficients are not indicated, it is intended that $\E=\Set$.}
 is the data for any closed $\Sch$-pair $(X,Z)$ of a $\pi_*$-structure $\Pi_*(X,Z)$
 in $\E$ (\Cref{df:pi*-structure}), and of the following additional structures:
\begin{itemize}
\item \emph{contravariance}: for $f:Y \rightarrow X$ in $\Sch$,
 a pullback map $f^*:\Pi_n(X,Z) \mapsto \Pi_n(Y,f^{-1}(Z))$,
 morphism of groups for $n>0$, of pointed sets for $n=0$
\item \emph{covariance}: for a closed immersion $i:Z \rightarrow T$,
 $X$ in $\Sch$, there are
 $i_*:\Pi_n(X,Z) \mapsto \Pi_n(X,T)$ morphism of groups for $n>0$, of pointed sets for $n=0$ and

\item \emph{long exact sequence}: for a closed immersion $i:Z \rightarrow T$,
 with open complement $j:X-Z \rightarrow X$,
 there exists an exact long homotopy sequence (\Cref{df:long_htp_seq}):
\begin{equation}\label{eq:htp-long-ex}
\cdots
\Rightarrow\Pi_n(X,Z) \xrightarrow{i_*} \Pi_n(X,T) \xrightarrow{j^*} \Pi_n(X-Z,T-Z) 
 \xRightarrow{\partial_{X,Z}} \Pi_{n-1}(X,Z) \to \cdots
\end{equation}
\end{itemize}
Let again $t=\zar, \nis$. Then we say that $\Pi_*$ is $t$-local
 if the following properties hold:
\begin{itemize}
\item \emph{First additivity}: given closed $\Sch$-pairs $(X,Z)$ and $(Y,T)$,
 the canonical map: \\ $\Pi_*(X\sqcup Y,Z\sqcup T)\rightarrow \Pi_*(X,Z) \times \Pi_*(Y,T)$
 is an isomorphism.
\item \emph{t-excision}: for any closed $\Sch$-pair $(X,Z)$,
 any map $f:V \rightarrow X$ which is an open immersion in case $t=\zar$ and an \'etale map if $t=\nis$,
 and such that $T=f^{-1}(Z) \rightarrow Z$ is an isomorphism,
 the induced map $f^*: \Pi_n(X,Z) \rightarrow \Pi_n(V,T)$ is an isomorphism.\footnote{Recall that when $t=\nis$,
 we say that $f:(V,T) \rightarrow (X,Z)$ is \emph{excisive}.}
\item \emph{Second additivity}: Given a scheme $X$ and disjoint closed subschemes $Z$, $T$ of $X$,
 the canonical map $\Pi_*(X,Z \sqcup T) \rightarrow \Pi_*(X-T,Z) \times \Pi_*(X-Z,T) \simeq \Pi_*(X,Z) \times \Pi_*(X,T)$,
 where the second map is defined using Zariski excision, is an isomorphism.
\end{itemize}
The reader can check that, to define the preceding unstable coniveau spectral sequence,
 we only need the axioms of a cohomotopy theory with support.
 The fact that it is $t$-local will come into consideration solely for computations.
\end{notation}

\begin{remark}\label{rem:cohtp_support}
On the one hand, it is useful to have the preceding down-to-earth axiomatic.
 On the other hand, it is always satisfied in the abstract context of \Cref{num:coniveau_ssp}.
 Indeed, given an additive cohomotopy functor $\Pi_*:\iSh(X_t) \rightarrow \E_{\pi_*}$,
 one puts:
\begin{equation}\label{eq:cohtp_support}
\Pi_n(X,Z)=\Pi_n^Z(X):=\Pi_n(X/X-Z)
\end{equation}
where $X/X-Z$ is the (homotopy) cofiber of the open immersion $j:(X-Z) \rightarrow X$
 computed in $\iSh(X_t)$. This defines a cohomotopy theory with supports
 on $X_t$ with coefficients in $\E$, as follows from \Cref{prop:unstable_triangulated} (applied for cofiber sequences,
 see \Cref{rem:fibers&cofibers&triangulated}(1)).
 Besides it is $t$-local as defined above.\footnote{Use the isomorphism
 $X/(X-{Z \sqcup T}) \simeq X/(X-Z) \wedge X/(X-T)$ in $\iSh(X_t)$ for the second additivity.}
\end{remark}

\begin{notation}\label{num:restricted-product}
Let $I$ be a set. In any category $\C$ with finite products and filtered colimits,
 one can define the \emph{restricted product}
 of a family $(X_i)_{i \in I}$ of objects of $\C$ by the following formula:
\begin{equation}\label{eq:rest-product}
{\prod}'_{i \in I} X_i=\colim_{S \in \mathcal P^f(I)} \left(\prod_{i \in S} X_i\right)
\end{equation}
where the set $\@P^f(I)$ consists of $S\subset I$ such that $S$ is a finite set and $\@P^f(I)$ is ordered by inclusion. Note that in an additive category with filtered colimits,
 the restricted products coincide with coproducts.
 Note also that if we consider a family of pointed objects,
 then their restricted product is canonically pointed.

With this notation, we can state the following computation which
 gives the general shape of unstable coniveau spectral sequences.
\end{notation}
\begin{proposition}\label{prop:coniveau_ssp_refined}
Let $\Pi_*$ be a $t$-local cohomotopy theory with support with coefficients
 in an arbitrary topos $\E$, as defined previously.
 Then for any couple of integers $(p,q)$, there exists a canonical isomorphism:
$$
E_{1,c}^{p,q}(X,\Pi_*) \simeq 
\begin{cases}
\bigoplus_{x \in X^{(p)}} \Pi_{q-p}^x(X_{(x)}^t) & p \leq q-2, \\
\prod'_{x \in X^{(q-1)}} \Pi_1^x(X_{(x)}^t) & p=q-1, \\
\prod'_{x \in X^{(q)}} \Pi_0^x(X_{(x)}^t) & p=q,
\end{cases}
$$
where $\prod'$ denotes the restricted product (in the appropriate category), 
 $X^{(p)}$ is the set of codimension $p$ points,
 $X^t_{(x)}$ denotes the pro-scheme of $t$-neighborhoods of $x$ in $X$,
 and we have used notation \eqref{eq:cohtp_support}
 for $\Pi_*^x(X_{(x)}^t)$, cohomotopy  of $X_{(x)}^t$ with support in the closed point $x$.
\end{proposition}
\begin{proof}
Once the additivity property as been properly identified,
 the proof goes as in \cite[Lem. 1.14]{Deg11}
 taking into account formula \eqref{eq:rest-product}. 
\end{proof}

\begin{corollary}\label{cor:coniveau_ssp_refined}
Consider the assumptions of the preceding proposition.
 
Then for any integer $q \geq 0$,
 the $q$-truncated cohomotopical complex $E_{1,c}^{*,q}(X,\Pi_*)$
 (\Cref{df:coaugmented_cohtp_cpx})
 associated with the coniveau unstable spectral sequence \eqref{eq:coniv_ssp}
 has the following shape:
\begin{equation}\label{eq:coniveau_ssp_refined}
\bigoplus_{x \in X^{(0)}} \Pi_{q}(\kappa(x))
 \rightarrow \bigoplus_{x \in X^{(1)}} \Pi^x_{q}(X^t_{(x)})
 \rightarrow \cdots
 \rightarrow \underset{x \in X^{(q-1)}}{{\prod}'} \Pi_1^x(X^t_{(x)})
 \Rightarrow \underset{x \in X^{(q)}}{{\prod}'} \Pi_{0}^x(X^t_{(x)})
\end{equation}
where we have assumed $X$ is reduced to simplify the first term.
\end{corollary}

\begin{definition}\label{df:unst-Gersten-complex}
Consider the assumptions of the preceding corollary.
 The unstable $q$-truncated homotopical complex \eqref{eq:coniveau_ssp_refined} will be called
 the \emph{homotopical Gersten complex} of $X$ with coefficients in $\Pi_*$ and in degree $q$.
 We will use the notation:
$$
C^p(X,\Pi_q):=E_{1,c}^{p,q}(X,\Pi_*)=\underset{x \in X^{(q)}}{{\prod}'} \Pi_{q-p}^x(X^t_{(x)}).
$$
\end{definition}
Crucially, the homotopical Gersten complex is bi-augmented:
 according to \Cref{num:co_augmentations_ssp}, it
 admits an augmentation map
$$
\tau_q^X:\Pi_q(X) \rightarrow C^0(X,\Pi_q)=\underset{x \in X^{(0)}}{{\prod}'} \Pi_q^x(X^t_{(x)})
$$
and, according to \Cref{num:augmentation-ssp}, a co-augmentation map:
$$
\gamma:C^q(X,\Pi_q)=\underset{x \in X^{(q)}}{{\prod}'} \Pi_0^x(X^t_{(x)})
 \rightarrow \Pi_0(X_+^{\leq q}).
$$
It is usually more accurate to consider the following co-augmentation:
$$
\epsilon_q^X:C^q(X,\Pi_q)=\underset{x \in X^{(q)}}{{\prod}'} \Pi_0^x(X^t_{(x)})
 \xrightarrow{\gamma} \Pi_0(X_+^{\leq q}) \rightarrow \Pi_0(X_+^{\leq q})/\Pi_0(X_+^{\leq q+1}).
$$

\begin{notation}\label{num:functorial-coniv-ssp-E-coef}
It is well-known that the coniveau filtration is functorial with respect
 to flat pullbacks.
 In particular, given a flat morphism $f:Y \rightarrow X$ of schemes,
 as the pullback along $f$ respects codimension, one gets
 a morphism of ordered sets: $f^{-1}:\flag(X) \rightarrow \flag(Y), Z^* \mapsto T^*=Z^* \times_X Y$.
 Moreover, $f$ induces morphisms of closed pairs $(Y,T^p) \rightarrow (X,Z^p)$
 and therefore one deduces that the diagram \eqref{eq:octaedron-coniv} is natural
 with respect to $f$, resulting in a commutative diagram in $\iPSh(\Sch)$:
$$
\xymatrix@R=-4pt@C=10pt{
& Y^{\leq p}_+\ar_{}[ld]\ar[rd]\ar@{..>}[rddd] && Y^{\leq p-1}_+\ar_{}[ll]\ar[rd]\ar@{..>}[rddd] & \\
Y_+\ar@{}|/2pt/{(*)}[r]\ar[rd]\ar@{..>}[rddd]
 & & Y^{=p}\ar@{-->}[ru]\ar[rd]\ar@{}|/3pt/{(*)}[u]\ar@{}|/3pt/{(*)}[d]\ar@{..>}[rddd]
 & & Y_+\ar@{}|/2pt/{(*)}[l]\ar[ld]\ar@{..>}[rddd] \\
& Y^{>p}\ar@{-->}[uu]\ar@{-->}[ru]\ar@{..>}[rddd] && Y^{>p+1}\ar[ll]\ar@{-->}[uu]\ar@{..>}[rddd] & \\
&& X^{\leq p}_+\ar_{}[ld]\ar[rd] && X^{\leq p-1}_+\ar_{}[ll]\ar[rd] & \\
&X_+\ar@{}|/2pt/{(*)}[r]\ar[rd]
 & & X^{=p}\ar@{-->}[ru]\ar[rd]\ar@{}|/3pt/{(*)}[u]\ar@{}|/3pt/{(*)}[d]
 & & X_+\ar@{}|/2pt/{(*)}[l]\ar[ld] \\
&& X^{>p}\ar@{-->}[uu]\ar@{-->}[ru] && X^{>p+1}\ar[ll]\ar@{-->}[uu] &
}
$$
Let us now come back to the situation of a small site $X_t$ for one of the topologies:
 $t=\zar, \nis$. If $f:V \rightarrow W$ is a morphism of schemes in $X_t$,
 it is in particular flat (in fact \'etale). Thus the above diagram can be taken
 in the $\infty$-topos $\iSh(X_t)$.

In particular, given a homotopy functor $\Pi_*:\Sh(X_t) \rightarrow \E_{\pi_*}$,
 one deduces that diagram \eqref{eq:double-ec-coniv} and therefore the left and right unstable
 coniveau exact couples are functorial in $X_t$. Consequently,
 the same holds for the unstable coniveau spectral sequence \eqref{eq:coniv_ssp}
 and for the biaugmented unstable Gersten complex \eqref{df:unst-Gersten-complex}.
 As a result, one gets pullback maps:
\begin{equation}\label{eq:Gersten-functorial}
f^*:C^*(U,\Pi_*) \rightarrow C^*(V,\Pi_*)
\end{equation}
in the category of (biaugmented) unstable complexes with coefficients in $\E$.
\end{notation}

\subsection{Unstable Gersten complexes}

\begin{notation}\label{notn:unstable-gersten}
Let again $t=\zar, \nis$, and $X$ be a scheme.
The classical framework of $t$-local cohomology theory with supports can be extended
 to the non-abelian setting (see \cite[\textsection IV]{Hart66} for Zariski sheaves of abelian groups,
 and \cite[\textsection 4.3]{DFJ22} in the Nisnevich case).
 Given a pointed $t$-sheaf $F$ of sets, in $\Sh(X_t)_*$,
 and $Z \subset X$ a closed subscheme, one puts:
 $$\Gamma_Z(X,F)=\{\rho \in F(X) \mid \forall x \notin Z, \rho_x=* \}.$$
 Similarly, $\uG_Z(F)$ is the pointed $t$-sheaf
 $V \mapsto \uG_{Z \times_X V}(F|_V)$.
 If $T \subset Z$, let $\uG_{Z/T}(F)$ be the cokernel,
 in the category of pointed $t$-sheaves, of the natural pointed map
 $\uG_T(F) \rightarrow \uG_Z(F)$.

We next use the notation of the coniveau tower laid down in \Cref{num:coniveau_ssp},
 and define:
\begin{align*}
\uG_{X^{\geq p}}(F)=\colim_{Z^* \in \flag(X)} \uG_{Z^p}(F), \quad 
\uG_{X^{=p}}(F)=\colim_{Z^* \in \flag(X)} \uG_{Z^p/Z^{p+1}}(F).
\end{align*}
The following lemma is the non-abelian version (and $t$-local) version
 of \cite[Prop. 2.1]{Hart66}, based on the notion of restricted product
 \eqref{eq:rest-product}.
\end{notation}
\begin{lemma}
Let $F$ be a pointed $t$-sheaf of sets on a scheme $X$ and consider the above notation.
 For any integer $p \geq0$,
 the following conditions are equivalent:
\begin{enumerate}[label=\emph{(\roman*)}]
\item There exists an isomorphism $F \simeq {\prod}'_{x \in X^{(p)}} {i_x}_*(M_x)$
 where for any point $x \in X^{(p)}$, $M_x$ is a pointed set,
 and $i_x:\{x\} \rightarrow X$ is the canonical immersion.
\item The canonical pointed maps $F \leftarrow \uG_{X^{\geq p}}(F) \rightarrow \uG_{X^{=p}}(F)$
 are isomorphisms.
\end{enumerate}
Moreover, when these conditions hold, the isomorphism of (i)
 induces an isomorphism $M_x \simeq F_x$.
\end{lemma}
\begin{proof}
The lemma easily follows once we get the computation, for any open $U \subset X$:
$$
\uG_{X^{=p}}(F)(U)\simeq {\prod}'_{x \in U^{(p)}} F_x
$$
which follows as in the proof of \Cref{prop:coniveau_ssp_refined}.
\end{proof}

We can now consider a non-abelian and $t$-local version of Cousin complexes (see \cite[Def. p. 241]{Hart66}).
\begin{definition}\label{df:Cousin-cpx}
Consider the above assumption and notation.
Let $X$ be a scheme and $p \geq 0$ be an integer.
 A pointed $t$-sheaf $F$ on $X$ is said to be \emph{supported in $X^{(p)}$}
 if the equivalent conditions of the preceding proposition hold.\footnote{This terminology seems
 preferable to ``lies on the $p$-th skeleton'' in \cite[Def. p. 231]{Hart66}.}

Let $C^*$ be a $q$-truncated cohomotopical complex with coefficients
 in the topos $\Sh(X_t)$ (\Cref{num:cohomotopical_notation}).
 We say that $C^*$ is a \emph{$q$-truncated $t$-local unstable Cousin complex},
 if for any integer $0 \leq p \leq q$,
 the sheaf $C^p$ is supported in $X^{(p)}$ in the above sense.
 If $q\geq \dim(X)$, we simply say, $C^*$ is a \emph{unstable $t$-local Cousin complex}.
\end{definition}
In fact, if $q \geq \dim(X)+2$, a $q$-truncated cohomotopical complex $C^*$
 with coefficients in the topos $\Sh(X_t)$ is simply a complex of abelian
 $t$-sheaves on $X$. In this case, $C^*$ is an unstable $\zar$-local (resp. $\nis$-local)
 Cousin complex over $X$ if and only if it is a Cousin complex over $X$ in the sense of Hartshorne
 \cite[Def. p. 241]{Hart66} (resp. \cite[Def. 4.3.7]{DFJ22}).

\begin{notation}\label{num:t-local-coniv-ssp}
We consider a scheme $X$ and one of the topologies $t=\zar, \nis$.
 We let $\Pi_*$ be a $t$-local cohomotopy theory with supports defined on $X_t$
 (for example, a cohomotopy functor $\Pi_*:\iSh(X_t)^{op} \rightarrow \Set$).

 We have associated to $\Pi_*$ respectively in \Cref{eq:octaedron-coniv},
 \Cref{eq:coniv_ssp}, and \Cref{df:unst-Gersten-complex},
 two unstable coniveau exact couples, a coniveau unstable spectral sequence,
 and a Gersten complex, with respect to $X$, and even to any scheme $V$ in $X_t$.
 We know from \Cref{num:functorial-coniv-ssp-E-coef}
 that they are in fact functorial in the scheme $V$ in $X_t$.
 The associated $t$-sheaf functor being exact, it preserves homotopy long exact sequences
 and (co)homotopical biaugmented complexes.
 This implies that we can $t$-sheafify all the constructions of \eqref{num:coniveau_ssp}
 meaning we obtain unstable exact couples with coefficients in the $1$-topos $\Sh(X_t)$:
\begin{equation}\label{eq:double-ec-coniv-sheaf}
\xymatrix@C=10pt@R=4pt{
& \uPi^{X^{\leq p}}_{q-p}\ar^{\alpha}[rr]\ar@{=>}^b[dd]
 && \uPi^{X^{\leq p-1}}_{q-p}\ar@{=>}[dd]\ar@{=>}^{\beta}[ld] & \\
\uPi^X_{q-p}\ar^a[ru]\ar@{}|/8pt/{(*)}[r]
 & & \uPi^{X^{=p}}_{q-p}\ar@{=>}_{\bar \gamma}[ld]\ar^{\gamma}[lu]\ar@{}|/2pt/{(*)}[u]\ar@{}|/2pt/{(*)}[d]
 & & \uPi^X_{q-p}\ar@{}|/8pt/{(*)}[l]\ar[lu] \\
& \uPi^{X^{>p}}_{q-p}\ar_{\bar \alpha}[rr]\ar^c[lu]
 && \uPi^{X^{>p-1}}_{q-p}\ar[ru]\ar_{\bar \beta}[lu] &
}
\end{equation}
where $\uPi_*=\uPi^X_*$ (resp. $\uPi^{X^{\leq p}}_*$, $\uPi^{X^{>p}}_*$, $\uPi^{X^{=p}}_*$) is the $t$-sheaf, on the small site $X_t$,
 associated with the presheaf $V \mapsto \Pi_*(V)$
 (resp. $\Pi_*(V^{\leq p})$, $\Pi_*(V^{>p})$, $\Pi_*(V^{=p})$),
 seen as a $\Sh(X_t)_*$-structure (in the sense of \Cref{df:pi*-structure}).
 Therefore, one deduces as in \emph{loc. cit.} an unstable spectral sequence with coefficients
 in $\Sh(X_t)$:
\begin{equation}\label{eq:coniv_ssp-sheaf}
\uE_{1,c}^{p,q}(X_t,\Pi_*):=\uPi^{X^{=p}}_{q-p} \Rightarrow \uPi^X_{q-p}
\end{equation}
which is part of both a left and a right unstable exact couple.
 The left one is given by the formula: $\underline D^{p,q}_{c}(X_t,\Pi_*)=\uPi^{X^{\leq p}}_{q-p}$.
\end{notation}
\begin{definition}\label{df:t-local-coniv-ssp}
The unstable spectral sequence \eqref{eq:coniv_ssp-sheaf}
 will be called the \emph{$t$-local unstable coniveau spectral sequence} associated with
 the cohomotopy theory with support $\Pi_*$.

For any integer $q \geq 0$, the cohomotopical complex $\Ge^*(X_t,\Pi_*,q):=\uE_{1,c}^{*,q}(X_t,\Pi_*)$
 corresponding to the $q$-th line of the preceding spectral sequence will be called the \emph{$t$-local (unstable) Gersten complex}
 in degree $q$ associated with $\Pi_*$ on the site $X_t$. When $t$ is clear, we simply say \emph{Gersten complex}.

When $\Pi_*=\Pi_*^\cX$ is represented by a simplicial $t$-sheaf $\cX$ as in \Cref{ex:rep_cohtp_functor},
 we will write: $\Ge^*_t(\cX,q)=\Ge^*(X_t,\Pi^\cX_*,q)$.
\end{definition}
The Gersten complex $\Ge^*(X_t,\Pi_*,q)$ associated with a cohomotopy theory $\Pi_*$ on the small site $X_t$
 is a cohomotopical complex in the topos $\Sh(X_t)$, which, according to \Cref{cor:coniveau_ssp_refined}
 can be computed by the following formula:
\begin{equation}
\Ge^p(X_,\Pi_*,q)={\prod}'_{x \in X^{(p)}} {i_x}_*\Big(\Pi_{q-p}^x(X_{(x)})\Big)
\end{equation}
where the restricted product is taken in the category of $t$-sheaves on $X$ respectively
 of pointed sets if $q-p=0$, groups if $q-p=1$ and abelian groups if $q-p>1$
 (in which case, it is simply a direct sum).
 Moreover,  it admits a canonical augmentation (as described in \Cref{num:co_augmentations_ssp}):
$$
\tau:\uPi_q \xrightarrow{\ c^{-1}\ } \uPi_q^{X^{>-1}} \xrightarrow{\ \beta\ } \uPi_q^{X=0}=\Ge^0(X_t,\Pi_*,q).
$$

\begin{remark}
We have restricted the above definition to homotopical functors $\Pi_*$ with values in $\Set$,
 but everything will work with coefficients in an arbitrary topos. We will not use this generality here,
 but the case of the classifying topos $BG$ associated with a discrete group $G$, or the pro-\'etale topos
 associated to a field, are interesting examples.
\end{remark}

\begin{lemma}\label{lm:ssp=Cousin}
Consider the notation of the above definition.
 Then the Gersten complex $\Ge^*(X_t,\Pi_*,q)$ associated with $\Pi_*$
 is a $q$-truncated $t$-local Cousin complex in the sense of \Cref{df:Cousin-cpx}.
\end{lemma}
This is immediate according to \Cref{cor:coniveau_ssp_refined}.
 
\begin{definition}\label{df:Gersten-ppty}
Consider the setting of preceding definition.
Given an integer $q \geq 0$, we will say that $\Pi_*$ is \emph{Gersten
 in degree $q$} on $X_t$ if the $q$-truncated augmented homotopical complex
 of sheaves on $X_t$
$$
\uPi_q \rightarrow \Ge^*(X_t,\Pi_*,q)
$$ 
is exact. We simply say that $\Pi_*$ is \emph{Gersten on $X_t$} 
 if it is so in all degrees $q\geq 0$.

If $\Pi_*=\Pi_*^\cX$ is represented by a simplicial sheaf $\cX$ on $X_t$
 (\Cref{ex:rep_cohtp_functor}),
 we simply say that $\cX$ is \emph{Gersten on $X_t$} (resp. \emph{and in degree $q$})
 if $\Pi_*^\cX$ is so.
\end{definition}

The remaining of the paper will give several situations where this property holds.
 According to our main result \Cref{thm:degeneracy}, we deduce the following characterization of
 the above (unstable) Gersten property.
\begin{proposition}\label{prop:carac-Gersten}
Let $t=\zar, \nis$, let $X$ be a scheme,
 and $\Pi_*$ be a $t$-local cohomotopy with supports defined on $X_t$.
 Then the following conditions are equivalent:
\begin{enumerate}
\item[(i)] $\Pi_*$ is Gersten in degree $q$ on $X_t$.
\item[(ii)] For any $x \in X$, the unstable coniveau spectral sequence
$$
E_{1,c}^{p,q}(X^t_{(x)},\Pi_*) \Rightarrow \Pi_{q-p}(X_{(x)}^t)
$$
with coefficients in the $t$-local cohomotopy theory with supports $\Pi_*$
 restricted to the small $t$-site of $X_{(x)}^t$,
 collapses on the column $j=0$ and the line $q$ (see \Cref{df:ssp-collapse}).
\item[(iii)] for any point $x \in X$, $\sX=X_{(x)}^t$,
 any integer $0 \leq p \leq q$,
 any $n=p, p+1$, any closed subscheme $Z \subset \sX$ such that $\codim_\sX(Z)>n$,
 and for any element $\alpha \in \Pi_{q-p}(\sX,Z)$,
 there exists a closed subscheme $T$ of $\sX$ such that $Z \stackrel i \subset T$, $\codim_\sX(T)\geq n$, 
 and $i_*(\alpha)=*$ in $\Pi_{q-p}(\sX,T)$.
\end{enumerate}
\end{proposition}
\begin{proof}
The equivalence of (i) and (ii) is obvious. 
 Then the equivalence between (ii) and (iii) follows\footnote{Use the following dictionary:
 $X_p$, $F_p$, $G_p$ replaced respectively by $X^{\leq p}$, $X^{=p}$, $X^{>p}$.} from \Cref{thm:degeneracy}
 as property (iii) above is a simple translation of property (i) in \emph{op. cit.}
\end{proof}

\begin{remark}
Property (iii) is a (slightly more precise) version of
 the \emph{effaceability condition} first highlighted by Bloch and Ogus
 in order to axiomatize the proof of Quillen (see \cite{BO}).
 The interest of the preceding proposition is that condition (iii) is not only sufficient
 but also necessary.
\end{remark}

\subsection{The abelian case and Eilenberg-MacLane sheaves}

\begin{notation}
In this section, we want to compare the unstable coniveau spectral sequence 
 with the more classical framework of the (stable) coniveau spectral sequence.
 We consider again the conventions of the preceding sections. 
 $X$ is an arbitrary scheme, with topology $t=\zar, \nis$.

Let $C$ be a complex of abelian sheaves on $X_t$.
 In \cite[Def. p. 277]{Hart66}, Hartshorne defines the Cousin complex
 associated with $C$, when $C$ is bounded and $t=\zar$.
 This construction has been, in an obvious way, extended to the case
 where $C$ is unbounded and $t=\zar,\nis$ in \cite[Def. 4.3.7, Th. 4.3.8]{DFJ22},
 by taking as a function $\delta=-\codim_X$ on points of $X$.\footnote{Beware
 that this is not in general a dimension function. Nevertheless, for the definitions
 and results of \emph{loc. cit.}, Section 4.3, one does not use this assumption.
 In fact, the only necessary condition is that if $x$ is a specialization of $y$,
 then $\delta(x)\leq \delta(y)$. In the latter case, we say that $\delta$
 is a weak dimension function.}

In fact, we can easily recall the definitions of \emph{loc. cit.}
 by using the exact method of \Cref{num:coniveau_ssp},
 up to changing the conventions on indexation.
 One first uses the cohomological functor on the stable $\infty$-category $\iDer(X_t,\Z)$
 with coefficients in the abelian category $\Sh(X_t,\Z)$:
$$
\uH^C:\iDer(X_t,\Z)^{\op} \rightarrow \Sh(X_t,\Z), D \mapsto \uH^0_t \uHom(D,C).
$$
As in \emph{loc. cit.}, it can be Kan-extended (by taking colimits)
 to pro-objects of $\iDer(X_t,\Z)$.
 Then one defines a cohomological exact couple (see \cite[1.1.1]{Deg11} for our conventions)
 in the abelian category $\Sh(X_t,\Z)$ 
\begin{align*}
^{st}\underline E_{1,c}^{p,q}(X_t,C)&=\uH^{p+q}_t\uHom(\Z(X^{=p}),C),  \\
^{st}\underline D_{1,c}^{p,q}(X_t,C)&=\uH^{p+q}_t\uHom(\Z(X^{\leq p-1}),C).
\end{align*}
This gives the $t$-local (stable) coniveau spectral sequence:
$$
^{st}\underline E_{1,c}^{p,q}(X_t,C) \Rightarrow \uH^{p+q}_t(C)
$$
The next definitions were given in \cite[Section 4.3]{DFJ22},
 again following \cite[Def. p. 247]{Hart66}.
\end{notation}
\begin{definition}
Consider the above notation.
 One defines the \emph{$t$-local Gersten complex} of the complex $C$
 as the following complex of abelian sheaves on $X_t$:
$$
\Ge^*(X,C,0):= {^{st}\underline E}_{1,c}^{*,0}(X_t,C).
$$
We say that $C$ is \emph{Cohen-Macaulay} on $X_t$ if the preceding $t$-local coniveau spectral sequence
 is concentrated on the line $q=0$.

When $C=F[0]$ is concentrated in degree $0$,
 we will put $\Cz_t^*(F)=\Ge^*_t(F[0])$. This is the \emph{$t$-local Cousin complex} associated with the sheaf $F$
 on $X_t$ (with respect to the codimension filtration of $X$).
\end{definition}
Note that the complex $\Ge^*(X_t,C,0)$ is in fact a (non-truncated) $t$-local Cousin complex made of abelian components
 in the sense of \Cref{df:Cousin-cpx}.

\begin{remark}
When $t=\zar$ (resp. $t=\nis$),
 and $C=F[0]$ is a single sheaf placed in degree $0$,
 the above definition coincides with that of \cite[Def. p. 247]{Hart66}
 (resp. \cite[Def. 4.3.6]{DFJ22}).
 Recall moreover that $F$ is Cohen-Macaulay on $X_t$ if and only if the canonical augmentation map
$$
F \rightarrow \Cz^*_t(F)
$$
is a quasi-isomorphism of $t$-sheaves.
\end{remark}

The preceding property has several remarkable consequences.
\begin{proposition}
Let $C$ be a Cohen-Macaulay complex of abelian sheaves on $X_t$.
 
Then the $t$-local coniveau spectral sequence converges. The edge morphisms of this spectral sequence
 induce for any $p\geq 0$ an isomorphism $\uH^p(C) \simeq \uH^p(\Ge^*(X,C,0))$ of abelian sheaves on $X_t$.
 Morevoer, the edge morphisms of the usual coniveau spectral sequence ${}^{st}E_{1,c}^{**}(X,C)$
 associated with the cohomology theory with coefficients in $C$ induces for any integer $p \in \Z$ an isomorphism:
$$
H^p\big(\Gamma(X,\Ge^*(X,C,0))\big)={}^{st}E_{2,c}^{p,0}(X,C) \simeq H^p(X_t,C).
$$
Finally, all these isomorphisms can be lifted to a quasi-isomorphism of complexes
 $C \rightarrow \Ge^*(X,C,0)$.
\end{proposition}
The first three assertions are consequences of the definitions.
 The last one is actually the generalized version of a theorem of Hartshorne-Suominen
 extended to the Nisnevich (and the Zariski) topology in \cite[Th. 4.3.8]{DFJ22}.

\begin{notation}
Consider the above notation, but assume in addition that $C$ is concentrated in non-negative degree.
 In that case,
 one defines the Eilenberg-Mac Lane simplicial $t$-sheaf $K(C)$ on $X_t$, which is an object of $\iSh(X_t)$,
 using the (toposic) Dold-Kan correspondence.
 The main case comes from an abelian sheaf $F$ on $X_t$
 and an integer $n \geq 0$; one puts: $K(F,n)=K(F[n])$.

The next proposition sheds light on the fact that unstable spectral sequences are \emph{fringed}.
\end{notation} 
\begin{proposition}
Consider the above notation.
 Then for any integers $p \leq q$, there exists a canonical isomorphism:
$$
\uE_{1,c}^{p,q}(X,K(C)) \simeq {}^{st}\underline{E}_{1,c}^{p,-q}(X,C).
$$
In fact, for any $q \geq 0$, there exists an isomorphism of cohomotopical complexes:
$$
\uE_{1,c}^{*,q}(X,K(C)) \simeq \tau^{\leq q}_{nv}\left({}^{st}\underline{E}_{1,c}^{*,-q}(X,C)\right)
$$
where $\tau^{\leq q}_{nv}$ denotes the naive truncation functor in cohomological degree less or equal to $q$.
\end{proposition}
The proof is now just a matter of checking definitions as, from the Dold-Kan correspondence,
 one gets: $\pi_n\Map(X,K_t(C)) \simeq H^{-n}_t(X,C)$ and this isomorphism is compatible
 with the long exact sequences of cohomotopy (resp. cohomology) with supports.

\begin{corollary}
Consider the above notation. Let $F$ be an abelian sheaf on $X_t$, and $n \geq 0$ be an integer.

Then for any integer $q\geq 0$, there exists an isomorphism of augmented (by $F$) cohomotopical complexes
$$
\Ge^*(X_t,K(F,n),q) \simeq \tau^{\leq q}_{nv}\left(\Cz^*_t(X,F[n-q])\right).
$$
In particular, if $X$ is finite dimensional and $q\geq \dim(X)$,
 one gets an isomorphism of augmented cohomotopical complexes
$$
\Ge^*(X_t,K(F,n),q) \simeq \Cz^*_t(F[n-q]).
$$
\end{corollary}

\begin{corollary}
Consider the above notation and the following conditions:
\begin{enumerate}[label=\arabic*]
\item $F$ is Cohen-Macaulay on $X_t$.
\item $K(F,n)$ is unstably Cohen-Macaulay in degree $n$ on $X_t$.
\end{enumerate}
Then (i) implies (ii). Moreover in that case, the unstable coniveau spectral
 sequence $E_{1,c}^{**}(X,K(F,n))$ is concentrated on the line $q=n$.

Assume that $n \geq \dim(X)$. Then (ii) implies (i),
 the cohomotopical complex $\Ge^*(X_t,K(F,n),n)$ is a complex of abelian sheaves
 and a cohomological resolution of $F$.
 In fact, there exists a canonical isomorphism of cohomological resolutions of $F$:
$$
\Ge^*(X_t,K(F,n),n) \simeq \Cz^*_t(X,F).
$$
\end{corollary}

The next proposition and its corollary are the abstract version
 for the extension of Bloch-Ogus theory in the unstable setting,
 taken into account \Cref{prop:carac-Gersten}.
 It also relates the \emph{Gersten setting} with the more classical
 Cousin setting as formalized by Grothendieck (see again \cite{Hart66}).
\begin{proposition}\label{prop:truncated-Cousin}
Let $X$ be a scheme, $t=\zar, \nis$ and $\Pi_*$ be a $t$-local cohomotopy theory with supports.
 Assume that $\Pi_*$ is $t$-locally Gersten in degree $q>1$.

Then $\Ge^*(X_t,\Pi_*,q)$ is a $q$-truncated $t$-Cousin complex augmented by 
 the abelian sheaf $\uPi_q$ on $X_t$.
 There exists a unique isomorphism of complex augmented by $\uPi_q$:
$$
\tau_{nv}^{\leq q}\Ge^*(X_t,\Pi_*,q) \simeq \tau_{nv}^{\leq q}\Cz_t^*(\uPi_q)
$$
where the right hand-side is the $t$-Cousin complex associated with the abelian sheaf $\uPi_q$.
 In particular, $\Ge^*(X_t,\Pi_*,q)$ is in fact a complex of abelian groups,
 and for any integer $0 \leq p \leq q$, and $x \in X^{(p)}$, there exists a canonical
 isomorphism of abelian groups for $p<q-1$, groups for $p=q-1$ and pointed sets for $q=p$:
$$
\Pi_{q-p}^x(X_{(x)}^t) \simeq  H^p_x(X_{(x)}^t,\uPi_q).
$$
Finally, the preceding isomorphism induces an isomorphism for all $0 \leq p <q$:
$$
E_{2,c}^{p,q}(X,\Pi_*) \simeq H^p\Gamma\big(X,\Cz_t^*(\uPi_q)\big)
 \simeq H^p(X_t,\uPi_q).
$$
\end{proposition}
\begin{proof}
The first statement is \Cref{lm:ssp=Cousin}.
 The second statement follows from the uniqueness of $q$-truncated cohomotopical Cousin complexes
 (see Appendix: \Cref{thm:unstable-unique-cousin}).
 The isomorphism then follows, as both complexes are $t$-flasque.
\end{proof}

In particular, if $q>\dim X$,
 we deduce from the proposition that $\uPi_q$ is Cohen-Macaulay.
 Using the finer study of the unstable spectral sequence, we can slightly improve that result.
\begin{proposition}
Consider the assumptions of the previous proposition and assume in addition that $q \geq \dim(X)$.
 Then the abelian sheaf $\uPi_q$ over $X_t$ (obtained by sheafification of $\Pi_q$) is Cohen-Macaulay,
 there exists a unique \emph{isomorphism} of resolutions of $\uPi_q$
$$
\Ge^*(X_t,\Pi_*,q) \simeq \Cz^*_t(\uPi_q)
$$
and for all $0 \leq p \leq q$,
$$
E_{2,c}^{p,q}(X,\pi_*) \simeq H^p(X_t,\uPi_q).
$$
\end{proposition}
\begin{proof}
As explained previously, the case $q>\dim(X)$ follows from the preceding proposition.
 For the case $q=\dim(X)$, we need to prove the exactness of the Gersten complex coaugmented
 by the constant sheaf $*$. This follows from \Cref{prop:q-large}.
\end{proof}

\subsection{The non-abelian case and Classifying spaces of sheaf of groups}

\begin{notation}
Based on our notion of homotopy complex, we can extend the classical definition
 of Cousin resolutions and Cohen-Macaulay sheaves to the non-abelian context.

As in the preceding sections, we let $X$ be a scheme equipped with one of the topologies $t=\zar, \nis$.
 Let $\cG$ be a sheaf groups on $X_t$, and let $B\cG$ be its classifying space (see e.g. \cite[\textsection 4.1]{MV}),
 as an object of $\iSh(X_t)$. Recall that:
$$
\pi_n(\Map(X,B\cG))=[S^n \wedge X_+,B\cG]_*\simeq \begin{cases}
H^1(X_t,\cG) & n=0, \\
\cG(X) & n=1 \\
* & n>1
\end{cases}
$$
where $H^1(X_t,\cG)$ denotes the pointed set of $t$-local $\cG$-torsors on $X$,
 and $[-,-]_*$ denotes the pointed homotopy classes in $\iSh(X_t)$.
 More generally, for any closed subset $Z \subset X$, $p=0,1$, we put:
$$
H^p_Z(X_t,\cG)=[S^{1-p} \wedge X/X-Z,B\cG]_*
$$
where $X/X-Z$ is pointed in the obvious way. According to this definition, we get a homotopy exact sequence
 in the classical sense (or ``long homotopy sequence'' in the final topos according to \Cref{df:long_htp_seq}):
$$
* \rightarrow H^0_Z(X_t,\cG) \rightarrow \cG(X) \xrightarrow{j^*} \cG(X-Z) \xRightarrow \partial
 H^1_Z(X_t,\cG) \rightarrow H^1(X_t,\cG) \xrightarrow{j^*} H^1\big((X-Z)_t,\cG\big).
$$
In particular, the boundary map corresponds to an action of $\cG(X-Z)$ on the (pointed) set $H^1_Z(X_t,\cG)$.
 In fact, $H^*(-_t,\cG)$ is the $t$-local cohomotopy theory with support in the sense of \Cref{num:cohtp_support}
 represented by $B\cG$.

We can therefore apply the previous considerations to this theory.
 For once, by taking colimits (see \Cref{prop:colimits-htp-seq}), this definition extends to any $t$-localization of $X$:
 given a point $x \in X$, we get in particular a homotopy exact sequence:
\begin{equation}\label{eq:local-cohtp}
* \rightarrow H^0_x(X_{(x)}^t,\cG) \rightarrow \cG(X_{(x)}^t) \xrightarrow{j^*} \cG(X_{(x)}^t-\{x\}) \xRightarrow \partial
 H^1_x\big(X_{(x)}^t,\cG\big) \rightarrow *
\end{equation}
so that $H^0_x(X_{(x)}^t,\cG)$ is the kernel of the morphism of groups $j^*$ and 
 $H^1_x(X_{(x)}^t,\cG)$ is the homogeneous set $\cG(X_{(x)}^t-\{x\})/\cG(X_{(x)}^t)$.
 For readability of the notation, we will put:
$$
H^p_x(X_{(x)}^t,\cG)=H^p_x(X_t,\cG).
$$

We also get the unstable coniveau spectral sequence in \Cref{df:coniv-ssp-E-coef},
 as well as its sheafified version as in \Cref{df:t-local-coniv-ssp}. Motivated by the abelian case,
 we adopt the following terminology.
\end{notation}
\begin{definition}
Consider the preceding notation. We define the homotopy Cousin complex of the sheaf $\cG$ over $X_t$
 as
$$
\Cz_t^*(\cG)=\uE_{1,c}^{*,1}(X_t,\Pi_*).
$$
It is a $2$-term homotopy complex, concentrated in degree $0$ and $1$, co-augmented by $\cG$:
$$
\xymatrix@C=30pt@R=4pt{
\cG\ar^-\tau[r] & {\prod}'_{\eta \in X^{(0)}} \eta_*(\cG(\kappa_\eta))\ar@{=>}[r]\ar@{=}[d]
 & {\prod}'_{x \in X^{(1)}} x_* H^1_x(X_t,\cG)\ar@{=}[d] \\
 & \Cz_t^1(\cG)\ar@{=>}[r] & \Cz_t^0(\cG).
}
$$
See \Cref{num:restricted-product} for the notation.
\end{definition}
Based on the theory developed so far, we get the following conditions,
 non-abelian analogous of the characterization of Cohen-Macaulay abelian sheaves:
\begin{proposition}\label{prop:htpy-CM}
Let $X$ be a scheme and $\cG$ be a sheaf of groups on $X_t$, $t=\zar, \nis$.
 The following conditions are equivalent:
\begin{enumerate}[label=\emph{(\roman*)}]
\item The co-augmented homotopical complex $\cG \rightarrow \Cz_t^*(\cG)$ is exact (in the sense of \Cref{df:coaugmented_cohtp_cpx}).
\item For any open (resp. étale if $t=\nis$) $V/X$, one gets an isomorphism:
$$
\cG(V) \simeq \left\{ g \in {\prod}'_{\eta \in V^{(0)}} \cG(\kappa_\eta) \mid 
 \forall x\in V^{(1)},
 g \text{ acts trivially on } H^1_x(V_t,\cG)\right\}
$$
induced by the canonical (restriction) map $\tau$ defined above.
\item For any open (resp. étale if $t=\nis$) $V/X$, any closed subset $Z \subset V$, the restriction map
$$
j^*:\cG(V) \rightarrow \cG(V-Z)
$$
is an isomorphism if $\codim_Z(V)>1$ and a monomorphism if $\codim_Z(V)=1$.
\item For any point $x \in V$, and any $i=0,1$, 
$$H^i_x\big(X,\cG\big)=*$$ if $i \neq \codim_X(x)$.
\item The simplicial sheaf $B\cG$ is Gersten on $X_t$ in the sense of \Cref{df:Gersten-ppty}.
\end{enumerate}
Moreover, when these conditions hold, the obvious map
$$
H^1(X_t,\cG) \rightarrow \left({\prod}'_{x \in X^{(1)}}  H^1_x(X_t,\cG)\right)/\left({\prod}'_{\eta \in X^{(0)}} \cG(\kappa_\eta)\right)
$$
is an injection of pointed sets. Finally, if $\dim(X)\leq 1$, the sequence of pointed sheaves
$$
* \rightarrow \cG \xrightarrow \tau \Cz_t^0(\cG) \xRightarrow d \Cz_t^1(\cG) \rightarrow *
$$
is exact and the above map is a bijection.
\end{proposition}
\begin{proof}
It is clear that (i) implies (ii), and (ii) implies (iii).
 One obtains that (iii) implies (iv) by applying colimits and using the homotopy exact sequence \eqref{eq:local-cohtp}.
 Condition (iv) implies that the $t$-local coniveau spectral sequence $\uE^{**}_{1,c}(X_t,B\cG)$ is concentrated
 on line $q=1$. Therefore it degenerates at $E_2$ and in fact, it collapses on the column $*=0$
 in the sense of \Cref{df:ssp-collapse}. In other words, (iv) implies (v).
 Finally, (v) obviously implies (i).

To get the remaining assertions, we consider the homotopy sequence, exact by (i):
$$
* \rightarrow \cG \xrightarrow \tau \Cz_t^0(\cG) \xRightarrow d \Cz_t^1(\cG)
$$
We let $F \subset \Cz_t^1(\cG)$ be the pointed $t$-sheaf which is the image of $d$, so that $F=\Cz_t^0(\cG)/\cG$ by exactness of the preceding sequence.
 Therefore one can apply \cite[Chapitre III, Prop. 3.2.2]{Gir} to the monomorphism $\tau$ to get a homotopy exact sequence:
$$
* \rightarrow \cG(X) \rightarrow \Gamma(X,\Cz_t^0(\cG)) \Rightarrow \Gamma(X,F) \rightarrow H^1(X_t,\cG) \rightarrow H^1(X_t,\Cz_t^0(\cG)).
$$
 Note that $H^1(X_t,\Cz_t^0(\cG))=*$, since $\Cz_t^0(\cG)$ is a $t$-flasque sheaf of groups
 (flasque when $t=\zar$, satisfying the Brown-Gersten property when $t=\nis$).
 Finally, if $\dim(X)=1$, we get the stated exactness by appealing to \Cref{prop:q-large}.
\end{proof}

\begin{remark}
If we use the \'etale topology, all the above conditions are again equivalent, but the last assertion does not
 necessarily hold.
\end{remark}

\begin{definition}\label{df:CM-non-ab}
When a sheaf of groups $\cG$ on $X_t$ satisfies the equivalent conditions above,
 we will say that $\cG$ is \emph{homotopy Cohen-Macaulay} on $X_t$.
\end{definition}

\begin{example}\label{ex:htp-CM}
\begin{enumerate}
\item An abelian sheaf which is Cohen-Macaulay is obviously a homotopy Cohen-Macaulay as sheaf of groups.
\item It follows from \cite[Th. 6.1]{MorelLNM} that over a perfect field $k$,
 for any pointed simplicial sheaf $\cX$ on the Nisenvich site $\Sm_k$,
 the sheaf of groups $\pi_1^{\AA^1}(\cX)$ is homotopy Cohen-Macaulay in the above sense
 over $X_t$ for any smooth $k$-scheme $X$. This fact was our inspiration to this section.
 We will give a direct proof in the next section.
\end{enumerate}
\end{example}

\begin{remark}\label{rem:htp-CM}
When $\cG$ is a sheaf of groups over the Nisnevich site $\Sm_k$,
 given a smooth $k$-scheme $X$,
 our augmented Cousin complex $\cG \rightarrow \Cz_{\nis}(\cG)$ on $X_\nis$ is precisely the restriction
 of the complex $1 \rightarrow \cG \rightarrow \cG^{(0)} \Rightarrow \cG^{(1)}$ restricted to $X_{\nis}$
 as considered by Morel in \cite[\textsection 2.2]{MorelLNM}.
\end{remark}

Thanks to a result of \v Cesnavi\v cius and Scholze, and according to our previous proposition,
 one deduces the following result.
\begin{theorem}\label{thm:alg-gp-CM}
Let $X$ be a noetherian scheme which is regular in codimension less than $3$.
 Let $G$ be a separated algebraic group scheme over $X$, which we assume to be affine if $\dim(X)>1$.

Then $G$ is a homotopy Cohen-Macaulay sheaf on $X_t$ for $t=\zar, \nis$ in the above sense.
\end{theorem}
\begin{proof}
In fact, one deduces property (iii) of \Cref{prop:htpy-CM} by applying \cite[Lem. 7.2.7]{CS21}
 to $Y=G$, $X=V$ and $Z=Z$ --- under the notation of (iii).
\end{proof}

\begin{corollary}\label{cor:alg-gp-CM}
Let $X$ be a connected Dedekind scheme with function field $K$,
 and $G$ be a separated $X$-group scheme.
 Then $G$ is homotopy Cohen-Macaulay on $X_t$ for $t=\zar, \nis$ and the canonical pointed map
$$
H^1(X_t,G) \rightarrow \left({\prod}'_{x \in X^{(1)}}  H^1_x(X_t,G)\right)/G(K)
$$
is a bijection.
\end{corollary}

We can compare this with~\cite[Theorem 0.1, Theorem 0.3]{Gro17}.
\begin{corollary}\label{cor:dedekind}
Let $X$ be a connected Dedekind scheme with function field $K$,
 and $G$ be a separated $X$-group scheme. For $t=\zar, \nis$, the map 
$$
H^1(X_t,\cG) \simeq G(\#O_{X_t})\backslash G(\#A_{X_t})/G(K)
$$
is a bijection, where $$G(\#O_{X_t}):= {\prod}'_{x \in X^{(1)}} G(X^t_{(x)})$$ and $$G(\#A_{X_t}):={\prod}'_{x \in X^{(1)}} G(X^t_{(x)}-\{x\}).$$
\end{corollary}

\begin{proof} As noted in \eqref{eq:local-cohtp}, we have $$ H^1_x(X_t,G)\simeq G(X^t_{(x)}) \backslash G(X^t_{(x)}-\{x\})$$ as the quotient set. The claim then follows. \end{proof}


For the rest of this section, we consider the case of torsors associated to groups schemes on a regular integral scheme of dimension 2 in \'etale topology. For groups schemes which are either a smooth reductive group scheme or a finite type of multiplicative type, we make the following observations. 

\begin{theorem}\label{dim2} For regular integral scheme $X$ of dimension 2 and $G$ either a smooth reductive $X$-group scheme or a finite type $X$-group scheme of multiplicative type, there is an exact sequence 
$$*\to {\prod}'_{x \in X^{(1)}}  H^1_x(X_{\et},G)/G(K)\xrightarrow{\gamma} H^1(X_{\et}, G)\xrightarrow{\theta} \im\bigg(H^1(X_{\et}, G)\to H^1((\Spec K)_{\et}, G)\bigg)\to * $$
of pointed sets.
\end{theorem}
To give a proof of this we analyze the coniveau exact couple as established in~\Cref{num:coniveau_ssp}.
We have part of the exact couple
 $$\xymatrix{
&				&				& 				& 	\Pi_0(X^{=2}) \ar[d] 	\\
&				&				& 				& 	 \Pi_0(X^{\leq 2}) \ar[d]  \\
1\ar[r] & \Pi_1(X^{\leq 1})\ar[r] & \Pi_1(X^{\leq 0}) \ar@{=>}[r] & \Pi_0(X^{=1}) \ar[r]  & \Pi_0(X^{\leq 1}) \ar[r] & \Pi_0(X^{\leq 0}) \\
}$$
Applying this via~\Cref{ex:rep_cohtp_functor} to the cohomotopy theory with supports associated to the space $\@X=B_{\et}G:=\&\alpha_*(BG)$ for an \'etale-sheaf of groups $G$ where 
$$\alpha^*: \iSh((X)\to  \iSh((X): \&\alpha_*$$ is the adjoint pair induced by the obvious morphims of sites. 
 $$\xymatrix@C=2em{
				&				& 				& 	\dlim_{\_Z}[X/X-Z^2, B_{\et}G] \ar[d]^{\-{\gamma}=*} 	\\
				&				& 				& 	H^1(X_{\et}, G)=[X, B_{\et}G] \ar[d]^-{\-\alpha} \ar[rd]^{\theta} \\
\ar[r] & G(K) \ar[r] &  {\prod}'_{x \in X^{(1)}}  H^1_x(X_{\et},G) \ar[r]^-{\gamma}  &  {\begin{array}{@{}c@{}}\dlim_{\_Z}[X-Z^2, B_{\et}G]\\ =\dlim_{\_Z}H^1((X-Z^2)_{\et}, G)\end{array}} \ar[r]^-{\alpha} &{\begin{array}{@{}c@{}}\dlim_{\_Z}[X-Z^1, B_{\et}G]\\ =H^1((\Spec K)_{\et}, G) \end{array}} \\
}$$
Here $[-, -]$ denotes the hom-set in the homotopy category associated with the $\infty$-category of Nisnevich sheaves on $X_{\nis}$.
We note that the sequences (horizontal long and vertical) are exact. As noted by \cite[paragraph after Question 6.2]{CTS79}, the map $\-\alpha$ is injective for regular integral scheme $X$ and $G$ an affine flat finite type $X$-group scheme (hence $\-\gamma=*$). 

Clearly, for each $x \in X^{(1)}$, the map $\dlim_{\_Z}H^1((X-Z^2)_{\et}, G)\to H^1((\Spec K)_{\et}, G)$ factors through
$$H^1((\Spec \@O_{X, x})_{\et}, G) \to H^1((\Spec K)_{\et}, G).$$ 
By \cite[proof of Proposition 6.8]{CTS79} we get that 
$$\im\left(\dlim_{\_Z}H^1((X-Z^2)_{\et}, G)\xrightarrow{\alpha} H^1((\Spec K)_{\et}, G) \right)=\cap_{x \in X^{(1)}}\im\bigg( H^1((\Spec \@O_{X, x})_{\et}, G)\to H^1((\Spec K)_{\et}, G) \bigg).$$
In particular, codimension 1 purity holds for $G$ on $X$ iff $\im(\alpha)=\im(\theta)$.

\begin{proof}[Proof of~\Cref{dim2}]
In the above discussion, let us assume in addition that $X$ is a regular integral scheme of dimension 2. Then by \cite[Theorem 6.13]{CTS79}, in the case $G$ is a smooth reductive $X$-group scheme and by~\cite[Corollaire 6.9]{CTS79} in the case $G$ is a finite type $X$-group scheme of multiplicative type, the map
$$\-\alpha: H^1(X_{\et}, G)\to \dlim_{\_Z}H^1((X-Z^2)_{\et}, G)$$ is bijective.
By~\Cref{lm:abstract-ssp-diag}, we have $$E_2^{1,1}=\gamma^{-1}\big(\im\big(\Pi_0(X) \xrightarrow a \Pi_0(X_1)\big)\big)/\Pi_1(X_{0}) = {\prod}'_{x \in X^{(1)}}  H^1_x(X_{\et},G)/G(K).$$
Note that $E_2^{1,1}=E_{\infty}^{1,1}$. 
Similarly, we can compute $$E_{\infty}^{0,0}=E_{2}^{0,0}=\im(\alpha)=\im(\theta).$$
By the convergence of the spectral sequence, we have an exact sequence of pointed sets
$$*\to E_{\infty}^{1,1}\to H^1(X_{\et},G)\to E_{\infty}^{0,0}\to *$$
Hence the result follows. 
\end{proof}

\begin{corollary}\label{cor:dim2} For regular integral scheme $X$ of dimension 2, and  $G$ a smooth reductive $X$-group scheme such that $G$ is a special group (\ie Zariski cohomology and \'etale cohomology agree), there is a bijection
$$ {\prod}'_{x \in X^{(1)}}  H^1_x(X_{\et},G)/G(K)\xrightarrow{\gamma} H^1(X_{\et}, G).$$
In particular, this holds for $G=GL_n$ ($n\geq 1$). 
\end{corollary}
\begin{proof}
By \Cref{dim2} it follows since $H^1((\Spec K)_{\et}, G)=H^1((\Spec K)_{\zar}, G)=*$.
\end{proof}
\begin{remark}
We note that \cite[Corollary B.7]{CRR20} has shown a similar result in view of the comparison result~\cite[Remarque 3.6.5(5)]{Gir}.
\end{remark}

\section{Examples of unstable Gersten resolutions}

\subsection{An unstable Bloch-Ogus-Gabber theorem}

\begin{notation}\label{num:htp-functor&closed-pairs}
Throughout this section, we fix a base scheme $S$. Let $\Sm_S$ be the category of smooth $S$-schemes
 and a $\nis$-local cohomotopy theory with supports $\Pi_*$ on $\Sm_S$ with coefficients in an arbitrary topos $\E$
 in the sense of \Cref{num:cohtp_support}. By taking colimits as usual, we can extend this
 theory to the category $\eSm_S$ of essentially smooth $S$-schemes.
 
We can therefore apply the constructions of the previous section to any essentially smooth $S$-scheme $X$,
 by restriction to $X_\nis$, and get the (resp. Zariski or Nisnevich-local) unstable coniveau spectral sequence
 of $X$ with coefficients in $\Pi_*$ of \Cref{df:coniv-ssp-E-coef}
 (resp. \Cref{df:t-local-coniv-ssp}).

As $\Pi_*$ is defined over the smooth site, we can use the classical methods
 to get the Zariski-local Gersten property (and more).
 To this end, we introduce the following key property on $\Pi_*$.
\end{notation}
\begin{definition}\label{def:gabber}
We will say that the cohomotopy theory $\Pi_*$ satisfies the \emph{Gabber property},
 if for any closed pair $(X,Z)$ such that $X$ is a smooth affine $S$-scheme and any $q \geq 0$,
 the following diagram is commutative:
$$
\xymatrix@R=-4pt@C=35pt{
 \Pi_q(\PP^1_X,\PP^1_Z)\ar[rr]^{j^*}\ar[rd]_{s_{\infty}^*} && \Pi_q(\AA^1_X,\AA^1_Z) \\
 &  \Pi_q(X,Z)\ar[ru]_{\pi^*} &
}
$$
where $j$ (resp. $s_\infty$, $\pi$) is the obvious open immersion
 (resp. section at $\infty$, canonical projection).
\end{definition}
Note that using the homotopy long exact sequence \eqref{eq:htp-long-ex},
 one deduces from the above that for any closed subscheme $Z \subset X$,
 it suffices to check the above property in the case $Z=\varnothing$.

\begin{example}\label{ex:rep}
If $\Pi_*$ is $\A^1$-invariant, one easily checks that it automatically satisfies the Gabber property.
 A key example will be the case of the homotopy functor $\Pi_*=\Pi_*^\cX$ as in \Cref{ex:rep_cohtp_functor}
 for $\cX$ an $\A^1$-local sheaf.
\end{example}

The following lemma follows classical lines in the unstable setting (see \cite{CHK}).
\begin{lemma}\label{keylemma}
Assume $\Pi_*$ satisfies the Gabber property.

Let $V$ be an affine smooth $S$-scheme and let $i:Z \hookrightarrow \AA^1_V$ be a closed subscheme
 such that the composite map $p:Z \xrightarrow i \AA^1_V \xrightarrow{\pi} V$ is finite.
 Let $F:=p(Z)_{red}$, so that one has a closed immersion $k:Z \rightarrow \AA^1_F$ which factors $i$.

Then for any $q \geq 0$, the following pointed map is trivial:
$$k_*:\Pi_q(\A^1_V,Z)\to \Pi_q(\A^1_V,\A^1_{F}).$$
\end{lemma}
\begin{proof}
We embed all the schemes in $\PP^1_V$ via the open immersion $j:\AA^1_V \subset \PP^1_V$.
As $s_{\infty}(V)\cap Z=\emptyset$, there is a factorization 
 $s_{\infty} : V\xrightarrow{s'} (\PP^1_V-Z)\xrightarrow l \PP^1_V.$
According to the Gabber property, one deduces a commutative diagram:
$$\xymatrix@R=4pt@C=24pt{ 
\Pi_q(\PP^1_V,Z)\ar[r]^{\tilde k_*} \ar[dd]^{\simeq}_{excision} & \Pi_q(\PP^1_V,\PP^1_{F}) \ar[rd]_/2pt/{s_{\infty}^*} \ar[dd]^{j^*} \ar[rr]^{l^*}
 && \Pi_q(\PP^1_V-Z,\PP^1_F-Z)\ar[dl]^{s'^*} \\
 && \Pi_q(V,F)  \ar[ld]^-{\pi^*} & \\
\Pi_q(\AA^1_V,Z)\ar[r]^{k_*} & \Pi_q(\AA^1_V,\A^1_{F}) &  
}$$
According to the localization long exact sequence \eqref{eq:htp-long-ex},
 one deduces that $l^*\circ \tilde k_*=*$ and this concludes.
\end{proof}

The strategy of \cite{CHK} to exploit the previous lemma works fine
 over a (not-necessarily perfect) field.
\begin{lemma}(Effaceability)\label{eff}
Assume $\Pi_*$ satisfies the Gabber property, and $S=\Spec(k)$ be the spectrum of an arbitrary field.

Let $X$ be a smooth affine $k$-scheme and $F \subset X$ be a finite subset.
 If $k$ is finite, we assume that $F=\{x\}$ is a single point.

Then, for any closed subscheme $Z\subset X$ of codimension $p>0$,
 there exists an open neighborhood $U$ of $F$ in $X$, a closed subscheme $T\subset U$, of codimension $p-1$, containing $Z_U := Z\cap U$,
 $i:Z_U \rightarrow T$,
 and such that for any integer $q\geq0$, the following pointed map is trivial:
$$\Pi_q(U,Z_U)\xrightarrow{\ i_*\ } \Pi_q(U,T).$$
\end{lemma}
\begin{proof}
According to \cite[3.1.1, 3.1.2]{CHK} if $k$ is infinite,
 or \cite{HK20} if $k$ is finite,
 there exists a Zariski open neighbourhood $U$ of $S$ in $X$,
 an excisive morphism $f:(U,Z_U) \rightarrow (\AA^1_V,W)$
 for some closed pair $(\AA^1_V,W)$ such that $V$ is a smooth $k$-scheme $V$ and 
  the composite map $W \subset \A^1_V \xrightarrow{\pi} V$ is finite
 (take $f=\varphi|_U$ in the notation of \cite[3.1.1]{CHK}).
 Consider the closed subscheme $F:=\pi(W)_{red} \subset V$ and $T= f^{-1}(\AA^1_F)$.
 Note that there is an inclusion $i:Z_U \rightarrow T$.
 Hence there is a commutative diagram of pointed sets 
$$
\xymatrix@R=10pt@C=20pt{
 \Pi_n(\A^1_V,T) \ar[r] \ar_{f^*}[d]^{\simeq} &  \Pi_n(\A^1_V,\A^1_F) \ar[d] \\
  \Pi_n(U,Z_U) \ar^{i_*}[r] &  \Pi_n(U,T)
}
$$
and excision for $f$, as well as \Cref{keylemma} allows us to conclude.
\end{proof}

The lemma gives exactly condition (i) of \Cref{thm:degeneracy} for the unstable coniveau
 exact couple of the semi-localization $X_{(F)}$ with coefficients in $\Pi_*$.
 In other words, wa have obtained.
\begin{proposition}\label{prop:Gabber&Gersten}
Assume $\Pi_*$ is a cohomotopy theory with supports  defined on $\Sm_k$,
 where $k$ is an arbitrary field, which satisfies the Gabber property.

Then for any essentially smooth semi-local (resp. local if $k$ is finite) $k$-scheme $X$,
 the unstable coniveau exact sequence $E_{1,c}^{**}(X,\Pi_*)$
 of $X$ with coefficients in $\Pi_*$ (as in \Cref{df:coniv-ssp-E-coef}),
 collapses on the column $*=0$ in the sense of \Cref{df:ssp-collapse}.

In other words, for any $q \geq 0$, the bi-augmented coniveau unstable complex:
\begin{align*}
* \rightarrow &\Pi_q(X)\xrightarrow{\tau} \bigoplus_{\eta \in X^{(0)}} \Pi_q\big(\kappa(\eta)\big)
 \xrightarrow{d_1^{0,q}} \bigoplus_{x \in X^{(1)}} \Pi_{q-1}^x\big(X_{(x)}^h\big) \rightarrow \cdots  \\
& \rightarrow \bigoplus_{x \in X^{(q-1)}} \Pi_1^x\big(X_{(x)}^h\big)
 \xRightarrow{d_1^{q-1,q}} \bigoplus_{s \in X^{(q)}} \Pi_0^s\big(X_{(s)}^h\big)
 \xrightarrow{\gamma} \Pi_0(X^{\leq q}_+)
\end{align*}
is exact in the sense of \Cref{unstable-cx}.
Moreover, for $q=dim(X)+1$ (resp. $q \leq \dim(X)$), one can replace the augmentation $\gamma$ 
 by the augmentation (see \Cref{num:co_augmentations_ssp} for the first one):
$$
\bigoplus_{s \in X^{(q)}} \Pi_0^s\big(X_{(s)}^h\big) \xrightarrow{\epsilon} \Pi_0(X^{\leq q}_+)/\Pi_0(X^{\leq q+1}_+)
 \quad \text{resp. } \bigoplus_{s \in X^{(q)}} \Pi_0^s\big(X_{(s)}^h\big) \xrightarrow{ct} *
$$
\end{proposition}

The local case immediately implies that $\Pi_*$ is both Zariski and Nisnevich
 locally Gersten on any essentially smooth $k$-scheme $X$ (see \Cref{df:Gersten-ppty}).
 Taking into account the results of the previous section, one further gets:
\begin{corollary}\label{cor:unstable-BOG}
Consider the above assumptions, and let $q>1$ be an integer.

Let $\uPi_q$ be the Zariski sheaf on $\Sm_k$ associated with $\Pi_q$. Then the following results hold:
\begin{enumerate}
\item For any smooth $k$-scheme $X$, $\uPi_q$ is Cohen-Macaulay up to degree $q$ on $X_\zar$,
 and it is fully Cohen-Macaulay if $\dim(X)\leq q$.
\item $\uPi_q$ is a Nisnevich sheaf on $\Sm_k$.
 For any $0\leq p<q$, and any smooth $k$-scheme $X$,
 there exists isomorphisms which are natural with respect to flat pullbacks:
$$
E_{2,c}^{p,q}(X,\Pi_*) \simeq H^p_\zar(X,\uPi_q) \simeq H^p_\nis(X,\uPi_q).
$$
\item For $t=\zar, \nis$, and any smooth $k$-scheme $X$,
 there exists a unique isomorphism of complex of sheaves on $X_t$:
$$
\Ge^*(X_t,\Pi_*,q) \simeq \tau_{nv}^{\leq q} \Cz_t^*(\uPi_q^X)
$$
between the $t$-local Gersten complex in degree $q$ of $\Pi_*$ and the naively truncated
 $t$-local Cousin complex on the $t$-sheaf $\uPi_q$ restricted to the small site $X_t$.
 This isomorphism extends to the site of smooth $k$-schemes with morphisms as the flat (equivalently syntomic)
 ones.

In particular, for any point $x \in X^{(p)}$, one deduces isomorphisms:
$$
\Pi_{q-p}^x(X_{(x)}) \simeq \Pi_{q-p}^x(X^h_{(x)}) \simeq H^{p}_x(X_\zar,\uPi_q) \simeq H^{p}_x(X_\nis,\uPi_q).
$$
\end{enumerate}
\end{corollary}

\begin{remark}\label{morel-CM}
We can apply the above general formalism in the particular context as discussed in \Cref{ex:rep}, to the cohomotopy theory with supports $\Pi^\cX_*$ with values in the punctual topos $\Set$ (as in~\Cref{ex:rep_cohtp_functor}) associated to an $\#A^1$-local space  $\cX$. In particular, $\Pi^\cX_*$ satisfies the Gabber property and \Cref{cor:unstable-BOG} applies to the $\#A^1$-homotopy sheaves $\piA_q(\@X)=\_{\Pi}^\cX_q$. More explicitly, it shows that over any field, certain $\AA^1$-homotopy sheaves are Cohen-Macaulay in the sense of \cite{Hart66}, up to some truncation. We note that over a perfect field $k$, it follows from \cite[Theorem 6.1]{MorelLNM} that for any pointed simplicial sheaf $\cX$ on the Nisnevich site $\Sm_k$, the sheaf of groups $\pi_1^{\AA^1}(\cX)$ is homotopy Cohen-Macaulay in the above sense over $X_t$ (for $t=\zar, \nis$) for any smooth $k$-scheme $X$. Our result extends this result to any field $k$ and  
also extends recent works of \cite{DFJ22} and \cite{DKO} on Cousin resolutions to the unstable setting.

\end{remark}

Finally, given the current technology, one can improve the previous results
 and get a weaker Gersten property over positive dimensional bases.
\begin{lemma}(Effaceability)\label{eff2}
Assume $\Pi_*$ satisfies the Gabber property over $\Sm_S$,
 and let $e=\dim(S)$.

Let $X$ be a smooth affine $S$-scheme, $x \in X$ be a point,
 and $Z\subset X$ be a closed subscheme of codimension $p>e$.
 Then there is a Nisnevich neighbourhood $U$ of $x$ in $X$, a closed subscheme $T\subset U$, of codimension $p-1$,
 containing $Z_U := Z\cap U$, $i:Z_U \rightarrow T$,
 such that for any integer $q\geq0$, the following pointed map is trivial:
$$\Pi_q(U,Z_U)\xrightarrow{\ i_*\ } \Pi_q(U,T).$$
\end{lemma}
\begin{proof}
This is same as the proof for \Cref{eff}, except that we appeal to
 the relative presentation lemma \cite[Th. 1.1]{DHKY}.
 Under the stated hypothesis, we get that $\dim(Z_s)<\dim(X_s)$,
 so $Z_s$ does not contain any irreducible component of $X_s$, then we can apply the relative presentation lemma~\cite{DHKY}. 
 Indeed, $\dim(Z_s)<\dim(X_s)$: to see this note that $\dim(X_s)=d-e$ because $X$ is $S$-smooth, hence $\dim(Z_s) \leq \dim(Z)=d-p$ (by definition of $p$), so $p>e$ implies $d-p<d-e$.

As in the proof of \Cref{eff}, by the presentation lemma, there is an Nisnevich neighbourhood $U$ of $x$ in $X$ and also replacing $S$ by a Nisnevich neighbourhood of $s=\pi(x)$ in $S$, we have excisive pair $f:(U,Z_U) \rightarrow (\AA^1_V,W)$
 for some closed pair $(\AA^1_V,W)$ such that $V$ is a smooth $S$-scheme $V$ and 
  the composite map $W \subset \A^1_V \xrightarrow{\pi} V$ is finite
 (take $f=\varphi|_U$ in the notation of \cite[3.1.1]{CHK}).
 Consider the closed subscheme $F:=\pi(W)_{red} \subset V$ and $T= f^{-1}(\AA^1_F)$.
 Note that there is an inclusion $i:Z_U \rightarrow T$.
 Hence there is a commutative diagram of pointed sets 
$$
\xymatrix@R=10pt@C=20pt{
 \Pi_n(\A^1_V,T) \ar[r] \ar_{f^*}[d]^{\simeq} &  \Pi_n(\A^1_V,\A^1_F) \ar[d] \\
  \Pi_n(U,Z_U) \ar^{i_*}[r] &  \Pi_n(U,T)
}
$$
and excision for $f$ combined with \Cref{keylemma} allows us to conclude.
\end{proof}

As a corollary and using \Cref{prop:truncated-degeneracy}, we immediately get:
\begin{corollary}\label{cor:overbase}
Assume $\Pi_*$ is a cohomotopy theory with supports defined on $\Sm_S$. Assume $\Pi_*$ satisfies the Gabber property.

Then for any essentially smooth local $S$-scheme $X$ of absolute dimension $e$,
 for any $q \geq e$, the augmented naively $e$-cotruncated coniveau unstable complex:
\begin{align*}
\bigoplus_{x \in X^{(e)}} \Pi_{q-e}^x\big(X_{(x)}^h\big) \xrightarrow{d_1^{e,q}} \cdots 
 \rightarrow \bigoplus_{x \in X^{(q-1)}} \Pi_1^x\big(X_{(x)}^h\big)
 \xRightarrow{d_1^{q-1,q}} \bigoplus_{s \in X^{(q)}} \Pi_0^s\big(X_{(s)}^h\big)
 \xrightarrow{\gamma} \Pi_0(X^{\leq q}_+)
\end{align*}
is exact in the sense of \Cref{unstable-cx}.
\end{corollary}

\subsection{Artin-Mazur \'etale homotopy}\label{sec:AM-htp}

\begin{notation}
In this section our goal is to study natural cohomotopy theories with supports associated to the \emph{\'etale homotopy types} in the sense of Artin-Mazur,
 via the formalism developed in the previous section.

The \'etale homotopy type $\Et(X)$ of a scheme $X$ has been introduced by Artin and Mazur in \cite{AM69}
 --- we follow the notation of Friedlander \cite{FriedEt}.
 More recently, the construction has been formulated in terms of $\infty$-categories after To\"en and Vezzosi,
 and we recall the definition from \cite{HoyoisGal} as we find it illuminating.
 Given a scheme $X$, we let $\ihSh(X_\et)$ be the $\infty$-topos of hypercomplete sheaves.\footnote{See \cite[p. 663]{LurieHTT}.
 We use this level of generality for the sake of presentation.
 As we restrict to essentially smooth schemes over a separably closed field, the reader can freely
 use the $\infty$-category of \'etale sheaves.}
 The obvious (unique) morphism of sites $\varnothing \rightarrow X_\et$ induces the canonical geometric morphism of $\infty$-topos
 $p^*:\!S=\iSh(\varnothing) \rightarrow \ihSh(X_\et)$, the constant sheaf functor. As a morphism of $\infty$-topos, it always admits
 a right adjoint, but as it commutes with finite limits, it also admits a pro-left adjoint
 $p_!:\iSh(X_\et) \rightarrow \pro-\iSet$, with values in the $\infty$-category of pro-spaces.
 With this notation, we put: $\Et(X):=p_!p^*(*)=p_!(X)$, where we have abusively denoted by $X$ the constant \'etale sheaf on $X$ with value $*$.\footnote{After To\"en-Vezzosi-Lurie,
 this is also called the \emph{shape} of the $\infty$-topos $\ihSh(X_\et)$.}
 It is proved in \cite[Cor. 5.6]{HoyoisGal} that this definition coincides with the one of Artin-Mazur if $X$ is locally connected
 (e.g. locally noetherian).

Note that $\Et(X)$ behaves like a homology theory in $X$: it is covariant with respect to arbitrary maps of schemes,
 and contravariant with respect to \'etale covers. For this reason, we will consider the associated cohomotopy theory
 with coefficients in an arbitrary pro-space.\footnote{Note also that according to a fundamental theorem of Artin and Mazur (\cite[Th. 11.1]{AM69}), $\Et(X)$ is profinite whenever $X$ is geometrically unibranch. So in this case, one can rather consider $\Et(X)$ as taking values in the $\infty$-category of profinite spaces as in \cite{Q08}. In particular, for our purpose, one could as well restrict to coefficients in a profinite space instead of an arbitrary pro-space.}
\end{notation}
\begin{definition}
We consider a category of schemes $\Sch$ as in \Cref{num:cohtp_support}.
 Let $\mathbb G$ be a pro-space, that we view as a pro-$\infty$-groupoid of coefficients.
 We define the \emph{\'etale cohomotopy theory with supports and coefficients in $\Coef$} by associating to a closed $\Sch$-pair $(X,Z)$
 the $\pi_*$-structure in $\Set$ (see \Cref{df:pi*-structure}):
$$
\Pi_{\ast}^{\et}(X, Z; \Coef):=\pi_*\Map\big(\Et(X/X-Z),\Coef\big)
$$
where $\Et(X/X-Z)$ denotes the homotopy cofiber of the map $\Et(X-Z) \rightarrow \Et(X)$
 and $\Map$ denotes the mapping space in the $\infty$-category $\pro-\iSet$.
\end{definition}
Indeed, it follows as in \Cref{rem:cohtp_support} that the above definition
 does indeed define a cohomotopy theory with supports. Besides, it is Nisnevich-local
 as the \'etale homotopy type satisfies Nisnevich excision. In fact, one can define this theory
 in much simpler term. Let us write $\Coef=(\Coef_i)_{i \in I}$.
 Then, one gets:
$$
\Pi_*^{\et}(X, Z; \Coef) \simeq \pi_*\Big[\lim_{i \in I} \Map_{\ihSh(X_\et)}\big(X/X-Z,\Coef_{i,X}\big)\Big]
$$
where $X/X-Z$ is the homotopy cofiber of the map $(X-Z) \rightarrow X$ computed on the associated represented sheaves
 in $\ihSh(X_\et)$, and $\Coef_{i,X}=p^*\Coef_i$ is the constant \'etale sheaf on $X$ with values in the space
 $\Coef_i$.\footnote{This easily follows from the definition of $\Et(X)$ and the fact $p_!$ is a pro-left adjoint to $p^*$.}

The main result of this section is the following.
\begin{theorem}\label{prop:et-gabber}
Let $k$ be a separably closed field and $\Coef$ be a pro-space.
 Then the cohomotopy theory with supports $\Pi_*(-;\Coef)$ restricted to the category $\Sm_k$ satisfies the Gabber property
 (\Cref{def:gabber}).
\end{theorem}

This follows easily from the following observation which is of independent interest. 
\begin{lemma}\label{lem:et-gabber} Let $k$ be a separably closed field.
 Let $X$ be a smooth affine $k$-scheme and $(X, Z)$ be a closed pair.
Consider the diagram of schemes
$$\xymatrix@C=20pt@R=10pt {
\A^1_X \ar[r]^j \ar[rd]_{\pi} & \#P^1_X \ar[d]^/-2pt/{\-\pi} \ar[d] & X \ar[l]_{s_{\infty}} \ar@{=}[dl] \\
 & X & 
}$$
where $j$ (resp. $s_\infty$, $\pi$) is the obvious open immersion
 (resp. section at $\infty$, canonical projection).
Then the induced diagram
$$
\xymatrix@C=30pt@R=4pt{
\Et (\#A^1_X/\#A^1_X-\#A^1_Z) \ar[rd]_{\pi_*} \ar[rr]^{j_*} & & \Et (\#P^1_X/\#P^1_X-\#P^1_Z) \\
 & \Et (X/X-Z) \ar[ru]_{s_{\infty*}} &   
}
$$
is commutative in the $\infty$-category $\pro-\iSet$.
\end{lemma}
\begin{proof}
Note that one needs only to prove the diagram is commutative in the associated homotopy category,
 as this will produce a homotopy between $j_*$ and $s_{\infty*} \circ \pi_*$.
 By definition of $\Et(X/X-Z)$ as a homotopy cofiber, one is reduced to the case $X=Z$.
 That is, we need to prove that the following diagram commutes in $\Ho \pro-\iSet$:
$$
\xymatrix@C=30pt@R=-4pt{
\Et (\#A^1_X) \ar[rd]_{\pi_*} \ar[rr]^{j_*} & & \Et (\#P^1_X). \\
 & \Et (X) \ar[ru]_{s_{\infty*}} &   
}
$$
Let $p:\PP^1_X \rightarrow \PP^1_k$ be the canonical projection.
 Then we can apply the K\"unneth formula for the \'etale homotopy type proved in \cite[Th. 5.3]{Chough} so that we get that the canonical map:
$$
\Et(\#P^1_X) \xrightarrow{(\-\pi_*,p_*)} \Et(\#P^1_k) \x \Et(X)
$$
is an equivalence.\footnote{Note that under our assumptions, both $\Et(\#P^1_k)$ and $\Et(X)$ are profinite so that \emph{loc. cit.} does apply.}
 So to prove the above diagram is homotopy commutative, one can post-compose by the equivalence $(\-\pi_*,p_*)$.
 This reduces to the case $X=\Spec(k)$.
We prove this separately depending on the characteristic of the base field $k$. 
\paragraph{\bf Case 1: characteristic $k=0$ } In this case we can assume $k\subset \#C$. Then by \cite[Corollary 12.10]{AM69}, we have that $\Et(\#A^1_{k})\simeq \^{\#A^1(\#C)}=*$, hence the diagram commutes, since both the maps are constant.
 \paragraph{\bf Case 2: characteristic $k=p>0$ } In this case $k$ is a separably closed field of positive characteristic.
 This implies that $\Et(\#P^1_{k})\simeq K(\^{\#Z}(1), 2)$ where the profinite group $\^{\#Z}(1)=\mu(k)$ is the profinite group of roots of unity.
 Note that $\mu(k)\simeq \ilim_n \mu_n(k)$.

Using the relation between continuous \'etale cohomology and \'etale homotopy we get:
$$
[\Et(\#P^1_{k}), \Et(\#P^1_{k})] \simeq [\Et(\#P^1_{k}), K(\^{\#Z}(1), 2)]
 \simeq H^2_{{\rm cts}}(\#P^1_{k}, \^{\#Z}(1))=\ilim_n H^2_{\et}((\#P^1_{k}, \mu_n(k))
$$
where the first two groups are Hom-groups taken in the homotopy category of $\pro-\iSet$.
  By \cite[Equation (4.1)]{CHK}, we further have
  \begin{align}
  H^2_{\et}((\#P^1_{k}, \#Z/n)&=\begin{cases}
 0  &  p \mid n \\ 
   H^0_{\et}((\Spec k, \#Z/n(-1)) & p \nmid n.
    \end{cases}
  \end{align}

In the case $p \nmid n$, as noted in \cite[Lemma 4.1.3]{CHK}, the restriction of the maps $j^*, s_{\infty}^*$
 from  $H^2_{\et}((\#P^1_{k}, \#Z/n)$
 on the factor $H^0_{\et}((\Spec k, \#Z/n(-1))$ is 0.
 As a consequence the maps 
 $$j^*, \pi^*\circ s_{\infty}^* : [\Et(\#P^1_{k}), \Et(\#P^1_{k})] \to [\Et(\#A^1_{k}), \Et(\#P^1_{k})] $$ are equal.
 In particular, $j_{\et}=j^*(\id)= \pi^*\circ s_{\infty}^*(\id)=(s_{\infty})_{\et}\circ \pi_{\et}$. Hence the lemma holds.
\end{proof}

\section*{Appendix}

\subsection*{Uniqueness of truncated Cousin complexes}

The goal of this appendix is to give details on a proof of a result of \cite{Hart66},
 showing the uniqueness of Cousin resolutions.
 This is used to show that the result can be extended to the case
 of truncated Cousin resolutions that appeared naturally
 in our unstable context. 
 
 Throughout this section we recall the \Cref{notn:unstable-gersten}. Let $t=\zar, \nis$, and $X$ be a scheme.
The classical framework of $t$-local cohomology with supports can be extended
 to the non-abelian setting (see \cite[\textsection IV]{Hart66} for Zariski sheaves of abelian groups,
 and \cite[\textsection 4.3]{DFJ22} in the Nisnevich case).
 Given a pointed $t$-sheaf $F$ of sets, in $\Sh(X_t)_*$,
 and $Z \subset X$ a closed subscheme, one puts:
 $$\Gamma_Z(X,F)=\{\rho \in F(X) \mid \forall x \notin Z, \rho_x=* \}.$$
 Similarly, $\uG_Z(F)$ is the pointed $t$-sheaf
 $V \mapsto \uG_{Z \times_X V}(F|_V)$.
 If $T \subset Z$, one let $\uG_{Z/T}(F)$ be the cokernel,
 in the category of pointed $t$-sheaves, of the natural pointed map
 $\uG_T(F) \rightarrow \uG_Z(F)$.

First we review the well-known results~\cite[Lemma 2.2, Proposition 2.3]{Hart66}. 
\begin{lemma}(\cite[Lemma 2.2]{Hart66})
Let $X \supset Z \supset Z'$ such that $Z'$ is stable under specialization and such that every $x\in Z-Z'$ is maximal (ordered by the specialization) in $Z$. 

Given a sheaf of abelian groups $F$ on $X$ with supports in $Z$, 
\noindent
\begin{enumerate}
\item
the sheaf $\_{H}^0_{Z/Z'}(F)$ lies on the $Z/Z'$-skeleton of $X$ and  the canonical map $$F\to \_{H}^0_{Z/Z'}(F) $$ is $Z'$-isomorphism, \ie  the kernel and cokernel are supported on $Z'$.

\item given a $Z'$-isomorphism $$F\to G$$ into a sheaf $G$ which lies on the  $Z/Z'$-skeleton, it factors uniquely (unique upto unique isomorphism) through the canonical map $$\xymatrix{
F\ar[rd] \ar[rr]& & G\\ 
& \_{H}^0_{Z/Z'}(F) \ar[ru]^{\simeq}& 
}$$
\end{enumerate} \label{kercoker0}
\end{lemma}

\begin{corollary} \label{kercoker}
Let $F$ be a sheaf of abelian groups on $X$ and $X\supset Z\supset Z'$ (as above). Then the map $$\_{H}^0_{Z/Z'}(F)\to \_{H}^0_{Z/Z'}(\_{H}^0_{Z/Z'}(F)) $$ is $Z'$-isomorphism \ie the kernel and cokernel is supported on $Z'$.
\end{corollary}
\begin{proof}
This follows by noting that for a sheaf of abelian groups $F$ on $X$, the sheaf $\_{H}^0_{Z/Z'}(F)$ is with supports in $Z$.
\end{proof}

\begin{theorem}\label{thm:unique-cousin}(assumptions as in~\cite[Proposition 2.3]{Hart66})
Let $F$ be a sheaf of abelian groups. Then there is a unique (unique upto unique isomorphism) augmented complex 
$$F\to C^{\Dot}$$
 such that 
 \noindent
 \begin{enumerate}

 \item  for each $p\geq 0$, $C^p$ lies on the $Z^p/Z^{p+1}$-skeleton.
 \item for each $p>0$ $C^p/\im(C^{p-1})$ has supports in $Z^{p+1}$.
 \item the map $F\to H^0(C^{\Dot})$ has kernel supported on $Z^1$ and $C^0/\im (F)$ is supported on $Z^1$.
 
 \end{enumerate}

 Furthermore, the complex $C^{\Dot}$ satisfies the conditions 
 \begin{enumerate} [label=\emph{(\roman*)}]
   \item  for each $p\geq 0$, $C^p$ lies on the $Z^p/Z^{p+1}$-skeleton.

 \item for each $p>0$ $H^p(C^{\Dot})$ has supports in $Z^{p+2}$.
 \item the map $F\to H^0(C^{\Dot})$ has kernel supported on $Z^1$ and cokernel is supported on $Z^2$.
 \end{enumerate}
\end{theorem}
\begin{remark}
We will see in the proof below that the complex we construct satisfies the properties (i), (ii) and (iii). 

Conversely, if we construct a complex $C^{\Dot}$ satisfying the properties (i), (ii) and (iii) then it satisfies the properties (1), (2) and (3). 

More precisely, (ii)+ (i) $\Rightarrow$ (2) and (iii)+ [(i) for $p=1$)] implies (3).

\end{remark}

\begin{proof}[Proof of \Cref{thm:unique-cousin}]
We construct such a complex by induction on $p\geq 0$. 

\noindent
\_{\bf Step 1:} ($p=0$) Let $C^0:=\_H^0_{Z/Z^1}(\_H^0_{Z/Z^1}(F))$ with the canonical map $$F\to \_H^0_{Z/Z^1}(F)\to C^0$$
We claim that kernel $\ker (F\to C^0)$ and $C^0/\im F$ is supported on $Z^1.$
Note that we have an exact sequence $$0\to \_\Gamma_{Z^1}F\to \ker(F\to C^0)\to \ker(\_H^0_{Z/Z^1}(F)\to C^0)$$  
Hence by above \Cref{kercoker} since $\ker(\_H^0_{Z/Z^1}(F)\to C^0)$ is supported on $Z^1$ and by definition  $\_\Gamma_{Z^1}F$ is supported on $Z^1$, we get that $\ker(F\to C^0)=\ker (F\to H^0(C^{\Dot}))$ is supported on $Z^1$. 
Similarly, there is an exact sequence
$$\dfrac{\_H^0_{Z/Z^1}(F)}{\im (F)} \to \dfrac{C^0}{\im(F)} \to \dfrac{C^0}{\_H^0_{Z/Z^1}(F)}\to 0$$ 
By~\Cref{kercoker}, we have that $\dfrac{C^0}{\_H^0_{Z/Z^1}(F)}$ is supported on $Z^1$.  

To see that $\dfrac{\_H^0_{Z/Z^1}(F)}{\im (F)}$ is supported on $Z^1$, note that there is an exact sequence~\cite[Page 219, Motif B (sheafified version)]{Hart66} 
$$0\to \_\Gamma_{Z^1}F \to F\to \_H^0_{Z/Z^1}(F)\to \_H^1_{Z^1}(F) $$
This implies that the cokernel $\dfrac{\_H^0_{Z/Z^1}(F)}{\im (F)}$ is a sub-object of  $\_H^1_{Z^1}(F) $, which is supported on $Z^1$. Hence the claim. 

Thus we have a map $$F\to C^0$$ such that the 
kernel $\ker (F\to C^0)$ and the cokernel $C^0/\im F$ are supported on $Z^1.$\\
\noindent
\_{\bf Step 2:} (Induction) Assume that the complex $C^{\Dot}$ is determined up to degrees $\leq p$ satisfying the desired properties. 
So we have $$F=C^{-1}\to C^0\to C^1\to\cdots \to C^{p-1}\to C^p.$$

Since by the inductive assumption $C^p/\im C^{p-1}$ has support in $Z^{p+1}$, we can apply the above~\Cref{kercoker0}  to $C^p/\im C^{p-1}$ with $Z=Z^{p+1}$ and $Z'=Z^{p+2}$. Let us define $$C^{p+1}:=\_H^0_{Z^{p+1}/Z^{p+2}}\bigg(\dfrac{C^p}{\im C^{p-1}}\bigg).$$  
Then the canonical map  
$$\dfrac{C^p}{\im C^{p-1}}\to C^{p+1}$$ has \begin{align*}
\ker &=\begin{cases}
H^0(C)/F \ \ for \ p=0 \\
 H^p(C^{\Dot}) \ \ for \ p>0
\end{cases}
\end{align*}
 and its cokernel $\dfrac{C^{p+1}}{\im C^{p}}$ both have supports in $Z^{p+2}$. 
Also $C^{p+1}$ lies on the skeleton $Z^{p+1}/Z^{p+2}$. 

Thus by induction we have constructed the augmented complex $$F\to C^{\Dot}$$ satisfying the desired properties. \\
\noindent
\_{\bf Step 3:} (Uniqueness) Assume that there is an augmented complex $$F\to C'^{\Dot}$$ satisfying the listed properties.
We will again show this by induction on $p$.

Assume that the uniqueness is true for complexes upto degrees $\leq p$. We need to show that it induces uniqueness for complexes upto degrees $\leq p+1$.

By inductive assumption, there is a unique isomorphism of augmented complexes $\tau_{\leq p}C^{\Dot}\to \tau_{\leq p}C'^{\Dot}$. 
$$\xymatrix{
F\ar[r] \ar@{=}[d] \ar[r] & \cdots \ar[r]  & C^p \ar[d]^{\simeq} \ar[r]  & C^{p+1} \\
F\ar[r] & \cdots \ar[r] & C'^p \ar[r] & C'^{p+1}
}$$
Hence there is a unique isomorphism $$C^p/\im C^{p-1} \to C'^p/\im C'^{p-1}.$$
Composing this we have a map $$C^p/\im C^{p-1} \to C'^p/\im C'^{p-1} \to C'^{p+1}$$ where by assumption on $C'^{\Dot}$, $C'^{p+1}$ lies on the skeleton $Z^{p+1}/Z^{p+2}$ , hence by~\Cref{kercoker0}, there is a unique $Z^{p+2}$-isomorphism 
$$C^{p+1}:=\_H^0_{Z^{p+1}/Z^{p+2}}\bigg(\dfrac{C^p}{\im C^{p-1}}\bigg)\to  C'^{p+1}.$$ It clearly commutes with previous differentials. 

We note that the $Z^{p+2}$-isomorphism $$C^{p+1}\to C'^{p+1}$$ is actually an isomorphism as 
$$\xymatrix{
C^{p+1}\ar[r]\ar[d]^{\simeq} \ar[d]& C'^{p+1} \ar[d]^{\simeq} \\
\_H^0_{Z^{p+1}/Z^{p+2}}(C^{p+1}) \ar[r]^{\simeq} & \_H^0_{Z^{p+1}/Z^{p+2}}(C'^{p+1})
}$$\end{proof}

\begin{definition} The Cousin complex $\Cz^*(F)$ of sheaf of abelian groups $F$ is the complex satisfying the properties as in~\Cref{thm:unique-cousin}.
 
\end{definition}

Let us recall from \Cref{num:cohomotopical_notation}, a $q$-truncated unstable augmented homotopical complex $$F\to C^0\to C^1\to \cdots C^{q-1}\Rightarrow C^q$$
is such that $C^p$ is an abelian group object for $0\leq p\leq q-2$, $C^{q-1}$ is a group object and $C^q$ is a pointed set with the group $C^{q-1}$ acting on $C^q$. 
\begin{theorem}\label{thm:unstable-unique-cousin}
Let $F$ be a sheaf of abelian groups on $X_t$. Let $q>1$. Let $F\to C^{\Dot}$ be a $q$-truncated unstable augmented homotopical complex $$F\to C^0\to C^1\to \cdots C^{q-1}\Rightarrow C^q$$ such that 
 \noindent
 \begin{enumerate}
 \item  for each $0\leq p\leq q$, $C^p$ is supported in $X^{(p)}$ (in the sense of \Cref{df:Cousin-cpx}).
 \item for each $0<p<q $ $C^p/\im(C^{p-1})$ has supports in $X^{\geq p+1}$ \ie isomorphic to $ \uG_{X^{\geq p+1}}\bigg(C^p/\im(C^{p-1})\bigg)$.
 \item the map $F\to H^0(C^{\Dot})$ has kernel supported on $X^{\geq 1}$ and $C^0/\im (F)$ is supported on $X^{\geq 1}$.
 \end{enumerate}
 Then there is a unique isomorphism $$\tau^{\leq q}\Cz^*(F)\to C^*$$ of augmented complexes.  
\end{theorem}
\begin{proof}
The proof follows along the same lines as in the proof of \Cref{thm:unique-cousin} to give a map 
$$\tau^{\leq q}\Cz^*(F)\to C^*$$ upto degrees $q-1$ \ie we have
$$\xymatrix{
F\ar[r] \ar@{=}[d] & \cdots \ar[r] &\Cz^{q-1}(F) \ar[r] \ar[d]^{\simeq} & \Cz^{q}(F) \\
F\ar[r] & \cdots \ar[r] & C^{q-1} \ar[r] & C^q
}$$
Hence there is a unique isomorphism $$\Cz^{q-1}/\im \Cz^{q-2} \to C^{q-1}/\im C^{q-2}$$ sheaves of pointed sets.
Composing this, we have a map $$\Cz^{q-1}/\im \Cz^{q-2} \to C^{q-1}/\im C^{q-2} \to C^{q}$$ where by assumption on $C^{\Dot}$, $C^{q}$ is supported in $X^{(q)}$ , hence by~\Cref{lem:unstable-skeleton}, there is a unique $X^{\geq q+1}$-isomorphism 
$$\Cz^{q}=\uG_{X^{\geq q}/X^{\geq q+1}}\bigg(\dfrac{\Cz^{q-1}}{\im \Cz^{q-2}}\bigg)\to  C^{q}.$$ It clearly commutes with previous differentials. 
We note that the $X^{\geq q+1}$-isomorphism $$\Cz^{q}\to C^{q}$$ is actually an isomorphism as 
$$\xymatrix{
\Cz^{q}\ar[r] & C^{q}  \\
\uG_{X^{\geq q}}(\Cz^{q}) \ar[u]^{\simeq} \ar[r]^{\simeq} & \uG_{X^{\geq q}}(C^{q})\ar[u]^{\simeq}
}$$
\end{proof}
\begin{lemma}\label{lem:unstable-skeleton}
Let $X \supset Z \supset Z'$ such that $Z'$ is stable under specialization and such that every $x\in Z-Z'$ is maximal (ordered by the specialization) in $Z$. 

Given a sheaf of pointed sets $F$ on $X$ with supports in $Z$, 
\noindent
\begin{myenum}
\item
the sheaf $\_{\Gamma}_{Z/Z'}(F)$ lies on the $Z/Z'$-skeleton of $X$ and  the canonical map $$F\to \_{\Gamma}_{Z/Z'}(F) $$ is $Z'$-isomorphism, \ie isomorphism when restricted to the open complement $Z-Z'$.

\item given a $Z'$-isomorphism $$F\to G$$ into a sheaf $G$ which lies on the  $Z/Z'$-skeleton, it factors uniquely (unique upto unique isomorphism) through the canonical map $$\xymatrix{
F\ar[rd] \ar[rr]& & G\\ 
&  \_{\Gamma}_{Z/Z'}(F) \ar[ru]^{\simeq}& 
}$$
\end{myenum}
\end{lemma}


\bibliographystyle{amsalpha}
\bibliography{UnstConiveau}

\end{document}